\documentclass[3p]{elsarticle}
\usepackage[utf8]{inputenc}

\usepackage{amsmath, amsthm, amsfonts, amssymb}
\usepackage{bm}
\usepackage{todonotes}
\usepackage{xcolor}
\usepackage{subcaption}
\usepackage{aligned-overset}
\usepackage{hyperref}

\newcommand{\energy}{\mathcal{E}}
\newcommand{\pf}{\varphi}

\newcommand{\strain}[1][u]{{\bm \varepsilon}(\bm #1)}

\newcommand{\Div}{\nabla\cdot}
\newcommand{\Span}{\mathrm{span}}

\newcommand{\vecSpL}{\bm{L^2(\Omega)}}
\newcommand{\vecSpH}{\bm{H^1(\Omega)}}
\newcommand{\vecSpHTr}{\bm{H_0^1(\Omega)}}
\newcommand{\vecSpHM}{\bm{H^{-1}(\Omega)}}
\newcommand{\tensSpL}{\bm{\mathcal{L}^2(\Omega)}}
\newcommand{\regularization}{\eta}

\newtheorem{lemma}{Lemma}
\newtheorem{prop}{Proposition}
\newtheorem{theorem}{Theorem}
\newtheorem{corollary}{Corollary}
\theoremstyle{definition}
\newtheorem{definition}{Definition}
\theoremstyle{remark}
\newtheorem{remark}{Remark}

\usepackage{tikz}
\usepackage{pgfplots,pgfplotstable}
\pgfplotsset{compat=newest}
\definecolor{blau0} {RGB}{ 131 198 216} %56 112 163}
\definecolor{grau}  {RGB}{  0  84 159}
\definecolor{rot}   {RGB}{204   7  30}
\definecolor{blau2} {RGB}{  0  61 128}
\definecolor{grun2} {RGB}{  0 85   0}
\definecolor{rot2}  {RGB}{120   7  30}
\definecolor{rot3}{RGB}{255 190 203}
\definecolor{gelb}  {RGB}{70 70 70}
\definecolor{blau}{RGB}{0 144 188}
\definecolor{newblue1}{RGB}{0 144 188}
\definecolor{newblue2}{RGB}{197 216 227}
\definecolor{newgreen1}{RGB}{0 144 118}
\definecolor{grun}{RGB}{255 137 0}
\definecolor{newgreen2}{RGB}{197 222 215}
\definecolor{neworange1}{RGB}{255 137 0}
\definecolor{neworange2}{RGB}{255 205 105}

\newcommand{\revision}[1]{\textcolor{black}{#1}}

\newcommand\mathbox[1]{\mbox{$#1$}}

\makeatletter
\def\ps@pprintTitle{
 \let\@oddhead\@empty
 \let\@evenhead\@empty
 \def\@oddfoot{}
 \let\@evenfoot\@oddfoot}
\makeatother

\title{Well-posedness analysis of the Cahn-Hilliard-Biot model}

\author[inst1]{Cedric Riethmüller\corref{cor1}}
\ead{cedric.riethmueller@ians.uni-stuttgart.de}
\author[inst2]{Erlend Storvik}
\ead{Erlend.Storvik@hvl.no}
\author[inst3]{Jakub Wiktor Both}
\ead{Jakub.Both@uib.no}
\author[inst3]{Florin Adrian Radu}
\ead{Florin.Radu@uib.no}

\cortext[cor1]{Corresponding author}

\affiliation[inst1]{organization={
Institute of Applied Analysis and Numerical Simulation,
University of Stuttgart},%Department and Organization
            country={Germany}}

\affiliation[inst2]{organization={Department of Computer science, Electrical engineering and Mathematical sciences, Western Norway University of Applied Sciences},%Department and Organization
            country={Norway}}

\affiliation[inst3]{organization={Center for Modeling of Coupled Subsurface Dynamics, Department of Mathematics, University
of Bergen},
country={Norway}}

\begin{document}

\begin{abstract}
We investigate the well-posedness of the recently proposed Cahn-Hilliard-Biot model. The model is a three-way coupled PDE of elliptic-parabolic nature, with several nonlinearities and the fourth order term known to the Cahn-Hilliard system. We show existence of weak solutions to the variational form of the equations and uniqueness under certain conditions of the material parameters and secondary consolidation, adding regularizing effects.  Existence is shown by discretizing in space and applying ODE-theory (the Peano-Cauchy theorem) to prove existence of the discrete system, followed by compactness arguments to retain solutions of the continuous system. In addition, the continuous dependence of solutions on the data is established, in particular implying uniqueness. Both results build strongly on the inherent gradient flow structure of the model.
\end{abstract}

% Research highlights - 
% \begin{highlights}
% \item Existence of weak solutions to the Cahn-Hilliard-Biot system is established in mixed formulation.
% \item A viscoelastic extension of Kelvin-Voigt type is considered adding regularization.
% \item The inherent generalized gradient flow structure of the system is exploited to prove continuous dependence for state-independent material parameters.
% \end{highlights}

\begin{keyword}
Cahn-Hilliard \sep poroelasticity \sep well-posedness \sep existence analysis \sep uniqueness \sep gradient flows
% %% MSC codes here, in the form: \MSC code \sep code
% %% or \MSC[2008] code \sep code (2000 is the default)
\end{keyword}

\maketitle

\section{Introduction}

% General intro
Fluid flow in deformable media subject to phase changes appears in various natural systems. Relevant applications include tumor growth, wood growth and biogrout, which all center around dynamically changing material properties. In addition, these media can be viewed as porous materials limiting possible fluid dynamics to laminar flow conditions. Finally, an immediate coupling between the flow and structural properties of these materials requires multi-physics models to describe such systems covering the interaction of the different subproblems. Due to the tightly coupled and possibly nonlinear characteristics of such models, their mathematical analysis as well as development of robust numerical methods for their solution is often challenging and requires a problem-tailored effort.  

% Our model of interest
In this work, we consider the Cahn-Hilliard-Biot model that was recently proposed in \cite{Storvik2022}. This model is an extension of the quasi-static, linearized Biot equations for modeling single-phase flow in linearly elastic porous media. The solid material, besides being exposed to deformation due to typical elastic and hydraulic effects, is assumed to consist of two phases subject to phase change, feeding back on the elastic deformation itself as well. The evolution of phase changes is driven by interface tension, here approximated through a diffuse interface model closely related to the Cahn-Hilliard equation. The final model nonlinearly couples both the Biot equation and the Cahn-Hilliard model, and it allows for material parameters to continuously change with respect to the solid phase, resulting in a three-way coupled, nonlinear system of PDEs.

% Research gap

The goal of this work is to establish well-posedness results for the Cahn-Hilliard-Biot model. In particular, we investigate the existence of solutions in a proper weak sense, as well as the sensitivity of solutions with respect to data, implying uniqueness. The main difficulty lies in the both coupled, nonlinear, and (partly) non-convex character of the problem.

There are several well-posedness results for similar models. For the Cahn-Hilliard equation, we refer the reader to the classical works \cite{Elliott1986, Elliott1996}. A similar analysis, to the one that is presented here, is conducted for the so-called Cahn-Larch\'e equations, which couple the Cahn-Hilliard equation with linear elasticity equations in both \cite{Garcke2021} and \cite{FritzSubdiffusive2021}, the latter with a fractional time derivative. Other well-posedness results cover a related extended Cahn-Hilliard model coupling a 3D tumor growth model to a 1D network modeling angiogenesis~\cite{FritzMultispecies2021}, the Cahn-Hilliard equation coupled to nutrient transport~\cite{Colli2015}, a Cahn-Hilliard--Navier-Stokes model~\cite{Colli2012}, the Cahn-Hilliard equation coupled with Darcy's law~\cite{Garcke2016, Garcke2018}, and a Cahn-Hilliard-Brinkman model~\cite{Ebenbeck2019}.

On the other hand, the linear Biot equations have been extensively studied in the last decades. The existence of strong solutions has been established in the seminal work~\cite{Auriault1977}, providing the first well-posedness results for the linear Biot equations. Results including existence of weak solutions~\cite{Zenisek1984}, well-posedness from the perspectives of semi-group theory~\cite{Showalter2000} as well as generalized gradient flow theory~\cite{Both2019} have followed. Recently, the study of nonlinear extensions of the linear Biot equations, including nonlinear constitutive laws and/or additional physics, have seen increased interest, with this work falling into the same category and bridging the analysis of Biot equations and Cahn-Hilliard equations.

% Methodology
This work draws inspiration from many of the aforementioned works. The paper is structured in a two-fold manner where we first focus on proving existence of a solution to the system of equations and then move the attention to continuous dependence of the solution with respect to initial data. In order to prove the existence of a solution, a standard technique is employed, see e.g.~\cite{Garcke2021, Garcke2016, Lowengrub2013, Frigeri2015, FritzDarcy2019, FritzECM2019}. First, the weak system of equations are discretized in space using a Galerkin method, transforming the problem to a system of nonlinear ODEs. This system of ODEs are then showed to satisfy sufficient continuity conditions so that the Peano-Cauchy theorem can be applied to guarantee the existence of a, local-in-time, solution to the system of ODEs. Furthermore, bounds are provided for the discrete-in-space solutions which eventually imply global-in-time solutions and, using compactness arguments, converging sequences in appropriate function spaces. Finally, the limiting solution is showed to solve the system of equations in a proper weak sense. In order to discuss several nonlinearities in the model, connected to the non-convex character of the problem, sufficiently strong convergence is required. To ensure this, the Cahn-Hilliard-Biot model is enhanced to model light secondary consolidation. This acts as regularization, previously applied in different works on the analysis of nonlinear poroelasticity systems, e.g.,~\cite{Showalter2000,Bociu2016,Both2021,Bociu2023}; on the other hand, Kelvin-Voigt type viscoporoelasticity is highly relevant in biomedical applications~\cite{Greiner2021,Bociu2016,Bociu2023}. Furthermore, in part constant material parameters are required, while several parameters are still allowed to remain nonlinear, retaining the non-convex character of the free energy associated to the Cahn-Hilliard-Biot model.

The coupled Cahn-Hilliard-Biot model has an inherent, generalized gradient flow structure with respect to a natural combination of the Helmholtz energies associated with poroelasticity and the Cahn-Hilliard equation~\cite{Storvik2022}. In this paper, considering phase-independent material parameters, the gradient flow structure and the related energy minimization principles are exploited to derive a natural dissipation-energy identity and eventually show continuous dependence of solutions to the equation with respect to data. 

Additionally, we mention that the generalized gradient flow structure can be utilized to make robust solution strategies for the system of equations, as separately demonstrated for Cahn-Larch\'e equations~\cite{Storvik2023} and Biot equations~\cite{Both2019} - extensions to the model equations of interest are left for future research.

% \es{Very recently a similar study was presented independently in \cite{Fritz2023}. In that paper, the well-posedness of the Cahn-Hilliard-Biot model was discussed for the same system, but in a hybrid formulation for the flow subproblem, using a pressure formulation with a weak relation between the pressure and volumetric fluid content. In contrast, we utilize a standard mixed formulation of the flow subproblem in this paper, and introduce a secondary consolidation term, for regularizing effect, in order to prove existence of a weak solution. In this work, we also explicitly utilize the gradient flow structure of the problem to provide continuous dependence of the solution with respect to data. Moreover, the work \cite{Fritz2023} also provide numerical experiments, where this work solely is concerned with a thorough proof of well-posedness.}

\revision{Very recently, two similar studies were presented independently in \cite{Fritz2023, Garcke2024}. The focus of the papers lies in the well-posedness of the Cahn-Hilliard-Biot model as well. Nevertheless, the model formulations in our and the aforementioned works differ. While~\cite{Fritz2023, Garcke2024} employ a hybrid formulation for the flow subproblem, using a pressure formulation with a weak relation between the pressure and volumetric fluid content, we exploit a standard mixed formulation of the flow subproblem, consistent with conservation properties of the model. Additionally, our analysis uses slightly different techniques and is strongly based on the highlighted gradient flow structure of the problem, in particular to provide continuous dependence of the solution with respect to data, formulated in terms of stronger regularity, similar to what is achieved in \cite{Garcke2024}. Moreover, due to the different formulation, and a regularized flow subproblem, the authors of \cite{Garcke2024} manage to provide well-posedness to the system with compressibility coefficient $M$ and Biot-Willis coefficient $\alpha$ depending on the phase-field.}%Finally, our analysis requires several assumptions on material laws and added secondary consolidation to deduce strong convergence results and with this, the existence of weak solutions.

The remainder of the paper is structured as follows. In Section~\ref{sec:model}, the Cahn-Hilliard-Biot model is recalled, extended to include viscoporoelastic effects.  Additionally, a weak formulation is defined. In Section~\ref{sec:existence}, the existence of weak solutions is established for the nonlinear model, yet under the assumption of added secondary consolidation. The continuous dependence is discussed in Section~\ref{sec:continuous-dependence} for a simpler model with constant material parameters ensuring a convex Helmholtz free energy. Section~\ref{sec:conclusion} concludes the paper with remarks on future research directions and remaining open problems.

\section{The Cahn-Hilliard-Biot model and its variational formulation\label{sec:model}}
The Cahn-Hilliard-Biot system, as proposed in \cite{Storvik2022}, is a three-way coupled system of equations describing the coupled flow and deformation of a porous medium subject to phase changes. Let the medium be modeled by a bounded domain $\Omega\subset \mathbb{R}^d$, $d \in \{2,3\}$, with Lipschitz boundary~\cite{Alt2016}, and $[0,T]\subset \mathbb{R}$ the time interval with final time $0<T<\infty$. As primary variables for the system we consider the phase-field $\pf$, taking value $\pf = 1$ for one pure phase, and $\pf= -1$ for the other, together with the chemical potential $\mu$, the infinitesimal deformation $\bm u$, the volumetric fluid content $\theta$, and the fluid flux $\bm q$. In addition to~\cite{Storvik2022}, we extend the model with potential Kelvin-Voigt type viscoelastic effects, commonly important in biomedical applications~\cite{Greiner2021,Bociu2016,Bociu2023}.

We remark that the system proposed in \cite{Storvik2022} employs a pressure formulation for the fluid. However, the system has an inherent thermodynamical structure, as a result of it being a generalized gradient flow, and the equivalent volumetric fluid content formulation is the natural choice in this particular structure. For this reason, the Helmholtz energies related to the system are presented first, then balance laws for the phase-field, the linear momentum and volumetric fluid content are described. The system is closed by inferring natural thermodynamical relations between the rates of change of the free energies and the corresponding potentials, and constitutive equations. In contrast to the presentation in~\cite{Storvik2022}, we include secondary consolidation, adding a regularizing, viscoelastic effect, as considered in other works on well-posedness~\cite{Showalter2000,Bociu2016,Both2021,Bociu2023}.

\subsection{The free energy of the system} 
The free energy of the system as proposed in \cite{Storvik2022} is composed of an additive combination of the elastic energy $\energy_\mathrm{e}(\pf, \strain)$, the fluid energy $\energy_{\mathrm{f}}(\pf, \strain, \theta)$, and the regularized interface energy $\mathcal{E}_\mathrm{i}(\pf)$
\begin{equation} \label{eq:totenergy}
    \mathcal{E}(\varphi,\strain, \theta) = \mathcal{E}_\mathrm{e}(\varphi, \strain) +\mathcal{E}_\mathrm{f}(\varphi,\strain, \theta)+ \mathcal{E}_\mathrm{i}(\varphi).
\end{equation}

Here, the elastic energy takes a form that is commonly used in models coupling Cahn-Hilliard with elasticity \cite{Garcke2021}, and in Cahn-Larch\'e systems
\begin{equation}
    \mathcal{E}_\mathrm{e}(\pf,\strain) = \int_\Omega \frac{1}{2} \big({\bm \varepsilon}(\bm u) - \mathcal{T}(\varphi)\big)\!:\!\mathbb{C}(\varphi)\big({\bm \varepsilon}(\bm u) - \mathcal{T}(\varphi)\big) \;dx,
\end{equation}
where ${\bm \varepsilon}(\bm u) = \frac{1}{2}\left(\nabla \bm u + \nabla\bm u^\top\right)$ is the symmetric linearized strain tensor, $\mathcal{T}(\pf)$ denotes the eigenstrain which corresponds to the state of the strain tensor if the material was uniform and unstressed \cite{Fratzl1999}, and $\mathbb{C}(\varphi)$ denotes the potentially anisotropic and heterogeneous elastic stiffness tensor. We note that the special choice $\mathcal{T}(\pf) = \xi(\pf - \bar{\pf})\bm I$ accounts for swelling effects \cite{Areias2016} with swelling parameter $\xi$, (initial) reference value $\bar{\pf}$ and identity tensor $\bm I$, and follows Vegard's law \cite{Garcke2021}.

The fluid energy takes the form 
\begin{equation}
    \mathcal{E}_\mathrm{f}(\varphi,\strain, \theta) = \int_\Omega \frac{M(\varphi)}{2}\left(\theta - \alpha(\varphi)\nabla\cdot \bm u\right)^2\; dx
\end{equation}
where $\alpha(\pf)$ is the Biot-Willis coupling coefficient and $M(\pf)$ is the compressibility coefficient. Both are dependent on the solid phase $\pf$.

Finally, the regularized interface energy is given by
\begin{equation}
    \mathcal{E}_\mathrm{i}(\varphi) := \gamma\int_\Omega \frac{1}{\ell}\Psi(\varphi) + \frac{\ell}{2}|\nabla\varphi|^2\;dx,
\end{equation}
where $\gamma > 0$ denotes the interfacial tension, $\ell > 0$ is a small parameter associated with the width of the regularization layer, and $\Psi(\varphi)$ is a double-well potential penalizing deviations from the pure phases. Typically it is given as
\begin{equation}
    \Psi(\varphi) := \left(1-\varphi^2  \right)^2,
\end{equation}
having minimal values for the two pure phases $\pf = 1$ and $\pf = -1$, however, there are other \revision{more physically justified} examples in the literature as well, including the logarithmic one proposed in \cite{ElliottLogarithm1992} \revision{that becomes singular in the limit $|\pf| \rightarrow 1$}.

\subsection{Balance laws}
For the coupled system of equations the following balance laws are assumed to hold. 
The balance of linear momentum, neglecting inertial effects, balances the Cauchy stress and the externally applied body forces $\bm f$
\begin{equation}
    -\nabla \cdot \bm \sigma = \bm f.
\end{equation}
Furthermore, balance of volumetric fluid content with negligible density gradients are assumed
\begin{equation}
    \partial_t \theta + \nabla \cdot \bm q = S_f,
\end{equation}
where the volumetric fluid content $\theta$ changes with respect to the fluid flux $\bm q$ and sources $S_f$. 
The phase-field is assumed to change with respect to a flux $\bm J$ and reactions $R$
\begin{equation}
\partial_t\varphi + \nabla \cdot {\bm J} = R.
\end{equation}

\subsection{Thermodynamical relations}
Following thermodynamical principles, dual quantities can be derived by considering variational derivatives of the system's free energy $\energy$, introducing a generalized potential in form of stress, pressure and interfacial potential. 
% the elastic stress $\bm \sigma_e = \delta_{\bm \varepsilon} \energy$, the pore pressure $p = \delta_\theta \energy$, and the interfacial potential $\mu =\delta_\varphi \energy$.  To be precise, the Cauchy stress is defined as 
In this work, we consider media with an effective Kelvin-Voigt behavior. Thus, we decompose the Cauchy stress into an elastic and a viscoelastic contribution $\bm \sigma = \bm \sigma_e + \bm \sigma_v$, with merely the elastic contribution being derived as dual variable
$$\bm \sigma_e := \delta_{\bm \varepsilon}\energy(\pf, \strain, \theta) = \delta_{\bm \varepsilon}\energy_\mathrm{e}(\pf,\strain) + \delta_{\bm \varepsilon}\energy_\mathrm{f}(\pf,\strain, \theta) ,$$
where 
\begin{align*}
\delta_{\bm \varepsilon}\energy_\mathrm{e}(\pf,\strain) &= \mathbb{C}(\pf)\big(\bm \varepsilon(\bm u) - \mathcal{T}(\pf) \big) \\
\delta_{\bm \varepsilon}\energy_\mathrm{f}(\pf,\strain, \theta) &= - M(\pf)\alpha(\pf)\big(\theta - \alpha(\pf)\nabla\cdot\bm u\big)\bm I
\end{align*}
The viscous part is restricted to depend on the volumetric strain with a state-independent scaling $\regularization \geq 0$
\begin{align*}
    \bm \sigma_v = \regularization \partial_t (\nabla\cdot\bm u) \bm I
\end{align*}
as commonly employed in the analysis of poroelasticity~\cite{Showalter2000,Bociu2016,Both2021,Bociu2023}. General laws $\bm{\sigma}_v = \mathbb{C}_v \partial_t \bm{\varepsilon}(\bm{u})$ can be discussed identically.

The pore pressure is defined dual to the volumetric fluid content 
$$p:= \delta_\theta\energy(\pf, \strain, \theta) = M(\pf)\big(\theta - \alpha(\pf)\nabla\cdot\bm u\big)$$
and the interfacial potential dual to the phase-field
$$\mu := \delta_{\pf}\energy(\pf, \strain, \theta) = \delta_\pf\mathcal{E}_\mathrm{i}(\pf) + \delta_\pf\energy_\mathrm{e}(\pf,\strain)+ \delta_\pf\energy_\mathrm{f}(\pf,\strain,\theta),$$
where 
\begin{align*}
\delta_\pf\mathcal{E}_\mathrm{i}(\pf)& = \frac{\gamma}{\ell} \Psi'(\pf) - \gamma\ell\Delta\pf, \\
\delta_\pf\energy_\mathrm{e}(\pf, \strain)& = \frac{1}{2} \big({\bm \varepsilon}(\bm u) - \mathcal{T}(\varphi)\big)\!:\!\mathbb{C}'(\varphi)\big({\bm \varepsilon}(\bm u) - \mathcal{T}(\varphi)\big)  - \mathcal{T}'(\varphi)\!:\!\mathbb{C}(\varphi)\big({\bm \varepsilon}(\bm u) - \mathcal{T}(\varphi)\big), \\
\delta_\pf\energy_\mathrm{f}(\pf,\strain,\theta)& = \frac{M'(\varphi)}{2}\left(\theta - \alpha(\varphi)\nabla\cdot \bm u\right)^2 - \alpha'(\pf)M(\varphi)\big(\theta - \alpha(\varphi)\nabla\cdot \bm u\big)\nabla\cdot\bm u.
\end{align*}
Concluding, we can identify the extension of a conventional effective stress decomposition for the Cauchy stress~\cite{Coussy2004} in the context of poro-viscoelasticity to Cahn-Hillard modeling
\begin{align*}
\bm \sigma = \regularization \partial_t \mathrm{tr}\big(\bm \varepsilon(\bm u)\big) \bm I + \mathbb{C}(\pf)\big(\bm \varepsilon(\bm u) - \mathcal{T}(\pf) \big) - \alpha(\pf) p \bm I.
\end{align*}

\subsection{Constitutive equations}
In order to close the model we impose constitutive relations in the form of 
 Darcy's law for the fluid flux
\begin{equation}
 % \kappa(\pf)^{-1}\bm q = -\nabla p = -\nabla\delta_\theta\energy(\pf, \strain,\theta),
 \bm q = - \kappa(\pf) \nabla p,
\end{equation}
and Fick's law for the phase-field flux
\begin{equation}
 \bm J = -m(\pf)\nabla\mu,
\end{equation}
where the the permeability $\kappa(\pf)$ and the mobility $m(\pf)$ are allowed to depend on $\pf$.

\subsection{System of equations}
In total, we get the problem: Find $(\pf, \mu, \bm u, \theta, \bm q)$ for $(x,t)\in \Omega \times [0,T]$ such that
\begin{subequations}
\begin{align}
\partial_t \varphi - \nabla \cdot (m(\pf) \nabla \mu) &= R \label{eq:ch1theta}\\
\mu - \delta_\varphi\mathcal{E}_\mathrm{i}(\varphi) - \delta_\varphi\mathcal{E}_\mathrm{e}(\varphi, \strain) - \delta_\varphi \mathcal{E}_\mathrm{f}(\varphi, \strain, \theta) &= 0 \label{eq:ch2theta}\\
-\nabla \cdot\Big(\regularization \partial_t (\nabla \cdot \bm u) \bm I \Big) -\nabla \cdot\Big(\delta_{\bm \varepsilon}\energy_\mathrm{e}(\pf,\strain) + \delta_{\bm \varepsilon}\energy_\mathrm{f}(\pf,\strain,\theta) \Big) &= \bm f \label{eq:elasticitytheta}\\ 
\partial_t\theta + \nabla\cdot \bm q &= S_f \label{eq:flowtheta}\\
\kappa(\pf)^{-1} \bm q +
\nabla \delta_\theta\energy_\mathrm{f}(\pf,\strain,\theta) &= 0 \label{eq:darcyflowtheta}
\end{align}
\label{eq:model}%
\end{subequations}
accompanied with the initial and boundary conditions
\begin{subequations}
\begin{align}
 \pf = \pf_0 \quad &\mathrm{on} \quad \Omega\times\{0\} \label{eq:initpf}\\
 \theta = \theta_0 \quad &\mathrm{on} \quad \Omega\times\{0\}\label{eq:initmu}\\
 \nabla \cdot \bm{u} = \varepsilon_{v,0}  \quad &\mathrm{on} \quad \Omega\times\{0\}\label{eq:initdivu}\\
 \nabla \pf\cdot\bm n = 0 \quad&\mathrm{on} \quad \partial\Omega\times[0,T]\label{eq:initgradpf}\\
 \nabla \mu\cdot\bm n = 0 \quad&\mathrm{on} \quad \partial\Omega\times[0,T]\label{eq:initgradmu}\\
 \bm u = 0 \quad &\mathrm{on} \quad \partial\Omega\times[0,T]\label{eq:initu}\\
 \bm q\cdot\bm n = 0 \quad &\mathrm{on} \quad \partial\Omega\times[0,T]\label{eq:initq},
\end{align}
\label{eq:init}%
\end{subequations}
where $\bm{n}$ denotes the outward unit normal on $\partial\Omega$.

\begin{remark}[Compatible initial displacement]\label{remark:initial-displacement}
In view of secondary consolidation being mainly introduced for technical reasons in the existence analysis, we restrict the initial conditions for the volumetric strain $\varepsilon_{v,0}$ to natural compatibility conditions, also relevant without active regularization, i.e., $\regularization=0$. For given $\theta_0$ and $\pf_0$, we define the corresponding natural initial displacement $\bm u_0$ satisfying the compatibility condition
\begin{align*}
    \bm{u}_0 := \underset{\bm{u} \ \mathrm{s.t.}\ \bm u|_{\partial \Omega} = \bm{0}}{\mathrm{arg\,min}}\, \mathcal{E}(\pf_0, \strain[u], \theta_0).
\end{align*}
omitting for now the notion of function spaces. Then, we assume $\varepsilon_{v,0} = \nabla \cdot \bm{u}_0$.
\end{remark}

\subsection{Functional analysis notation}
We apply standard notation from functional analysis. We will consider the space of square integrable functions $L^2(\Omega)$, square integrable functions with zero mean 
$$L^2_0(\Omega) := \left\{f\in L^2(\Omega) \Big| \ f^\Omega:= \frac{1}{|\Omega|}\int_\Omega f \; dx = 0\right\},$$ and the Sobolev spaces $H^1(\Omega)$ and $H_{\bm n}^2(\Omega) := \{f\in H^2(\Omega) | \ \nabla f \cdot \bm n = 0 \textrm{ on } \partial\Omega\}$. $\left(H_0^1(\Omega)\right)^d$ denotes the space of vector-valued $H^1(\Omega)$ functions with zero traces at the boundary and $H(\nabla\cdot,\Omega)$ is the space of vector-valued $L^2(\Omega)$ functions with divergence in $L^2(\Omega)$. Furthermore, we introduce the Hilbert space $H_0(\Div,\Omega) := \{\bm f \in H(\Div,\Omega) | \ \bm f\cdot\bm n = 0 \textrm{ on } \partial\Omega\}$. For ease of notation, we define the later occurring spaces for vector functions of dimension $d$ as $\vecSpL := \left(L^2(\Omega)\right)^d$ and $\vecSpHTr := \left(H_0^1(\Omega)\right)^d$, and additionally the space of tensor-valued $L^2(\Omega)$ functions $\tensSpL := \left(L^2(\Omega)\right)^{d \times d}$. Each of the spaces are equipped with norms $\| \cdot \|_X$, and inner products $( \cdot,\cdot )_X$, and if the subscripts are omitted we refer to the respective $L^2(\Omega)$ versions. We also introduce the dual space of $X$ as $X'$ equipped with the dual norm $\|\cdot \|_{X'}$, and the duality pairing $\langle\cdot,\cdot\rangle_{X',X}$. For the duality pairings with respect to $H^1(\Omega)',H^1(\Omega)$ and $\vecSpHM,\vecSpHTr$ we simply write $\langle \cdot,\cdot \rangle$. Moreover, we use the theory of Bochner spaces (e.g.~\cite{Boyer2013}, §5; \cite{TRoubi2013}, §1.5, §7.1). For a given Banach space $X$, we define
\begin{align*}
 &L^p([0,T];X):=\left\{f\!:[0,T] \rightarrow X \Big| \ f \textrm{ Bochner measurable, } \int_0^T \|f(t)\|_X^p\; dt < \infty\right\} \quad \textrm{for $1 \leq p < \infty$,} \\ 
 &L^\infty([0,T];X):=\left\{f\!:[0,T] \rightarrow X \Big| \ f \textrm{ Bochner measurable, } \underset{t\in[0,T]}{\mathrm{ess\, sup}}\,\|f(t)\|_X < \infty\right\} \quad \textrm{and} \\ 
 &H^1([0,T];X):=\left\{f \in L^2([0,T];X) \Big| \ \partial_t f \in L^2([0,T];X)\right\} \textrm{ (cf.~\cite{Evans2010}, §5.9.2).}
\end{align*}
In addition, we will frequently make use of classical inequalities such as the Cauchy-Schwarz, Hölder's, Young's and Minkowski's inequality (see e.g.~\cite{Hardy1988}). For the Poincar\'e-Wirtinger inequality we refer to Proposition III.2.39 of \cite{Boyer2013}, for Korn's inequality we refer to Theorem 1.33 of \cite{TRoubi2013} and we use the following adapted form of Grönwall's inequality (\cite{Teschl2012}, Lemma 2.7) with the extra term $v(t)$ on the left-hand side:
\begin{lemma}[Grönwall] \label{lemma:Grönwall}
Let $a$, $b$, $u$ and $v$ be real-valued functions defined on the interval $[0,T]$. Assume that $a$ is integrable and that $b$, $u$ and $v$ are continuous and non-negative. If $u$ and $v$ satisfy the integral inequality
\begin{align*}
 u(t) + v(t) \leq a(t) + \int_0^t b(s)u(s) \;ds \quad \textrm{for all } t \in [0,T],
\end{align*}
then
\begin{align*}
 u(t) + v(t) \leq a(t) + \int_0^t a(s)b(s)\exp\!{\left(\int_s^t b(r) \;dr\right)} \;ds \quad \textrm{for all } t \in [0,T],
\end{align*}
and for constant $b > 0$ and non-decreasing $a$
\begin{align*}
 u(t) + v(t) \leq a(t) e^{b t} \quad \textrm{for all } t \in [0,T]
\end{align*}
holds.
\end{lemma}

\subsection{Variational equations}
\begin{definition}[Weak solution] \label{def:Weak solution}
We say that the quintuple $(\pf, \mu, \bm u, \theta, \bm q)$ is a weak solution to the Cahn-Hilliard-Biot system if it holds that
\begin{align*}
\pf\in L^\infty\big([0,T];H^1(\Omega)\big)\cap H^1\big([0,T];H^1(\Omega)'\big)&, \quad \mu\in L^2\big([0,T];H^1(\Omega)\big), \\
\bm u\in L^\infty\big([0,T];\vecSpHTr\big)&, \quad \partial_t \nabla \cdot \bm u \in L^2\big([0,T];L^2(\Omega)\big), \\
\theta\in L^\infty\big([0,T];L^2(\Omega)\big)\cap H^1\big([0,T];H^1(\Omega)'\big)&, \quad \bm q\in L^2\big([0,T];\vecSpL\big)
\end{align*}
with $\pf(0) = \pf_0 \in H^1(\Omega)$, $\nabla \cdot \bm{u}(0) = \varepsilon_{v,0} \in L^2(\Omega)$, complying with Remark~\ref{remark:initial-displacement}, and for $\theta_0 \in L^2(\Omega)$  
\begin{align*}
 \forall g \in H^1(\Omega)\!: \quad \left\langle \theta(0),g \right\rangle = \left\langle \theta_0,g \right\rangle
\end{align*}
such that
\begin{subequations}
\begin{align}
 \int_0^T \left\langle \partial_t \pf, \eta^\mathrm{ch} \right\rangle + \big( m(\pf)\nabla\mu, \nabla\eta^\mathrm{ch} \big)\;dt &= \int_0^T\big( R, \eta^\mathrm{ch} \big)\;dt \label{eq:ch1weak}\\
 \int_0^T\big(\mu,\eta^\mathrm{ch}\big) - \big( \delta_\varphi\mathcal{E}_\mathrm{i}(\varphi),\eta^\mathrm{ch} \big) - \big(\delta_\varphi\mathcal{E}_\mathrm{e}(\varphi, \strain),\eta^\mathrm{ch}\big) - \big(\delta_\varphi \mathcal{E}_\mathrm{f}(\varphi, \strain, \theta),\eta^\mathrm{ch}\big) \; dt &= 0
 \label{eq:ch2weak}\\
 \int_0^T \regularization\big( \partial_t \nabla \cdot \bm u, \nabla \cdot \bm \eta^{\bm u}\big) +  \big(\delta_{\bm \varepsilon}\mathcal{E}_\mathrm{e}(\varphi, \strain),\bm \varepsilon(\bm \eta^{\bm u})\big) + \big(\delta_{\bm \varepsilon}\mathcal{E}_\mathrm{f}(\varphi, \strain, \theta),\bm \varepsilon(\bm \eta^{\bm u})\big)  \; dt &= \int_0^T \left\langle \bm f,\bm \eta^{\bm u} \right\rangle \; dt \label{eq:elasticityweak}\\
 %\int_0^T \big(\delta_{\bm \varepsilon}\mathcal{E}_\mathrm{e}(\varphi, \strain),\bm \varepsilon(\bm \eta^{\bm u})\big) + \big( \delta_{\bm \varepsilon}\energy_\mathrm{f}(\pf,\strain,\theta),\bm \varepsilon(\bm \eta^{\bm u})\big) \; dt &= \int_0^T \left\langle \bm f,\bm \eta^{\bm u} \right\rangle \; dt \label{eq:elasticityweak}\\
 \int_0^T\left\langle\partial_t\theta,\eta^{\theta}\right\rangle - \big( \bm q, \nabla\eta^{\theta}\big) \; dt &= \int_0^T\big( S_f,\eta^{\theta}\big)\; dt
 \label{eq:flowweak}\\
 \int_0^T \big( \kappa(\pf)^{-1} \bm q, \bm \eta^{\bm q} \big) - \big( \delta_\theta\energy_\mathrm{f}(\pf,\strain,\theta),\nabla \cdot \bm \eta^{\bm q} \big)\; dt &= 0 \label{eq:darcyflowweak} 
\end{align}
\label{eq:weak}%
\end{subequations}
for all $\eta^\mathrm{ch}\in L^2\big([0,T];H^1(\Omega)\big)$, $\bm \eta^{\bm u}\in L^2\big([0,T];\vecSpHTr\big)$, $\eta^\theta\in L^2\big([0,T];H^1(\Omega)\big)$, $\bm\eta^{\bm q} \in L^2\big([0,T];H(\nabla\cdot,\Omega)\big)$.
\end{definition}

\begin{remark}\label{remark:hdiv-flux}
    We highlight the weakened formulation of the original $\Div \bm{q}$ term in the mass balance equation~\eqref{eq:flowtheta} and the corresponding weak variant~\eqref{eq:flowweak}, related through integration by parts. Thus, local mass conservation is only enforced weakly. Additionally, the continuous weak formulation of the continuous system does not integrate a natural saddle-point structure in the fluid flow equations. Weak formulations of mass balance equations are commonly used in the literature~\cite{Feireisl2021}. Note, however, that the subsequent analysis also will employ spatial discretizations based on the natural saddle-point structure for simpler manipulation of the equations, as well as common practical discretization approaches using mixed methods.
\end{remark}

\section{Existence of solution to the variational system\label{sec:existence}}
Here, we state and prove the first main result, the existence of a weak solution to the variational system~\eqref{eq:weak} under a set of assumptions.
\subsection{Statement of the main existence result}
\begin{theorem}[Existence of weak solutions]\label{thm:existence}
Suppose that the following assumptions (A1)--(A7) hold:
\begin{itemize}
 \item[(A1)] The double-well potential $\Psi(\pf)$ is non-negative\revision{, i.e., $\Psi(\pf) \geq 0$ for all $\pf \in \mathbb{R}$,} and twice continuously differentiable with $\Psi'(0) = 0$. Furthermore, we assume that there exist positive constants $c_\Psi$ and $C_\Psi$ such that \revision{for all $\pf \in L^2(\Omega)$ with $\Psi(\pf) \in L^1(\Omega)$}
 \begin{equation}
  \|\Psi'(\pf)\|_{L^1(\Omega)}\leq C_\Psi\left(\|\Psi(\pf)\|_{L^1(\Omega)} + \|\pf\|^2_{L^2(\Omega)}\right) 
 \label{eq:boundpsiprime}
 \end{equation}
 and for all $s \in \mathbb{R}$
 \begin{equation}
  \Psi''(s) + c_\Psi \geq 0.
  %\Psi''(\pf) + c_\Psi \geq 0.
 \label{eq:boundpsidoubleprime}
 \end{equation}
 Note that this holds true for the classical choice of a double-well potential $\Psi(\pf) = (1-\pf^2)^2$ with constants $C_\Psi = 2$ and $c_\Psi\geq 4$.
 \item[(A2)] The Biot modulus $M$ (with integrated fluid compressibility) and Biot-Willis coupling coefficient $\alpha$ are state-independent and given by positive constants. The mobility $m$ and permeability $\kappa$ as functions of $\pf$ are continuous and uniformly bounded such that there exist suitable constants $c_m,\ C_m,\ c_\kappa,\ C_\kappa$ satisfying for all $\pf \in \mathbb{R}$
\begin{alignat*}{8}
0 &<& \hspace*{0.1cm} c_m &\leq& \hspace*{0.1cm} m(\pf) &\leq& \hspace*{0.1cm} C_m &<& \hspace*{0.1cm} \infty, \\ 
0 &<& c_\kappa &\leq& \kappa(\pf) &\leq& C_\kappa &<& \infty.
\end{alignat*}

\item[(A3)] The elastic stiffness tensor $\mathbb{C}(\pf)$ is Lipschitz continuous with Lipschitz constant $L_\mathbb{C}$ and differentiable. We further assume that there exist positive constants $c_{\mathbb{C}}$ and $C_{\mathbb{C}}$ such that
\begin{equation}
 c_{\mathbb{C}}|\bm \varepsilon|^2 \leq \bm \varepsilon\!:\!\mathbb{C}(\pf) \bm \varepsilon \leq C_{\mathbb{C}}|\bm \varepsilon|^2
\label{eq:elastictensorbound}
\end{equation}
holds for all $\pf \in \mathbb{R}$ and $\bm \varepsilon \in \mathbb{R}^{d\times d}_\mathrm{sym}$ and that the symmetry conditions on $\mathbb{C}$ known from linear elasticity, cf., e.g.,~\cite{Coussy2004}, are fulfilled.
\item[(A4)]
The eigenstrain ({\it stress-free strain} or {\it intrinsic strain}) $\mathcal{T}(\pf)$ is bounded by the $L^2(\Omega)$ norm of $\pf$ such that \revision{for all $\pf \in L^2(\Omega)$}
\begin{align*}
\|\mathcal{T}(\pf)\|_{\tensSpL} \leq C_{\mathcal{T}} \|\pf\|_{L^2(\Omega)}
\end{align*}
for a positive constant $C_\mathcal{T}$, Lipschitz continuous with Lipschitz constant $L_\mathcal{T}$, differentiable and symmetric.
Note that this holds true for the linear ansatz of Vegard's law \cite{Garcke2021}, that is $\mathcal{T}(\pf) = \bm{\hat \varepsilon} \pf$ with constant symmetric tensor $\bm{\hat \varepsilon}$, and thus also for the choice $\mathcal{T}(\pf) = \xi \pf \bm I$ \cite{Areias2016} with (positive) constant $\xi$.
\item[(A5)] The source terms are (for simplicity) assumed to be autonomous and satisfy $R, S_f\in L^2(\Omega)$ and $\bm f \in \vecSpHM$, with $\|R\|_{L^2(\Omega)} = C_R$, $\|S_f\|_{L^2(\Omega)} = C_S$ and $\|\bm f\|_{\vecSpHM} = C_{\bm f}$. 
\item[(A6)]The initial data satisfies $\pf_0\in H^1(\Omega)$, $\Psi(\pf_0)\in L^1(\Omega)$, $\varepsilon_{v,0} \in L^2(\Omega)$, $\theta_0\in L^2(\Omega)$, where $\varepsilon_{v,0}$ complies with $\pf_0$ and $\theta_0$ in the sense of Remark~\ref{remark:initial-displacement}.

\item[(A7)] The regularization parameter is positive $\eta > 0$.
\end{itemize}
Then, there exists a weak solution quintuple $(\pf, \mu, \bm u, \theta, \bm q)$ to \eqref{eq:model} with initial and boundary data as in \eqref{eq:init} in the sense of Definition~\ref{def:Weak solution}. Moreover, the energy estimate
 \begin{equation}
 \begin{aligned}
 \underset{t \in [0,T]}{\mathrm{ess\, sup}} \Big(\|\Psi&(\pf(t))\|_{L^1(\Omega)} + \|\pf(t)\|_{H^1(\Omega)}^2 + \|\bm u(t)\|_{\vecSpH}^2 + \|\theta(t)\|_{L^2(\Omega)}^2\Big) + \regularization \|\partial_t \nabla \cdot \bm u\|^2_{L^2([0,T];L^2(\Omega))} \\
 &+ \|\pf\|_{H^1\left([0,T];H^1(\Omega)'\right)}^2 + \|\theta\|_{H^1\left([0,T];H^1(\Omega)'\right)}^2 + \|\mu\|_{L^2([0,T];H^1(\Omega))}^2 + \|\bm q\|_{L^2([0,T];\vecSpL)}^2 \leq C
\end{aligned}
\label{eq:boundenergy}
\end{equation}
holds with constant $C$ merely depending on the data and material constants.
% \CR{depending on $\Omega$, $d$, $T$, $\gamma$, $\ell$, $C_R$, $C_S$, $\bm C_{\bm f}$, $C_\Psi$, $c_m$, $C_m$, $c_M$, $C_M$, $L_M$, $C_\alpha$, $L_\alpha$, $c_\mathbb{C}$, $C_\mathbb{C}$, $L_\mathbb{C}$, $C_\mathcal{T}$, $L_\mathcal{T}$ and the initial data}.
\end{theorem}

\subsection{Structure of the proof}
The proof of existence follows an established procedure, which also has been applied to analyze other tumor models, see e.g.~\cite{Garcke2021, Garcke2016, Lowengrub2013, Frigeri2015, FritzDarcy2019, FritzECM2019}:
\begin{enumerate}
\item We apply a Galerkin-type discretization in space to get a finite dimensional system of nonlinear time-dependent equations, cf.\ Section~\ref{sec:disc-in-space}.
\item We rephrase the system to a system of nonlinear ordinary differential equations, cf.\ Section~\ref{sec:local-in-time-solution}. 
\item We apply the Peano-Cauchy existence theorem to get a local in time solution, cf.\ Section~\ref{sec:local-in-time-solution}.
\item We derive {\it a priori} estimates in order to extend the existence interval to $[0,T]$, cf.\ Section~\ref{sec:a-priori-estimates}.
\item We use some standard compactness arguments to move towards the limit in space, cf.\ Section~\ref{sec:passingtolimit}.
\item We show that the derived limit is a weak solution to the Cahn-Hilliard-Biot model in the sense of Definition~\ref{def:Weak solution}, cf.\ Section~\ref{sec:passingtolimit}.
\end{enumerate}

\begin{remark}[Regularity of initial data\label{rem:assumption_initial_pf}]
    Note that we first prove the theorem with $\pf_0 \in H_{\bm n}^2(\Omega)$. This is needed in Lemma~\ref{lemma:initial-data} to construct discrete approximations of the initial data with a finite (interface) energy, i.e., in particular sufficient regularity in terms of the double-well potential.  In Section~\ref{sec:passingtolimit}, we complete the proof for the general condition $\pf_0 \in H^1(\Omega)$ with $\Psi(\pf_0) \in L^1(\Omega)$. The entire process follows an identical discussion as in~\cite{Garcke2021}.
\end{remark}

\revision{\begin{remark}[State-dependent material parameters]
The Cahn-Hilliard-Biot model enjoys a gradient flow structure also with the Biot modulus $M$ and Biot-Willis constant $\alpha$ stated as functions of $\pf$. The steps 1-4 of the proof of Theorem~\ref{thm:existence} can also be established for state-dependent, bounded, material parameters using a slightly altered discrete-in-space system (not utilizing the discrete energy from Section~\ref{sec:discrete_energy}). It does, however, get unnecessarily complicated to present both cases here, as the assumption of state independence is needed for the final steps of the proof.
\end{remark}}

\begin{remark}[Secondary consolidation]
Kelvin-Voigt type viscoelasticity has two effects. It models secondary consolidation, and it naturally adds more regularity to the displacement variable. The latter is essential in step 5 of our proof - the first four steps also hold for omitted regularization. The additional regularity in time on the volumetric strain, accompanied by the volume content will allow to deduce sufficient regularity for the pressure, under assumption (A2). Overall, this will be sufficient to discuss limits of nonlinear terms.
\end{remark}

\begin{remark}[On assumptions and Lemmas]
    The presentation of the proof is structured in Lemmas, and although it is not explicitly stated in each Lemma, all assumptions of Theorem~\ref{thm:existence} are assumed to hold throughout Section~\ref{sec:existence}.
\end{remark}

\subsection{Discrete-in-space system} \label{sec:disc-in-space}
Here, we present the spatial discretization of the system using the Faedo-Galerkin method \cite{Evans2010, TRoubi2013}. We define the following finite dimensional approximation spaces:
\begin{itemize}
 \item Let $\mathcal{V}_k^\mathrm{ch} = \Span\{\eta_i^\mathrm{ch}\}_{i=1}^k \subseteq H^1(\Omega)$, where $\left\{\eta_i^\mathrm{ch}\right\}_{i=2}^k \subseteq H^2_{\bm n}(\Omega)\cap L_0^2(\Omega)$ is the span of the first $k-1$ eigenfunctions of the Neumann-Laplace operator, defined as the operator that maps functions $f\in L^2_0(\Omega)$ to solutions $z\in L^2_0(\Omega)$ of the system
\begin{equation}
 \left(\nabla z,\nabla v\right) = \left(f, v\right)
\end{equation}
for all $v\in H^1(\Omega)$, and $\eta_1^\mathrm{ch} = \frac{1}{\left|\Omega\right|^\frac{1}{2}}$. Note that by construction the resulting basis is orthonormal in $L^2(\Omega)$ and orthogonal in $H^1(\Omega)$.
 \item Let $\mathcal{V}_k^{\bm u} = \Span\{\bm \eta_i^{\bm u}\}_{i=1}^k \subseteq \vecSpHTr$ be the span of the first $k$ eigenfunctions of the Dirichlet-Laplace operator that maps functions $\bm f\in \bm H^1_0(\Omega)$ to solutions $\bm z\in {\bm H}^1_0(\Omega)$
\begin{equation*}
\left( \nabla \bm z, \nabla \bm v\right) = \left(\bm f, \bm v\right)
\end{equation*}
for all $\bm v\in \bm H^1_0(\Omega)$.
See also (\cite{Lepedev2013}, Theorem 3.12.1) for further details.
 \item Let $\mathcal{V}_k^{\theta} = \Span\{\eta_i^{\theta}\}_{i=1}^k = \mathcal{V}_k^\mathrm{ch} \subseteq H^1(\Omega)$\revision{, where $\eta_i^{\theta} = \eta_i^\mathrm{ch}$ for all $i \in \{1,\ldots, k\}$}.
 \revision{Let $\Pi_k^\theta$ define the canonical $L^2$-projection onto $\mathcal{V}_k^\theta$.}
 \item We define $\mathcal{V}_k^{\bm q} = \Span\{\bm \eta_i^{\bm q}\}_{i=1}^k\subseteq H_0(\nabla\cdot,\Omega)$ as span of the unique solutions of the Poisson equation in mixed formulation~\cite{Boffi2013}: Find $(\bm\eta_i^{\bm q},\tilde\eta_i^{\theta}) \in H_0(\Div,\Omega) \times L_0^2(\Omega)$ such that 
 \begin{equation}
 \begin{aligned}
  \big( \bm\eta_i^{\bm q}, \bm q \big) - \big(\tilde\eta_i^{\theta},  \Div\bm q \big) &= 0 \\
  \big( \Div \bm \eta_i^{\bm q},\theta  \big) &= \big( \eta_i^\theta,\theta \big)
 \label{eq:mixedSpacesDarcy}
 \end{aligned}
 \end{equation}
 holds for all $(\bm{q}, \theta) \in H_0(\Div,\Omega) \times L_0^2(\Omega)$ 
 % (need to discuss BC)
 and $i \in \{1,\ldots,k\}$.
\end{itemize}

\begin{lemma}[Characterization of $\mathcal{V}_k^{\bm q}$\label{lemma:flux-basis}]
 Let $0=\lambda_1 < \lambda_2 \leq  \dots \leq \lambda_k$ denote the $k$ eigenvalues corresponding to $\{ \eta_i^{\theta} \}_{i=1}^k$ being the basis of $\mathcal{V}_k^{\theta}$, then the basis functions of $\mathcal{V}_k^{\bm q}$ satisfy $\lambda_i\bm{\eta}_i^{\bm{q}} = -\nabla  \eta_i^{\theta}$, $i=1,...,k$.
\end{lemma}
\begin{proof}
The proof is trivial for $i=1$. Let $i\geq 2$ with $\lambda_i>0$. As $\eta_i^{\theta} \in H^2_{\bm{n}}(\Omega) \cap L_0^2(\Omega)$, the ansatz $(\bm{\eta}_i^{\bm{q}}, \tilde{\eta}_i^\theta):=\left(-\frac{1}{\lambda_i} \nabla \eta_i^{\theta}, \frac{1}{\lambda_i} \eta_i^{\theta}\right) \in H_0(\nabla\cdot,\Omega) \times L_0^2(\Omega)$ solves~\eqref{eq:mixedSpacesDarcy}. Uniqueness of solutions~\cite{Boffi2013} proves the assertion.
\end{proof}

\revision{
\subsubsection{Discrete free energy}\label{sec:discrete_energy}
We introduce a discrete free energy $\mathcal{E}_k$, almost identical to $\mathcal{E}$, cf.\ equation~\eqref{eq:totenergy}, but making use of a projection operator. Let $\mathcal{E}_k$ be defined as
\begin{equation} \label{eq:totenergy_k}
    \mathcal{E}_k(\varphi_k,\strain[u_k], \theta_k) := \mathcal{E}_\mathrm{i}(\varphi_k) + \mathcal{E}_\mathrm{e}(\varphi_k, \strain[u_k]) +\mathcal{E}_{\mathrm{f},k}(\strain[u_k], \theta_k),
\end{equation}
where the discrete fluid contribution utilizes the $L^2$-projection operator $\Pi_k^\theta$ 
\begin{equation} \label{eq:fluidenergy_k}
    \mathcal{E}_{\mathrm{f},k}(\strain[u_k], \theta_k) := \int_\Omega \frac{M}{2}\left( \Pi_k^\theta \left(\theta_k - \alpha \nabla\cdot \bm u_k\right) \right)^2\; dx.
\end{equation}
In addition, we define the shorthand notation for the discrete pore pressure
\begin{equation} \label{eq:porepressure_k}
    \pi_k = M \, \Pi_k^\theta \left(\theta_k - \alpha \nabla\cdot \bm u_k\right).
\end{equation}
This, including the independence of $\mathcal{E}_{\mathrm{f},k}$ on $\varphi_k$, results in
\begin{align*}
    \delta_\varphi \mathcal{E}_{\mathrm{f},k}(\strain[u_k], \theta_k) &= 0, \\ %\quad \big(\delta_\varphi \mathcal{E}_{\mathrm{f},k}(\strain[u_k], \theta_k), \eta_j^\mathrm{ch}\big) = 0\\
    \delta_{\bm \varepsilon}\energy_{\mathrm{f},k}(\strain[u_k],\theta_k) &= - \alpha \pi_k I, \\ %\quad \big( \delta_{\bm \varepsilon}\energy_{\mathrm{f},k}(\strain[u_k],\theta_k),\bm \varepsilon(\bm \eta_j^{\bm u})\big) = -\int_\Omega \alpha \pi_k \remove{\Pi_k^\theta} \nabla \cdot \bm \eta_j^{\bm u} \, dx \\
    \delta_\theta\energy_{\mathrm{f},k}(\strain[u_k],\theta_k) &= \pi_k. %\quad \big( \delta_\theta\energy_{\mathrm{f},k}(\strain[u_k],\theta_k),\nabla \cdot \bm \eta_j^{\bm q} \big) = \int_\Omega \pi_k \nabla \cdot \bm \eta_j^{\bm q} \, dx.
\end{align*}
}

\revision{\subsubsection{The Galerkin discrete-in-space problem}}
The Galerkin discrete-in-space problem is defined \revision{analogously to~\eqref{eq:weak}, but based on the discrete free energy $\mathcal{E}_k$}: 
Find $(\pf_k, \mu_k, \bm u_k, \theta_k, \bm q_k)$ of the form
\begin{equation}
 \pf_k(t) = \sum_{i=1}^k a_i^k(t)\eta^\mathrm{ch}_i,\text{ } \mu_k(t) = \sum_{i=1}^k b_i^k(t)\eta^\mathrm{ch}_i,\text{ } \bm u_k(t) = \sum_{i=1}^k c_i^k(t)\bm \eta^{\bm u}_i,\text{ } \theta_k(t) = \sum_{i=1}^k d_i^k(t)\eta^\theta_i,\text{ }  \bm q_k(t) = \sum_{i=1}^k e_i^k(t)\bm \eta^{\bm q}_i
 \label{eq:discinspaceansatz}
\end{equation}
such that for a.e.~$t \in [0,T]$
\begin{subequations}
\begin{align}
\left\langle \partial_t \varphi_k, \eta_j^\mathrm{ch}\right\rangle + \big( m(\pf_k) \nabla \mu_k,\nabla \eta_j^\mathrm{ch}\big) &= \big( R,\eta_j^\mathrm{ch} \big),
\label{eq:ch1discspace}\\
\big(\mu_k ,\eta_j^\mathrm{ch}\big) - \big( \delta_\varphi\mathcal{E}_\mathrm{i}(\varphi_k),\eta_j^\mathrm{ch} \big) - \big(\delta_\varphi\mathcal{E}_\mathrm{e}(\varphi_k, \strain[u_k]),\eta_j^\mathrm{ch}\big) 
&=0
\label{eq:ch2discspace},\\
\regularization \big(\partial_t \nabla \cdot \bm u_k, \nabla \cdot \bm{\eta}_j^{\bm u} \big) + 
\big(\delta_{\bm \varepsilon}\mathcal{E}_\mathrm{e}(\varphi_k, \strain[u_k]),\bm \varepsilon(\bm \eta_j^{\bm u})\big) + \big( \delta_{\bm \varepsilon}\energy_{\mathrm{f},\revision{k}}(\strain[u_k],\theta_k),\bm \varepsilon(\bm \eta_j^{\bm u})\big) &= \left\langle \bm f,\bm \eta_j^{\bm u} \right\rangle 
\label{eq:elasticitydiscspace},\\
\left\langle\partial_t\theta_k,\eta_j^\theta\right\rangle + \big(\nabla\cdot \bm q_k, \eta_j^\theta\big) &= \big( S_f,\eta_j^\theta\big)
\label{eq:flowdiscspace},\\
\big( \kappa(\pf_k)^{-1}\bm q_k, \bm \eta_j^{\bm q} \big) - \big( \delta_\theta\energy_{\mathrm{f},\revision{k}}(\strain[u_k],\theta_k),\nabla \cdot \bm \eta_j^{\bm q} \big) &= 0,
\label{eq:darcyflowdiscspace}
\end{align}\label{eq:discspace}%
\end{subequations} 
holds for $j=1,\dots,k$. We further introduce the vectors of the coefficient functions $\bm a^k = \left[a_1^k, a_2^k, \dots, a_k^k\right]^\top$, $\bm b^k = \left[b_1^k,b_2^k, \dots, b_k^k\right]^\top$, $\bm c^k = \left[c_1^k, c_2^k, \dots, c_k^k\right]^\top$, $\bm d^k = \left[d_1^k, d_2^k, \dots, d_k^k\right]^\top$, and $\bm e^k = \left[e_1^k, e_2^k, \dots, e_k^k\right]^\top\in\mathbb{R}^k$.
As initial conditions for the discrete problem we impose $a_i^k(0) := \big(\pf_0,\eta^\mathrm{ch}_i\big)$, and $d_i^k(0):=\big(\theta_0, \eta^{\theta}_i\big)$ for $i=1,\ldots,k$, in other words $\pf_{k,0} = \Pi_k^\mathrm{ch} \pf_0$ and $\theta_{k,0} = \Pi_k^\theta \theta_{0}$ for all $k\in \mathbb{N}$ with the $L^2$-projections $\Pi_k^\mathrm{ch}$ and $\Pi_k^\theta$ mapping to $\mathcal{V}_k^\mathrm{ch}$ and $\mathcal{V}_k^\theta$ respectively. Similarly, for $c_i^k(0)$, implicitly described by 
\begin{align}
\label{eq:initial-displacement}
    \bm{u}_{k,0} := \underset{\bm{u}_k \in \mathcal{V}_k^{\bm u}}{\mathrm{arg\,min}}\, \mathcal{E}_{\revision{k}}(\pf_{k,0}, \strain[u_k], \theta_{k,0}),
\end{align}
%such 
we have that $\nabla \cdot \bm{u}_k(0) = \nabla \cdot \bm{u}_{k,0}$ by construction.

Notice that in the following, in \eqref{eq:ch2discspace}, we may use the substitution
\begin{align}
\big( \delta_\varphi\mathcal{E}_\mathrm{i}(\varphi_k),\eta_j^\mathrm{ch} \big) = \big( \frac{\gamma}{\ell} \Psi'(\pf_k) - \gamma\ell\Delta\pf_k, \eta_j^\mathrm{ch} \big) = \frac{\gamma}{\ell} \big( \Psi'(\varphi_k),\eta_j^\mathrm{ch}\big) + \gamma\ell\big(\nabla\varphi_k,\nabla\eta_j^\mathrm{ch}\big).
\label{eq:ch2discspacealter}
\end{align}

\subsection{Local-in-time solution for the discrete-in-space system} \label{sec:local-in-time-solution}
The above system can be understood as a system of ordinary differential equations in the time-dependent coefficient vectors $\bm a^k$, $\bm b^k$, $\bm c^k$, $\bm d^k$ and $\bm e^k$. 
\begin{lemma}
 There exists a $T_k \in (0,T]$ such that there is a local-in-time solution to \eqref{eq:discspace} defined on $[0,T_k]$ for all $k \in \mathbb{N}$. Furthermore, the coefficient functions are continuously differentiable, that is $\bm a^k$, $\bm b^k$, $\bm c^k$, $\bm d^k$, $\bm e^k \in C^1\big([0,T_k]; \mathbb{R}^k\big)$.
\end{lemma}
\begin{proof}
The goal is to reduce the partially degenerate, algebraic ODE system corresponding to the discretization to a classical ODE system allowing for applying the Peano-Cauchy existence theorem (\cite{Coddington1955}, Theorem 1.2).
We start by eliminating $\bm b^k$ and $\bm e^k$ from the system of equations. At first, we define the matrices $(\bm A_{xy})_{ji} = \big(\nabla\eta_i^y,\nabla\eta_j^x \big)$, $(\bm M_{xy})_{ji} =\big(\eta_i^y,\eta_j^x\big) $, and $(\bm B_{xy})_{ji} =\big(\Div\eta_i^y,\eta_j^x\big)$, with $x$, $y$ being placeholders for the exponents of associated test functions. Inserting the ansatz \eqref{eq:discinspaceansatz} into equation \eqref{eq:ch1discspace}, and utilizing the $L^2(\Omega)$-orthonormal property of $\mathcal{V}_k^\mathrm{ch}$, we get for $j=1,...,k$
\begin{align}
 &\langle \partial_t \varphi_k, \eta_j^\mathrm{ch} \rangle = \Big \langle \partial_t \sum_{i=1}^k a_i^k \eta^\mathrm{ch}_i,\eta_j^\mathrm{ch} \Big\rangle = \sum_{i=1}^k \big( \eta_i^\mathrm{ch}, \eta_j^\mathrm{ch} \big) \partial_t a_i^k = \sum_{i=1}^k \delta_{ij} \partial_t a_i^k = \partial_t a_j^k
 \label{eq:helperprojectionofphi} \\
 &\big( m(\pf_k) \nabla \mu_k,\nabla \eta_j^\mathrm{ch}\big) = \Big( m(\pf_k(\bm a^k)) \nabla \sum_{i=1}^k b_i^k \eta^\mathrm{ch}_i,\nabla \eta^\mathrm{ch}_j \Big) = \sum_{i=1}^k \big( m(\pf_k(\bm a^k)) \nabla \eta^\mathrm{ch}_i,\nabla \eta^\mathrm{ch}_j \big) b_i^k, \nonumber
\end{align}
with $\delta_{ij}$ denoting the Kronecker delta. Thus, by defining $\left(\bm A_{m,\mathrm{chch}}(\bm a^k)\right)_{ji} = \big( m(\pf_k(\bm a^k)) \nabla \eta^\mathrm{ch}_i,\nabla \eta^\mathrm{ch}_j \big)$ and the source term vector $\bm{\hat{r}}_j = \big( R,\eta^\mathrm{ch}_j \big)$, we obtain the ordinary differential equation
\begin{equation}\label{eq:ch1matrix}
\partial_t\bm a^k + \bm A_{m,\mathrm{chch}}(\bm a^k) \bm b^k  = \bm{\hat{r}}.
\end{equation}
For the other equations we proceed in a similar way. We can rewrite equation \eqref{eq:ch2discspace} with the substitution \eqref{eq:ch2discspacealter} as
\begin{equation*} 
 \bm b^k - \gamma\ell \bm A_{\mathrm{chch}}\bm a^k - \frac{\gamma}{\ell} \bm \psi'(\bm a^k) - \bm \energy^{\delta_\pf}_{\mathrm{e}}(\bm a^k,\bm c^k) = 0,
\end{equation*}
where we introduce $\bm \psi'(\bm a^k)_j = \big(\Psi'\big(\varphi_k(\bm a^k)\big),\eta_j^\mathrm{ch}\big)$ and $\bm \energy^{\delta_\pf}_{\mathrm{e}}(\bm a^k,\bm c^k)_j 
= \big( \delta_\varphi\mathcal{E}_\mathrm{e}\big(\varphi_k(\bm a^k), \strain[u_k(\bm c^k)]\big),\eta_j^\mathrm{ch} \big)$. By substituting $\bm b^k$ into equation \eqref{eq:ch1matrix} we eliminate $\bm b^k$ from the system. Next, equation \eqref{eq:elasticitydiscspace} can be written as 
\begin{equation*}
\regularization \bm C_{\bm{\varepsilon}} \partial_t \bm c^k + \bm \energy^{\delta_{\bm \varepsilon}}_{\mathrm{e}}(\bm a^k,\bm c^k) + \bm \energy^{\delta_{\bm \varepsilon}}_{\mathrm{f}}(\bm c^k, \bm d^k) =  \bm{\hat{f}},
\end{equation*}
where we define $\big(\bm C_{\bm{\varepsilon}}\big)_{ji} 
= \big(\Div \bm \eta_i^{\bm u},\Div \bm \eta_j^{\bm u} \big)$, 
$\bm \energy^{\delta_{\bm \varepsilon}}_{\mathrm{e}}(\bm a^k,\bm c^k)_j 
= \big(\delta_{\bm \varepsilon}\energy_\mathrm{e}\big(\pf_k(\bm a^k),\strain[u_k(\bm c^k)]\big),\bm \varepsilon(\bm \eta_j^{\bm u}) \big)$, $\bm \energy^{\delta_{\bm \varepsilon}}_{\mathrm{f}}(\bm c^k, \bm d^k)_j 
= \big( \delta_{\bm \varepsilon}\energy_{\mathrm{f},\revision{k}}\big(\strain[u_k(\bm c^k)],\theta_k(\bm d^k)\big),\bm \varepsilon(\bm \eta_j^{\bm u}) \big)$ 
and $\bm{\hat{f}}_j = \left\langle \bm f, \bm \eta_j^{\bm u} \right\rangle$.
We can rewrite $\bm \energy^{\delta_{\bm \varepsilon}}_{\mathrm{e}}(\bm a^k,\bm c^k)_j$ as
\begin{align*}
\bm \energy^{\delta_{\bm \varepsilon}}_{\mathrm{e}}(\bm a^k,\bm c^k)_j 
&= \big( \mathbb{C}(\varphi_k(\bm a^k))\big({\bm \varepsilon}(\bm u_k(\bm c^k))-\mathcal{T}(\varphi_k(\bm a^k))\big),\bm \varepsilon(\bm \eta_j^{\bm u}) \big) \\
% &= \sum_{i=1}^k \big( \mathbb{C}(\varphi_k(\bm a^k))\bm \varepsilon(\bm \eta_i^{\bm u}),\bm \varepsilon(\bm \eta_j^{\bm u})\big) c_i^k(t) - \big( \mathbb{C}(\varphi_k(\bm a^k))\mathcal{T}(\varphi_k(\bm a^k)),\bm \varepsilon(\bm \eta_j^{\bm u}) \big) \\
&=: \bm E_{\bm{\varepsilon}}(\bm a^k)_j\bm c^k - \bm t_{\bm{\varepsilon}}(\bm a^k)_j,
\end{align*}
with $\big(\bm E_{\bm{\varepsilon}}(\bm a^k)\big)_{ji} = \big( \mathbb{C}(\varphi_k(\bm a^k))\bm \varepsilon(\bm \eta_i^{\bm u}),\bm \varepsilon(\bm \eta_j^{\bm u}) \big)$ and $\bm t_{\bm{\varepsilon}}(\bm a^k)_j = \big( \mathbb{C}(\pf_k(\bm a^k))\mathcal{T}(\pf_k(\bm a^k)),\bm \varepsilon(\bm\eta_j^{\bm u}) \big)$. \\ 
Similarly, \revision{by using the definition and linearity of $\Pi_k^\theta$,} for $\bm\energy^{\delta_{\bm \varepsilon}}_\mathrm{f}(\bm c^k,\bm d^k)_j$ we obtain  
\begin{align*}
 \bm \energy^{\delta_{\bm \varepsilon}}_{\mathrm{f}}(\bm c^k, \bm d^k)_j
 &= \big( -\alpha M\theta_k(\bm d^k)\bm I + M\alpha^2 \revision{\Pi_k^\theta} \Div \bm u_k(\bm c^k)\bm I ,\bm \varepsilon( \bm \eta_j^{\bm u})\big)\\
 % &= \big( -\alpha(\pf_k(\bm a^k))M(\varphi_k(\bm a^k))\theta_k(\bm d^k)\bm I,\bm \varepsilon( \bm \eta_j^{\bm u})\big) + \sum_{i=1}^k \big( M(\varphi_k(\bm a^k))\alpha(\pf_k(\bm a^k))^2 \Div \bm \eta_i^{\bm u}\Div \bm \eta_j^{\bm u}\big) c_i^k(t) \\
&=: {\bm u_{\bm{\varepsilon}}(\bm d^k)}_j + {\bm F_{\bm{\varepsilon}}}_j\bm c^k,
\end{align*}
with the matrix given as $\big(\bm F_{\bm{\varepsilon}}\big)_{ji} = \big( M\alpha^2 \revision{\Pi_k^\theta} \Div \bm \eta_i^{\bm u},\Div \bm \eta_j^{\bm u} \big)$ and the vector $\bm u_{\bm{\varepsilon}}(\bm d^k)_j = \big( -\alpha M\theta_k(\bm d^k)\bm I,\bm \varepsilon( \bm \eta_j^{\bm u})\big)$. 
Therefore, equation \eqref{eq:elasticitydiscspace} reduces to
\begin{align}
\label{eq:algebraic-compact-elasticity}
\regularization \bm C_{\bm{\varepsilon}} \partial_t \bm c^k + \left(\bm E_{\bm{\varepsilon}}(\bm a^k) + \bm F_{\bm{\varepsilon}}\right)\bm c^k = \bm{\hat{f}} + \bm t_{\bm{\varepsilon}}(\bm a^k) - \bm u_{\bm{\varepsilon}}(\bm d^k).
\end{align}
Since $\bm C_{\bm{\varepsilon}}$ is symmetric positive semi-definite, this allows for a spectral decomposition $\bm C_{\bm{\varepsilon}} = \bm{Q} \begin{bmatrix} \bm{D} & \bm{0} \\ \bm{0} & \bm{0}\end{bmatrix} \bm{Q}^\top$, where $\bm{Q}$ is orthogonal and $\bm{D}$ diagonal and invertible. Introducing a variable transformation $\bm{c}_Q^k:=\bm{Q}^\top \bm{c}^k$, we essentially split $\bm{c}^k$ into its volumetric and deviatoric contributions and thus write $\bm{c}_Q^k = [\bm{c}_v^k, \bm{c}_d^k$]. Introducing this transformation into~\eqref{eq:algebraic-compact-elasticity} yields
\begin{align*}
     \begin{bmatrix} \partial_t\bm c_v^k \\ \bm 0 \end{bmatrix} + \begin{bmatrix} \regularization^{-1}\bm{D}^{-1} & \bm 0 \\ \bm 0 & \bm I \end{bmatrix} \bm{Q}^\top \left(\bm E_{\bm{\varepsilon}}(\bm a^k) + \bm F_{\bm{\varepsilon}}\right) \bm{Q} \begin{bmatrix} \bm c_v^k \\ \bm c_d^k \end{bmatrix} = \begin{bmatrix} \regularization^{-1}\bm{D}^{-1} & \bm 0 \\ \bm 0 & \bm I \end{bmatrix} \bm{Q}^\top \left[\bm{\hat{f}} + \bm t_{\bm{\varepsilon}}(\bm a^k) - \bm u_{\bm{\varepsilon}}(\bm d^k)\right].
\end{align*}
Furthermore, $\bm E_{\bm{\varepsilon}}(\bm a^k) + \bm F_{\bm{\varepsilon}}$ is symmetric positive definite. To show this, we introduce a fixed but arbitrary vector $\bm x \in \mathbb{R}^k\setminus\{\bm 0\}$ and, using assumptions (A3) and (A2) \revision{and the definition and linearity of $\Pi_k^\theta$}, consider
\begin{align*}
 \bm x^T \bm E_{\bm{\varepsilon}}(\bm a^k) \bm x &= \sum_{j=1}^k x_j \sum_{i=1}^k \big( \mathbb{C}(\varphi_k(\bm a^k))\bm \varepsilon(\bm \eta_i^{\bm u}),\bm \varepsilon(\bm \eta_j^{\bm u}) \big) x_i 
 % = \left( \mathbb{C}(\varphi_k(\bm a^k)) \bm \varepsilon\left(\sum_{i=1}^k x_i \bm \eta_i^{\bm u}\right),\bm \varepsilon\left(\sum_{j=1}^k x_j \bm \eta_j^{\bm u}\right) \right) \\
 \overset{\eqref{eq:elastictensorbound}}{\geq} c_{\mathbb{C}} \bigg\|\bm \varepsilon\!\left(\sum_{i=1}^k x_i \bm \eta_i^{\bm u}\right)\bigg\|_{\tensSpL}^2 > 0,\\
% \end{align*}
% and similarly with the boundedness of $M(\cdot)$ and $\alpha(\cdot)$, cf.\ assumption (A2$^\star$),
% \begin{align*}
 \bm x^T \bm F_{\bm{\varepsilon}}\bm x &= \sum_{j=1}^k x_j \sum_{i=1}^k \big( M\alpha^2 \revision{\Pi_k^\theta} \Div \bm \eta_i^{\bm u}, \revision{\Pi_k^\theta} \Div \bm \eta_j^{\bm u} \big) x_i 
 % = \bigg( M(\varphi(\bm a^k)) \alpha(\pf(\bm a^k))^2 \Div\left(\sum_{i=1}^k x_i \bm \eta_i^{\bm u}\right),\Div\left(\sum_{j=1}^k x_j \bm \eta_j^{\bm u}\right) \bigg) \\
 \geq M \alpha^2 \bigg\|\revision{\Pi_k^\theta} \Div\left(\sum_{i=1}^k x_i \bm \eta_i^{\bm u}\right)\bigg\|_{L^2(\Omega)}^2 \geq 0.
\end{align*}
Thus, also $\bm{Q}^\top \left(\bm E_{\bm{\varepsilon}}(\bm a^k) + \bm F_{\bm{\varepsilon}}\right) \bm{Q}$ is symmetric positive definite, and we can eliminate $\bm{c}_d^k$, reducing to
\begin{align*}
     \partial_t\bm c_v^k + \bm A_{vv} \bm c_v^k = \bm b_v(\bm a^k, \bm d^k)
\end{align*}
for a suitable, invertible matrix $\bm A_{vv}$ and vector $\bm{b}_v$ continuously depending on $\bm{a}^k$ and $\bm{d}^k$.

The last two equations \eqref{eq:flowdiscspace} and \eqref{eq:darcyflowdiscspace} can be written as 
\begin{equation}
 \underbrace{\bm M_{\theta\theta}}_{= \bm I} \partial_t \bm d^k +  \bm B_{\theta\bm q}\bm e^k = \bm{\hat{s}_f} 
 \label{eq:flowdiscspacematrix}
\end{equation}
with $\bm{\hat{s}_f}{_j} = \big( S_f,\eta_j^\theta \big)$
and 
\begin{equation}
 \bm M_{\kappa,\bm q\bm q}(\bm a^k)\bm e^k - \bm \energy_\mathrm{f}^{\delta_\theta}(\bm c^k,\bm d^k) = 0,
\end{equation}
for $\bm M_{\kappa,\bm q\bm q}(\bm a^k)_{ji} = \big( \kappa(\pf_k(\bm a^k))^{-1}\bm \eta_i^{\bm q},\bm \eta_j^{\bm q} \big)$, and $\bm \energy_{\mathrm{f}}^{\delta_\theta}(\bm c^k,\bm d^k)_{j} = \big( \delta_\theta\energy_{\mathrm{f},\revision{k}}(\strain[u_k(\bm c^k)],\theta_k(\bm d^k)),\nabla \cdot \bm \eta_j^{\bm q} \big)$.
By inverting $\bm M_{\kappa,\bm q\bm q}(\bm a^k)$ we get the coefficient vector $\bm e^k$ that we can eliminate from the system by inserting it into equation \eqref{eq:flowdiscspacematrix}. Note that the inverse is well-defined since $\bm M_{\kappa,\bm q\bm q}(\bm a^k)$ is positive definite due to assumption (A2).
To summarize, the set of equations \eqref{eq:ch1discspace}--\eqref{eq:darcyflowdiscspace} can be written as a system of ordinary differential equations in the form of
\begin{equation}
\partial _t \begin{pmatrix} \bm a^k \\ \bm c_v^k \\ \bm d^k \end{pmatrix} = H\begin{pmatrix} \bm a^k \\ \bm c_v^k \\ \bm d^k \end{pmatrix}
\label{eq:generalODEsystem}
\end{equation}
where the right-hand side depends in a nonlinear way on the solution. Since we assume continuity for $m(\cdot)$, $\kappa(\cdot)$ and $\Psi'(\cdot)$ and by the Lipschitz continuity also for $\mathbb{C}(\cdot)$ and $\mathcal{T}(\cdot)$, cf.\ assumptions (A1)--(A4), we can infer that $H$ is a continuous function.
Considering the initial value problem resulting from adding the initial conditions $\bm a^k(0)$, $\bm c_v^k(0)$ and $\bm d^k(0)$ to the system, we can therefore apply the Peano-Cauchy existence theorem to obtain the existence of a $T_k \in (0,T]$ and a local solution triple $(\bm a^k, \bm c_v^k, \bm d^k)$ to the system of equations \eqref{eq:generalODEsystem} defined on $[0,T_k]$ for all $k \in \mathbb{N}$.  By the derivation this holds also true for the other coefficient functions $\bm b^k$, $\bm{c}_d^k$ (and thus $\bm{c}^k$) and $\bm e^k$, yielding the existence of a local-in-time solution to \eqref{eq:discspace} for $t\in [0,T_k]$. Furthermore, the Peano-Cauchy existence theorem allows us to infer that the coefficient functions are continuously differentiable, that is $\bm a^k$, $\bm b^k$, $\bm c^k$, $\bm d^k$, $\bm e^k \in C^1\big([0,T_k]; \mathbb{R}^k\big)$.
\end{proof}

The \textit{a priori} estimates derived in the subsequent section will be uniform in $k$ allowing to conclude that $T_k = T$ for all $k \in \mathbb{N}$. The constructed initial data has finite energy, as we briefly conclude in the following two lemmas.

\begin{lemma}\label{lemma:initial-displacement-bound}
For the initial displacement $\bm{u}_{k,0}$, defined in~\eqref{eq:initial-displacement}, there exists a positive constant $\zeta_{\bm \varepsilon, 0}$ such that
% \begin{equation}
%  \|\bm u_{k,0}\|_{\vecSpH} \leq \zeta_{\bm u, 0} \Big(\|\pf_{k,0}\|_{L^2(\Omega)} + \|\theta_{k,0}\|_{L^2(\Omega)} + 1\Big)
% \label{eq:initialDisplacement}
% \end{equation}
% and
\begin{equation}
 \frac{1}{C_K} \|\bm u_{k,0}\|_{\vecSpH} \leq \|\bm \varepsilon(\bm u_{k,0})\|_{\tensSpL} \leq \zeta_{\bm \varepsilon, 0} \Big(\|\pf_{k,0}\|_{L^2(\Omega)} + \|\theta_{k,0}\|_{L^2(\Omega)} + 1\Big).
\label{eq:initialStrain}
\end{equation}
\end{lemma}
\begin{proof}
The necessary conditions, corresponding to~\eqref{eq:initial-displacement}, read
\begin{align*}
 \big( \delta_{\bm \varepsilon}\energy_\mathrm{e}(\pf_{k,0},\strain[u_{k,0}]),\bm \varepsilon(\bm \eta_j^{\bm u}) \big) + \big( \delta_{\bm \varepsilon}\energy_{\mathrm{f},\revision{k}}(\strain[u_{k,0}],\theta_{k,0}),\bm \varepsilon(\bm \eta_j^{\bm u}) \big) = \left\langle \bm f,\bm \eta_j^{\bm u} \right\rangle\text{ for }j=1,...,k,
\end{align*}
which allow for a unique solution based on classical linear elasticity theory. As by construction $\bm u_{k,0} \in \mathcal{V}_k^{\bm u}$, suitable weighting of the test functions and summation over $j=1,...,k$ yields
\begin{align*}
 \big( \delta_{\bm \varepsilon}\energy_\mathrm{e}(\pf_{k,0},\strain[u_{k,0}]),\bm \varepsilon(\bm u_{k,0}) \big) + \big( \delta_{\bm \varepsilon}\energy_{\mathrm{f},\revision{k}}(\strain[u_{k,0}],\theta_{k,0}),\bm \varepsilon(\bm u_{k,0}) \big) = \langle \bm f,\bm u_{k,0} \rangle
\end{align*}
which\revision{, by using the definition of $\Pi_k^\theta$,} is equivalent to
\begin{equation}
\begin{aligned}
 \big(\mathbb{C}(\varphi_{k,0}){\bm \varepsilon}(\bm u_{k,0}),\bm \varepsilon(\bm u_{k,0})\big) &+ \big( M\alpha^2 \revision{\Pi_k^\theta} \nabla\cdot \bm u_{k,0}, \revision{\Pi_k^\theta}\nabla \cdot \bm u_{k,0}\big) \\
 &= \big( \mathbb{C}(\pf_{k,0})\mathcal{T}(\varphi_{k,0})+M\alpha\theta_{k,0}\bm I,\bm \varepsilon(\bm u_{k,0})\big) + \langle \bm f,\bm u_{k,0} \rangle.
\end{aligned}
\label{eq:initialdef}
\end{equation}
On the left hand side, we employ (A3) for the first term and drop the second one, while on the right hand side we employ the Cauchy-Schwarz inequality, the duality pairing and Korn's inequality together with assumptions (A2)--(A6), to obtain
\begin{align*}
    c_{\mathbb{C}} \|\bm \varepsilon(\bm u_{k,0})\|_{\tensSpL}^2
    \leq  
    \left(C{_\mathbb{C}} C{_\mathcal{T}} \|\pf_{k,0}\|_{L^2(\Omega)} + M \alpha d^{\frac{1}{2}} \|\theta_{k,0}\|_{L^2(\Omega)} + C_K\|\bm f\|_{\vecSpHM}\right) \, \|\bm \varepsilon(\bm u_{k,0})\|_{\tensSpL}.
\end{align*}
Setting $ \zeta_{\bm \varepsilon, 0} := \frac{1}{c_{\mathbb{C}}} \max\{C{_\mathbb{C}} C{_\mathcal{T}}, M \alpha d^{\frac{1}{2}}, C_KC_{\bm f}\}$ yields the asserted bounds and concludes the proof.
\end{proof}

\begin{lemma}[Stability of initial data\label{lemma:initial-data}]
For the initial data $\theta_{k,0}$, $\pf_{k,0}$, and $\bm u_{k,0}$ as constructed in Section~\ref{sec:disc-in-space}, there exists a constant $C_{\mathcal{E},0}$ such that for all $k$ it holds that
\begin{align}
\label{eq:bound-initial-data}
    \mathcal{E}_{\revision{k}}(\pf_{k,0}, \bm{\varepsilon}(\bm u_{k,0}), \theta_{k,0}) \leq C_{\mathcal{E},0}.
\end{align}
\end{lemma}
\begin{proof}
Under assumption (A6) and additionally $\varphi_{0}\in H^2_{\bm n}(\Omega)$, cf.\ Remark~\ref{rem:assumption_initial_pf}, it follows from the definition of the initial discrete data via projection that there exist positive constants $\bar \zeta_{\pf,0},\ \tilde \zeta_{\pf,0},\ \bar \zeta_{\theta,0}$ such that
\begin{alignat*}{7}
    \|\pf_{k,0}\|_{L^2(\Omega)} 
    &=& \hspace*{0.1cm} \|\Pi_k^\mathrm{ch} \pf_0\|_{L^2(\Omega)} 
    &\leq& \hspace*{0.1cm} \|\pf_0\|_{L^2(\Omega)} 
    &\leq& \hspace*{0.1cm} \bar \zeta_{\pf,0}, \\
    \|\varphi_{k,0}\|^2_{H^1(\Omega)}
    &\leq& \|\varphi_{k,0}\|^2_{H^2(\Omega)} 
    &\leq& \hspace*{0.1cm} C\|\varphi_{0}\|^2_{H^2(\Omega)} 
    &\leq& \tilde \zeta_{\pf,0},\\
    \|\theta_{k,0}\|_{L^2(\Omega)} 
    &=& \|\Pi_k^\theta \theta_0\|_{L^2(\Omega)} 
    &\leq& \|\theta_0 \|_{L^2(\Omega)} 
    &\leq& \bar \zeta_{\theta,0}.
\end{alignat*}
Moreover, carefully following the identical steps as in~\cite{Garcke2021}, as $\pf_0\in H^2_{\bf n}(\Omega)$, we have that $\pf_{k,0} = \Pi_k^\mathrm{ch} \pf_0$ converges strongly to $\pf_0$ in $H^2_{\bf n}(\Omega)$ and hence also a.e.~in $\Omega$. The continuity assumption (A1) and the fact that  $\Psi(\pf_0)\in L^1(\Omega)$ allow to deduce the uniform bound 
\begin{align}
 \|\Psi(\pf_{k,0})\|_{L^1(\Omega)} \leq \zeta_{\Psi,0}.
\label{eq:boundpsiinit}
\end{align}
For the initial displacement, we deduce from Lemma~\ref{lemma:initial-displacement-bound} and Korn's inequality
\begin{align*}
 \|\bm u_{k,0}\|_{\vecSpH} \leq C_K \zeta_{\bm \varepsilon, 0} \big(\bar \zeta_{\pf,0} + \bar \zeta_{\theta,0} + 1\big).
\end{align*}
Finally, using assumption (A2), and the above bounds, we obtain the collective estimate~\eqref{eq:bound-initial-data}, where $C_{\mathcal{E},0}$ depends on the above uniform bounds, as well as natural modeling constants.
\end{proof}

\subsection{\textit{A priori} estimates} \label{sec:a-priori-estimates}
In the following we aim to derive uniform \textit{a priori} estimates for the discrete solution. For this, we will make use of an energy-dissipation identity. 

\begin{lemma}[Energy-dissipation identity\label{lemma:energy-dissipation}]
The discrete-in-space solution $(\pf_k, \mu_k, \bm u_k, \theta_k, \bm q_k)$ of~\eqref{eq:discspace} satisfies
\begin{equation}
\begin{aligned}
 \revision{\mathcal{E}_k}\big(\pf_k(T_k), &\strain[u_k(T_k)], \theta_k(T_k)\big) + \int_0^{T_k} \left[\regularization \| \partial_t \nabla \cdot \bm u_k \|_{L^2(\Omega)}^2 + \|m(\pf_k)^{\frac{1}{2}}\nabla\mu_k\|_{\vecSpL}^2 + \|\kappa(\pf_k)^{-\frac{1}{2}}\bm q_k\|_{\vecSpL}^2 \right]\; dt\\
 &= \revision{\mathcal{E}_k}(\pf_{k,0}, \strain[u_{k,0}], \theta_{k,0}) + \int_0^{T_k} \left[\big( R,\mu_k \big) + \langle \bm f,\partial_t\bm u_k \rangle + \big( S_f,\delta_\theta\energy_{\mathrm{f},\revision{k}}(\strain[u_k],\theta_k) \big) \right] \;dt.
\end{aligned}
\label{eq:energydissipation2}
\end{equation}
Moreover, under the absence of external contributions, i.e., $R = 0$, $\bm f = \bm 0$, and $S_f = 0$, the energy dissipation property $\tfrac{d}{dt}\mathcal{E}_{\revision{k}}(\pf_k, \strain[u_k], \theta_k) \leq 0$ is satisfied.
\end{lemma}

\begin{proof}
We multiply equation \eqref{eq:ch1discspace} by $b_j^k(t)$, \eqref{eq:ch2discspace} by $-\left(a_j^k\right)'(t)$, \eqref{eq:elasticitydiscspace} by $\left(c_j^k\right)'(t)$, \eqref{eq:flowdiscspace} by the inner product $\big(\delta_\theta\energy_{\mathrm{f},\revision{k}}(\strain[u_k],\theta_k),\eta_j^\theta\big)$, and \eqref{eq:darcyflowdiscspace} by $e_j^k(t)$ and sum over $j=1,\ldots,k$ to get
\begin{align*}
\left\langle \partial_t \pf_k, \mu_k \right\rangle + \|m(\pf_k)^{\frac{1}{2}}\nabla\mu_k\|_{\vecSpL}^2 &= \big( R, \mu_k\big), \\
% K\left\langle \partial_t \pf_k, \pf_k \right\rangle + K\big( m(\pf_k)\nabla\mu_k,\nabla\pf_k\big) &= K\big( R, \pf_k\big) \\
-\left\langle \mu_k,\partial_t\pf_k \right\rangle + \left\langle\delta_\varphi \mathcal{E}_{\revision{k}}(\varphi_k, \strain[u_k], \theta_k), \partial_t\pf_k\right\rangle &= 0, \\
% -\left\langle \mu_k,\partial_t\pf_k \right\rangle + \left\langle\delta_\varphi\mathcal{E}_\mathrm{i}(\pf_k),\partial_t\pf_k\right\rangle + \left\langle\delta_\varphi\mathcal{E}_\mathrm{e}(\varphi_k, \strain[u_k]),\partial_t\pf_k\right\rangle + \left\langle\delta_\varphi \mathcal{E}_\mathrm{f}(\varphi_k, \strain[u_k], \theta_k), \partial_t\pf_k\right\rangle &= 0, \\
\regularization \| \partial_t \nabla \cdot \bm u_k \|_{L^2(\Omega)}^2 + \left\langle \delta_{\bm \varepsilon}\mathcal{E}_{\revision{k}}(\pf_k,\strain[u_k],\theta_k),\bm \varepsilon(\partial_t\bm u_k) \right\rangle &= \left\langle \bm f,\partial_t\bm u_k \right\rangle, \\
\left\langle \partial_t\theta_k,\delta_\theta\mathcal{E}_\mathrm{f,\revision{k}}(\strain[u_k], \theta_k) \right\rangle + \big(\nabla\cdot\bm q_k,\delta_\theta\mathcal{E}_{\mathrm{f},\revision{k}}(\strain[u_k], \theta_k)\big) &= \big( S_f,\delta_\theta\mathcal{E}_{\mathrm{f},\revision{k}}(\strain[u_k],\theta_k)\big), \\
\|\kappa(\pf_k)^{-\frac{1}{2}}\bm q_k\|_{\vecSpL}^2 - \big( \delta_\theta\mathcal{E}_{\mathrm{f},\revision{k}}(\strain[u_k],\theta_k),\nabla \cdot \bm q_k \big) &= 0.
\end{align*}
By adding the above equations and employing the chain rule we get
\begin{equation}
\begin{aligned}
&\frac{d}{dt}\mathcal{E}_{\revision{k}}(\pf_k, \strain[u_k], \theta_k) + \regularization \|\partial_t \nabla \cdot \bm{u}_k\|_{L^2(\Omega)}^2 + \|m(\pf_k)^{\frac{1}{2}}\nabla\mu_k\|_{\vecSpL}^2 + \|\kappa(\pf_k)^{-\frac{1}{2}}\bm q_k\|_{\vecSpL}^2 \\ 
&\hspace{1cm} = \big( R,\mu_k \big) + \big( S_f,\delta_\theta\energy_{\mathrm{f},\revision{k}}(\strain[u_k],\theta_k) \big) + \langle \bm f,\partial_t\bm u_k \rangle.
\end{aligned}
\label{eq:addedweakform-zero-k}
\end{equation}
Integration in time over the existence interval $[0,T_k]$ concludes the proof.
\end{proof}

\begin{corollary}
Let $(\pf_k, \mu_k, \bm u_k, \theta_k, \bm q_k)$ be the discrete-in-space solution of~\eqref{eq:discspace}. For any $K\in\mathbb{R}$ it holds that
\begin{equation}
\begin{aligned}
&\frac{d}{dt} \mathcal{E}_{\revision{k}}(\pf_k, \strain[u_k], \theta_k) + \frac{d}{dt} K\frac{\|\pf_k\|_{L^2(\Omega)}^2}{2} + 
\regularization \|\partial_t \nabla \cdot \bm{u}_k\|_{L^2(\Omega)}^2 + \|m(\pf_k)^{\frac{1}{2}}\nabla\mu_k\|_{\vecSpL}^2 + \|\kappa(\pf_k)^{-\frac{1}{2}}\bm q_k\|_{\vecSpL}^2 \\ 
&\hspace{1cm} = \big( R,\mu_k \big) + K\big( R,\pf_k \big) + \big( S_f, \delta_\theta\energy_{\mathrm{f},\revision{k}}(\strain[u_k],\theta_k) \big) - K\big( m(\pf_k)\nabla\mu_k,\nabla\pf_k \big) + \langle \bm f,\partial_t\bm u_k \rangle.
\end{aligned}
\label{eq:addedweakform}
\end{equation}
\end{corollary}
\begin{proof}
We multiply equation \eqref{eq:ch1discspace} by $Ka_j^k(t)$ and sum over $j=1,...,k$ to obtain
\begin{align*}
    K\left\langle \partial_t \pf_k, \pf_k \right\rangle + K\big( m(\pf_k)\nabla\mu_k,\nabla\pf_k\big) &= K\big( R, \pf_k\big).
\end{align*}
Summation onto~\eqref{eq:addedweakform-zero-k} proves the assertion.
\end{proof}

Before we continue with two results on \textit{a priori} estimates, we provide lower bounds for the \revision{discrete} elastic and hydraulic energies\revision{, followed by a lower bound for the discrete free energy}.

\begin{prop}\label{prop:elasticenergybound}
The elastic energy is bounded from below by
\begin{equation}
 \mathcal{E}_\mathrm{e}(\pf_k,\strain[{\bm u}_k]) \geq \frac{c_\mathbb{C}}{4}\|\bm \varepsilon({\bm u}_k)\|_{\tensSpL}^2 - \frac{c_\mathbb{C}C_{\mathcal{T}}^2}{2}\|\pf_k\|_{L^2(\Omega)}^2.
\end{equation}
\end{prop}
\begin{proof}
Using the Young's inequality $2\bm \varepsilon({\bm u}_k)\!:\!\mathcal{T}(\pf_k)\leq\frac{1}{2}|\bm \varepsilon({\bm u}_k)|^2+2|\mathcal{T}(\pf_k)|^2$ and assumptions (A3) and (A4) the proposition is easily verified
\begin{align*}
\mathcal{E}_\mathrm{e}(\pf_k,\strain[{\bm u}_k]) &= \frac{1}{2}\int_\Omega \left(\bm \varepsilon({\bm u}_k)-\mathcal{T}(\pf_k)\right)\!:\!\mathbb{C}(\pf_k)\left(\bm \varepsilon({\bm u}_k)-\mathcal{T}(\pf_k)\right) \;dx\\
&\geq \frac{c_\mathbb{C}}{2}\int_\Omega |\bm \varepsilon({\bm u}_k)-\mathcal{T}(\pf_k)|^2 \;dx\geq \frac{c_\mathbb{C}}{4}\|\bm \varepsilon({\bm u}_k)\|_{\tensSpL}^2 - \frac{c_\mathbb{C}C_{\mathcal{T}}^2}{2}\|\pf_k\|_{L^2(\Omega)}^2.
\end{align*} 
\end{proof}

\begin{prop}\label{prop:fluidenergybound}
The hydraulic energy is bounded from below by
 \begin{equation}
  \mathcal{E}_{\mathrm{f},\revision{k}}(\strain[{\bm u}_k],\theta_k) \geq C_\theta\|\theta_k\|_{L^2(\Omega)}^2 - \frac{c_\mathbb{C}}{8}\|\bm \varepsilon({\bm u}_k)\|_{\tensSpL}^2,
 \end{equation}
 where $C_\theta:=\frac{M}{2}\left(1-\frac{1}{2\left(\frac{c_\mathbb{C}}{8M\alpha^2d} + \frac{1}{2}\right)}\right)$.
\end{prop}
\begin{proof}
\revision{
By using the linearity of $\Pi_k^\theta$, the fact that $\Pi_k^\theta \theta_k = \theta_k$, and Young's inequality $2ab \leq \dfrac{a^2}{2\delta} + 2\delta b^2$ with $\delta  = \frac{c_\mathbb{C}}{8M\alpha^2d} + \frac{1}{2}$, we get
\begin{align*}
\mathcal{E}_{\mathrm{f},k}(\strain[{\bm u}_k],\theta_k) 
&= \int_\Omega \frac{M}{2} \left( \theta_k-\alpha \Pi_k^\theta \nabla\cdot\bm u_k \right)^2 \; dx 
\geq \frac{M}{2}\int_\Omega \left(1 - \frac{1}{2\delta}\right)\theta_k^2 - \alpha^2 \left(2\delta - 1\right) \left(\Pi_k^\theta \nabla\cdot\bm u_k\right)^2\;dx.
\end{align*}
For the last term on the right hand side, we note that $2\delta - 1>0$ and it holds
\begin{align*}
\int_\Omega \left(\Pi_k^\theta \nabla\cdot\bm u_k\right)^2\;dx \leq \int_\Omega \left(\nabla\cdot\bm u_k\right)^2\;dx \leq \int_\Omega d|\strain[{\bm u}_k]|^2\;dx.
\end{align*}
by definition of $\Pi_k^\theta$ and a standard AM-GM inequality. Thus, we get
\begin{align*}
\mathcal{E}_{\mathrm{f},k}(\strain[{\bm u}_k],\theta_k) \geq
C_\theta\|\theta\|_{L^2(\Omega)}^2 - \frac{c_\mathbb{C}}{8}\|\bm \varepsilon({\bm u}_k)\|_{\tensSpL}^2.
\end{align*}
}
\end{proof}

\revision{
\begin{corollary}\label{cor:freeenergybound}
The free energy is bounded from below by
\begin{equation}
 \mathcal{E}_k(\pf_k,\strain[{\bm u}_k],\theta_k) \geq \frac{\gamma}{\ell}\|\Psi(\pf_k)\|_{L^1(\Omega)} + \frac{\gamma \ell}{2} \|\nabla \pf_k\|_{\vecSpL}^2 - \frac{c_\mathbb{C}C_{\mathcal{T}}^2}{2}\|\pf_k\|_{L^2(\Omega)}^2 + \frac{c_\mathbb{C}}{8}\|\bm \varepsilon({\bm u}_k)\|_{\tensSpL}^2 + C_\theta\|\theta_k\|_{L^2(\Omega)}^2
\end{equation}
with constant $C_\theta$ as defined in Proposition~\ref{prop:fluidenergybound}.
\begin{proof}
For the regularized interface energy we have
\begin{align*}
 \mathcal{E}_\mathrm{i}(\pf_k) = \gamma \int_\Omega \frac{1}{\ell} \Psi(\pf_k) + \frac{\ell}{2} |\nabla\pf_k|^2\;dx = \frac{\gamma}{\ell}\|\Psi(\pf_k)\|_{L^1(\Omega)} + \frac{\gamma \ell}{2} \|\nabla \pf_k\|_{\vecSpL}^2
\end{align*}
employing the non-negativity of $\Psi(\cdot)$, cf.\ assumption (A1). Combining the latter with the energy bounds from Proposition~\ref{prop:elasticenergybound} and \ref{prop:fluidenergybound} then proves the assertion.
\end{proof}
\end{corollary}
}

\begin{lemma}\label{lemma:discrete-apriori}
For the discrete-in-space solution quintuple $(\pf_k, \mu_k, \bm u_k, \theta_k, \bm q_k)$ to \eqref{eq:discspace} we have the \textit{a priori} bound
\begin{equation}
\begin{aligned}
 \underset{t \in [0,T_k]}{\sup} \bigg[\Big(\|\Psi(\pf_k)\|_{L^1(\Omega)} &+ \|\pf_k\|_{H^1(\Omega)}^2 + \|\bm u_k\|_{\vecSpH}^2 + \|\theta_k\|_{L^2(\Omega)}^2\Big)(t)\bigg] \\ 
 % &+ \frac{C_{\nabla\mu}}{C_{\min}} \|\nabla\mu_k\|_{L^2([0,T_k];\vecSpL)}^2 + \frac{C_{\bm q}}{C_{\min}} \|\bm q_k\|_{L^2([0,T_k];\vecSpL)}^2
 &+\regularization \|\partial_t \nabla \cdot \bm{u}_k\|^2_{L^2([0,T];L^2(\Omega))}+\|\nabla\mu_k\|_{L^2([0,T_k];\vecSpL)}^2 + \|\bm q_k\|_{L^2([0,T_k];\vecSpL)}^2
 \leq C_T
\end{aligned}
\label{eq:bound}
\end{equation}
with positive constant $C_T$. The latter is not dependent on $T_k$, allowing us to extend the existence interval such that $T_k = T$ for all $k \in \mathbb{N}$. Furthermore, for the chemical potential the \textit{a priori} estimate
\begin{equation}
 \|\mu_k\|_{L^2\left([0,T];H^1(\Omega)\right)}^2 \leq C
\label{eq:boundmu}
\end{equation}
holds.
\end{lemma}
\begin{proof}
At first, we consider the right-hand side terms of \eqref{eq:addedweakform} for some $K$ to be determined later on. We proceed in chronological order. To start, we introduce the mean of the chemical potential $\mu_k^\Omega = \frac{1}{|\Omega|} \int_\Omega{\mu_k \;dx}$, yielding
\begin{align}
 \big( R,\mu_k \big) &= \big( R,\mu_k-\mu_k^\Omega \big) + \big( R,\mu_k^\Omega \big) \leq \frac{1}{4} \frac{C_P^2}{\lambda} C_R^2 + \lambda \|\nabla\mu_k\|_{\vecSpL}^2 + C_R |\Omega|^{\frac{1}{2}} \big|\mu_k^\Omega\big|.
 \label{eq:firstrhs}
\end{align}
where we used the Cauchy-Schwarz inequality, Young's inequality, the Poincar\'e-Wirtinger inequality with Poincar\'e constant $C_P(\Omega)$ and assumption (A5).
By the Cauchy-Schwarz inequality, Young's inequality and assumption (A5), for the second term we obtain 
\begin{align}
 K\big( R,\pf_k \big) \leq K \|R\|_{L^2(\Omega)} \|\pf_k\|_{L^2(\Omega)} \leq \frac{1}{4}K^2C_R^2 + \|\pf_k\|_{L^2(\Omega)}^2.
 \label{eq:secondrhs}
\end{align}
Next, employing the Cauchy-Schwarz inequality, Young's inequality and (A5) yield
\begin{align*}
 \big( S_f, \delta_\theta\energy_{\mathrm{f},\revision{k}}(\strain[u_k], \theta_k) \big) \leq \|S_f\|_{L^2(\Omega)} \|\delta_\theta\energy_{\mathrm{f},\revision{k}}( \strain[u_k], \theta_k)\|_{L^2(\Omega)} \leq \frac{1}{4} C_S^2 + \|\delta_\theta\energy_{\mathrm{f},\revision{k}}(\strain[u_k], \theta_k)\|_{L^2(\Omega)}^2.
\end{align*}
By \revision{the boundedness of the projection operator $\Pi_k^\theta$,} Minkowski's inequality, Young's inequality and assumption (A2), we get
\begin{align*}
 \|\delta_\theta\energy_{\mathrm{f},\revision{k}}(\strain[u_k], \theta_k)\|_{L^2(\Omega)}^2 &= \|M\revision{\Pi^\theta_k}(\theta_k-\alpha\nabla\cdot\bm u_k)\|_{L^2(\Omega)}^2 \leq M^2 \Big(\|\theta_k\|_{L^2(\Omega)} + \alpha \|\nabla\cdot\bm u_k\|_{L^2(\Omega)}\Big)^2 \nonumber \\
 &\leq 2 M^2 \|\theta_k\|_{L^2(\Omega)}^2 + 2 M^2 \alpha^2 d\|\strain[u_k]\|_{\tensSpL}^2,
 % \label{eq:pressurebound}
\end{align*}
and, therefore, the bound
\begin{align}
 \big( S_f, \delta_\theta\energy_{\mathrm{f},\revision{k}}(\strain[u_k], \theta_k) \big)\leq \frac{1}{4} C_S^2 + 2 M^2 \|\theta_k\|_{L^2(\Omega)}^2 + 2 M^2 \alpha^2 d\|\strain[u_k]\|_{\tensSpL}^2.
 \label{eq:thirdrhs}
\end{align}
For the fourth term we employ Hölder's inequality, Young's inequality and assumption (A2) to get
\begin{align}
 -K\big( m(\pf_k)\nabla\mu_k,\nabla\pf_k \big) \leq K\|\nabla\mu_k\|_{\vecSpL} \|m(\pf_k)\nabla\pf_k\|_{\vecSpL} \leq \lambda \|\nabla\mu_k\|^2_{\vecSpL} + \frac{K^2 C_m^2}{4\lambda} \|\nabla\pf_k\|_{\vecSpL}^2.
 \label{eq:forthrhs}
\end{align}
For the last term we have $\langle \bm f,\partial_t\bm u_k\rangle = \partial_t \langle \bm f,\bm u_k \rangle$.
For the left-hand side terms of \eqref{eq:addedweakform}, using the lower bound for $m(\cdot)$ and the upper bound for $\kappa(\cdot)$, cf.\ assumption (A2), we get
\begin{align*}
 \|m(\pf_k)^{\frac{1}{2}}\nabla\mu_k\|_{\vecSpL}^2 \geq c_m \|\nabla\mu_k\|_{\vecSpL}^2,
\qquad
\textrm{and}
\qquad
 \|\kappa(\pf_k)^{-\frac{1}{2}}\bm q_k\|_{\vecSpL}^2 \geq \frac{1}{C_\kappa} \|\bm q_k\|_{\vecSpL}^2.
\end{align*}
Thus, by inserting the estimates \eqref{eq:firstrhs} through \eqref{eq:forthrhs} into \eqref{eq:addedweakform}, together with collecting the constant terms in $C_1$ we obtain
\begin{align}
 &\frac{d}{dt}\left(\mathcal{E}_{\revision{k}}(\pf_k, \strain[u_k], \theta_k) + K\frac{\|\pf_k\|_{L^2(\Omega)}^2}{2} - \langle \bm f,\bm u_k \rangle\right) 
 \nonumber\\
 &\qquad+ \regularization \|\partial_t \nabla \cdot \bm{u}_k \|_{L^2(\Omega)}^2 + \big(c_m - 2\lambda\big) \|\nabla\mu_k\|_{\vecSpL}^2 + \frac{1}{C_\kappa} \|\bm q_k\|_{\vecSpL}^2 
 \label{eq:estimatewithmean} \\
 &\hspace{0.225cm}\leq C_R |\Omega|^{\frac{1}{2}} |\mu_k^\Omega| + \|\pf_k\|_{L^2(\Omega)}^2 + \frac{K^2 C_m^2}{4\lambda} \|\nabla\pf_k\|_{\vecSpL}^2 + 2 M^2 \alpha^2 d\|\strain[u_k]\|_{\tensSpL}^2 + 2 M^2 \|\theta_k\|_{L^2(\Omega)}^2 + C_1.
 \nonumber
\end{align}
Next, we estimate the first right-hand side term of \eqref{eq:estimatewithmean}. Note that by the definition of the mean $\mu_k^\Omega$ and the first basis function $\eta_1^\textrm{ch} = \frac{1}{|\Omega|^{\frac{1}{2}}}$ we have
\begin{align*}
 |\Omega|^{\frac{1}{2}} |\mu_k^\Omega| = \frac{1}{|\Omega|^{\frac{1}{2}}} \left|\int_\Omega \mu_k \;dx\right| = \left|\int_\Omega \mu_k \frac{1}{|\Omega|^{\frac{1}{2}}} \;dx\right| = \left|\big( \mu_k,\eta_1^\textrm{ch} \big)\right|.
\end{align*}
By considering \eqref{eq:ch2discspace} for $\eta_1^{ch}$ with the substitution \eqref{eq:ch2discspacealter}, we obtain
\begin{align*}
 \left|\big( \mu_k,\eta_1^\textrm{ch} \big)\right| &= \frac{1}{|\Omega|^{\frac{1}{2}}} \left|\frac{\gamma}{\ell}\int_\Omega \Psi'(\pf_k) \;dx + \int_\Omega\delta_\varphi\mathcal{E}_\mathrm{e}(\varphi_k, \strain[u_k]) \;dx \right| \\
 &\leq \frac{1}{|\Omega|^{\frac{1}{2}}} \left(\frac{\gamma}{\ell}\int_\Omega |\Psi'(\pf_k)| \;dx + \int_\Omega \Big|\frac{1}{2}\left({\bm \varepsilon}(\bm u_k) - \mathcal{T}(\varphi_k)\right)\!:\!\mathbb{C}'(\varphi_k)\left({\bm \varepsilon}(\bm u_k) - \mathcal{T}(\varphi_k)\right) \Big| \;dx \right. \\
 &\hspace{2cm} + \left. \int_\Omega \big| \mathcal{T}'(\varphi_k)\!:\!\mathbb{C}(\pf_k)\left({\bm \varepsilon}(\bm u_k) - \mathcal{T}(\varphi_k)\right)\big| \;dx \right).
\end{align*}
We estimate the integral terms separately. Firstly, by assumption (A1) with \eqref{eq:boundpsiprime} we have
\begin{align*}
 \int_\Omega |\Psi'(\pf_k)| \;dx = \|\Psi'(\pf_k)\|_{L^1(\Omega)} \leq C_{\Psi}\left(\|\Psi(\pf_k)\|_{L^1(\Omega)} + \|\pf_k\|_{L^2(\Omega)}^2\right).
\end{align*}
Since $\mathbb{C}(\varphi)$ is Lipschitz continuous and differentiable (assumption (A3)), we can infer that $\mathbb{C}'(\varphi)$ is uniformly bounded by the Lipschitz constant $L_\mathbb{C}$. \revision{By assumption A4,} the same also holds for $\mathcal{T}(\pf)$ with respective Lipschitz constant $L_\mathcal{T}$. Thus, we get 
\begin{align*}
 \int_\Omega \Big|\frac{1}{2}\left({\bm \varepsilon}(\bm u_k) - \mathcal{T}(\varphi_k)\right)\!:\!\mathbb{C}'(\varphi_k)\left({\bm \varepsilon}(\bm u_k) - \mathcal{T}(\varphi_k)\right) \Big| \;dx &\leq \frac{1}{2} L_\mathbb{C} \|\strain[u_k] - \mathcal{T}(\pf_k)\|_{\tensSpL}^2 \\
 &\leq L_\mathbb{C} \|\strain[u_k]\|_{\tensSpL}^2 + L_\mathbb{C} C_{\mathcal{T}}^2 \|\pf_k\|_{L^2(\Omega)}^2,
\end{align*}
where we used Minkowski's inequality, Young's inequality and assumptions (A3) and (A4).
Then, by Hölder's inequality, assumptions (A3) and (A4), Young's inequality, and similar steps as above we obtain
\begin{align*}
 \int_\Omega \big|\mathcal{T}'(\varphi_k)\!:\!\mathbb{C}(\pf_k)\left({\bm \varepsilon}(\bm u_k) - \mathcal{T}(\varphi_k)\right)\big| \;dx &\leq L_\mathcal{T} C_\mathbb{C} \|\strain[u_k] - \mathcal{T}(\pf_k)\|_{\tensSpL} |\Omega|^{1/2} \\
 &\leq \|\strain[u_k]\|_{\tensSpL}^2 + C_{\mathcal{T}}^2\|\pf_k\|_{L^2(\Omega)}^2 + \underbrace{\frac{1}{2}L_\mathcal{T}^2 C_\mathbb{C}^2 |\Omega|}_{= C(L_\mathcal{T}, C_\mathbb{C}, |\Omega|) =: C_2}.
\end{align*}
Thus, we obtain the estimate
\begin{align}
 \left|\left( \mu_k,\eta_1^{ch} \right)\right| &\leq \frac{C_\Psi}{|\Omega|^{\frac{1}{2}}}\frac{\gamma}{\ell} \|\Psi(\pf_k)\|_{L^1(\Omega)} + \frac{1}{|\Omega|^{\frac{1}{2}}}\Big(\frac{\gamma}{\ell}C_\Psi + L_\mathbb{C} C_{\mathcal{T}}^2 + C_{\mathcal{T}}^2\Big) \|\pf_k\|_{L^2(\Omega)}^2 \nonumber \\
 &\hspace{1cm}+ \frac{1}{|\Omega|^{\frac{1}{2}}} \left(L_\mathbb{C} + 1\right) \|\strain[u_k]\|_{\tensSpL}^2 + C_2.
\label{eq:mubound}
\end{align}
After multiplying with $C_R$, we insert \eqref{eq:mubound} into \eqref{eq:estimatewithmean}, yielding
\begin{align*}
 \frac{d}{dt}&\left(\mathcal{E}_{\revision{k}}(\pf_k, \strain[u_k], \theta_k) + K\frac{\|\pf_k\|_{L^2(\Omega)}^2}{2} - \langle \bm f,\bm u_k \rangle\right) \\
 &\hspace{1cm}+ \regularization \|\partial_t \nabla \cdot \bm{u}_k \|_{L^2(\Omega)}^2+ \big(c_m - 2\lambda\big) \|\nabla\mu_k\|_{\vecSpL}^2 + \frac{1}{C_\kappa} \|\bm q_k\|_{\vecSpL}^2 \\
 &\leq C_R \frac{C_\Psi}{|\Omega|^{\frac{1}{2}}}\frac{\gamma}{\ell} \|\Psi(\pf_k)\|_{L^1(\Omega)} + \left[\frac{C_R}{|\Omega|^{\frac{1}{2}}}\Big(\frac{\gamma}{\ell}C_\Psi + L_\mathbb{C} C_{\mathcal{T}}^2 + C_{\mathcal{T}}^2\Big) + 1\right] \|\pf_k\|_{L^2(\Omega)}^2 + \frac{K^2 C_m^2}{4\lambda} \|\nabla\pf_k\|_{\vecSpL}^2 \\
 &\hspace{1cm}+ \left[\frac{C_R}{|\Omega|^{\frac{1}{2}}} \left(L_\mathbb{C} + 1 \right) + 2 M^2 \alpha^2 d\right] \|\strain[u_k]\|_{\tensSpL}^2+ 2 M^2 \|\theta_k\|_{L^2(\Omega)}^2 + C_R C_2 + C_1.
\end{align*}
We now choose $\lambda = \frac{c_m}{4}$ and define the constants
\begin{align*}
 &C_{\nabla\mu} := c_m - 2\lambda = \frac{c_m}{2}, \quad
 C_{\bm q} := \frac{1}{C_\kappa}, \\ 
 &\zeta_\Psi := C_R \frac{C_\Psi}{|\Omega|^{\frac{1}{2}}}\frac{\gamma}{\ell}, \quad 
 \zeta_\pf := \max\left\{\frac{C_R}{|\Omega|^{\frac{1}{2}}}\Big(\frac{\gamma}{\ell}C_\Psi + L_\mathbb{C} C_{\mathcal{T}}^2 + C_{\mathcal{T}}^2\Big) + 1, \frac{K^2 C_m^2}{4\lambda} \right\}, \\
 &\zeta_{\bm \varepsilon} := \frac{C_R}{|\Omega|^{\frac{1}{2}}} \left(L_\mathbb{C} + 1\right) + 2 M^2 \alpha^2 d , \quad 
 \zeta_\theta := 2 M^2,\\
 &\zeta := C_R C_2 + C_1 = C(C_R, L_\mathcal{T}, C_\mathbb{C}, |\Omega|, C_p, \lambda^{-1}, K, C_S).
\end{align*}
Integrating in time (from 0 to $t \in (0,T_k]$) along with applying the fundamental theorem of calculus yields
\begin{align*}
 \bigg(&\mathcal{E}_{\revision{k}}(\pf_k, \strain[u_k], \theta_k) + K\frac{\|\pf_k\|_{L^2(\Omega)}^2}{2} - \langle \bm f,\bm u_k \rangle\bigg)(t)\\
 &\qquad+ \regularization \|\partial_t \nabla \cdot \bm{u}_k \|_{L^2([0,t]; L^2(\Omega))}^2 + C_{\nabla\mu} \|\nabla\mu_k\|_{L^2([0,t];\vecSpL)}^2 + C_{\bm q} \|\bm q_k\|_{L^2([0,t];\vecSpL)}^2 \\
 &\leq \zeta_\Psi \|\Psi(\pf_k)\|_{L^1([0,t];L^1(\Omega))} + \zeta_\pf \|\pf_k\|_{L^2([0,t];H^1(\Omega))}^2 + \zeta_{\bm \varepsilon} \|\strain[u_k]\|_{L^2([0,t];\tensSpL)}^2 + \zeta_\theta \|\theta_k\|_{L^2([0,t];L^2(\Omega))}^2 \\
 &\qquad+ \zeta t + \bigg(\mathcal{E}_{\revision{k}}(\pf_k, \strain[u_k], \theta_k) + K\frac{\|\pf_k\|_{L^2(\Omega)}^2}{2} - \langle \bm f,\bm u_k \rangle\bigg)(0).
\end{align*}
Using the definition of the duality pairing, Young's inequality and Korn's inequality for time $t$, as well as assumption (A5) and Lemma~\ref{lemma:initial-displacement-bound} for time $t=0$, we further have
\begin{align*}
 \langle \bm f,\bm u_k \rangle(t) &\leq \|\bm f\|_{\vecSpHM} \|\bm u_k(t)\|_{\vecSpHTr} \leq \frac{1}{4 \delta} \|\bm f\|_{\vecSpHM}^2 + \delta\|\bm u_k(t)\|_{\vecSpH}^2 \leq \frac{1}{4 \delta} C_{\bm f}^2 + \delta C_K^2 \|\bm \varepsilon(\bm u_k(t))\|_{\tensSpL}^2, \\
 \langle \bm f, \bm u_{k,0} \rangle &\leq \|\bm f\|_{\vecSpHM} \|\bm u_{k,0}\|_{\vecSpHTr} \leq C_{\bm f} C_K \zeta_{\bm \varepsilon, 0} \big(\|\pf_{k,0}\|_{L^2(\Omega)} + \|\theta_{k,0}\|_{L^2(\Omega)} + 1\big)
\end{align*}
with $\delta > 0$ yet to be determined. For (later) fixed $\delta$, from Lemma~\ref{lemma:initial-data}, as well as assumption (A5), we can deduce a constant $\zeta_0$, in particular depending on $C_{\mathcal{E},0}$, satisfying
\begin{align*}
    \frac{1}{4 \delta} C_{\bm f}^2 + \bigg(\mathcal{E}_{\revision{k}}(\pf_k, \strain[u_k], \theta_k) + K\frac{\|\pf_k\|_{L^2(\Omega)}^2}{2} - \langle \bm f,\bm u_k \rangle\bigg)(0) \leq \zeta_0,
\end{align*}
We now choose $K = c_\mathbb{C}C_{\mathcal{T}}^2 + \gamma\ell$ and define the constants
\begin{align*}
 &C_\pf := \frac{\gamma\ell}{2}, \qquad C_{\bm \varepsilon} := \frac{c_{\mathbb{C}}}{8} - \delta C_K^2. 
 % \qquad \zeta_{\pf,0} := \max\left\{\frac{c_\mathbb{C}C_{\mathcal{T}}^2 + \gamma\ell}{2},\frac{\gamma\ell}{2}\right\} = \frac{c_\mathbb{C}C_{\mathcal{T}}^2}{2} + \frac{\gamma\ell}{2}.
\end{align*}
Now choosing $\delta < \frac{c_{\mathbb{C}}}{8C_K^2}$ yields $C_{\bm \varepsilon} > 0$.
Using on the left hand side the energy bound from \revision{Corollary~\ref{cor:freeenergybound}}, we obtain
\begin{align*}
 &\Big(\frac{\gamma}{\ell}\|\Psi(\pf_k)\|_{L^1(\Omega)} + C_\pf \|\pf_k\|_{H^1(\Omega)}^2 + C_{\bm \varepsilon}\|\bm \varepsilon(\bm u_k)\|_{\tensSpL}^2 + C_\theta \|\theta_k\|_{L^2(\Omega)}^2\Big)(t) \\ 
 &\hspace{0.5cm}+ \regularization \|\partial_t \nabla \cdot \bm{u}_k \|_{L^2([0,t]; L^2(\Omega))}^2+ C_{\nabla\mu} \|\nabla\mu_k\|_{L^2([0,t];\vecSpL)}^2 + C_{\bm q} \|\bm q_k\|_{L^2([0,t];\vecSpL)}^2 \\
 &\leq \zeta_\Psi \|\Psi(\pf_k)\|_{L^1([0,t];L^1(\Omega))} + \zeta_\pf \|\pf_k\|_{L^2([0,t];H^1(\Omega))}^2 + \zeta_{\bm \varepsilon} \|\strain[u_k]\|_{L^2([0,t];\tensSpL)}^2 + \zeta_\theta \|\theta_k\|_{L^2([0,t];L^2(\Omega))}^2 + \zeta t + \zeta_0.
\end{align*}
Employing Grönwall's inequality, cf.\ Lemma~\ref{lemma:Grönwall}, and Korn's inequality, we finally obtain
\begin{align*}
 \underset{t \in [0,T_k]}{\sup} \bigg[\Big(\|\Psi(\pf_k)\|_{L^1(\Omega)} &+ \|\pf_k\|_{H^1(\Omega)}^2 + \|\bm u_k\|_{\vecSpH}^2 + \|\theta_k\|_{L^2(\Omega)}^2\Big)(t)\bigg] \\
 &+ \regularization \|\partial_t \nabla \cdot \bm{u}_k \|_{L^2([0,T_k]; L^2(\Omega))}^2+ \|\nabla\mu_k\|_{L^2([0,T_k];\vecSpL)}^2 + \|\bm q_k\|_{L^2([0,T_k];\vecSpL)}^2
 \leq C_T,
\end{align*}
for a positive constant $C_T$ merely depending on material parameters and data. 
The constant $C_T$ on the right-hand side of the above estimate is not dependent on $T_k$, allowing us to extend the existence interval such that $T_k = T$ for all $k \in \mathbb{N}$ (cf.~\cite{Zeidler1986}, §3.3.\hspace{0.06cm}f.; \cite{Walter1998}, II.§7).
By $\|\bm \varepsilon(\bm u_k)\|_{\tensSpL}^2 \leq \|\bm u_k\|_{\vecSpH}^2$ we further get the uniform boundedness of the strain.

Moreover, for the chemical potential we have 
\begin{align*}
\|\mu_k\|_{H^1(\Omega)} \leq \|\mu_k - \mu_k^\Omega\|_{L^2(\Omega)} + \|\mu_k^\Omega\|_{L^2(\Omega)} + \|\nabla\mu_k\|_{\vecSpL} \leq \big(C_P + 1\big) \|\nabla\mu_k\|_{\vecSpL} + |\Omega|\left|\mu_k^\Omega\right|,
\end{align*}
where we used Minkowski's inequality and the Poincar\'e-Wirtinger inequality, and by \eqref{eq:bound} and \eqref{eq:mubound} we obtain
\begin{equation*}
 \|\mu_k\|_{L^2\left([0,T];H^1(\Omega)\right)}^2 \leq C.
\end{equation*}
\end{proof}

We conclude with \textit{a priori} estimates for the time derivatives.
\begin{lemma}\label{lemma:a-priori-dt}
Let $(\pf_k, \theta_k)$ be part of the discrete-in-space solution of~\eqref{eq:discspace}. For the time derivatives of $\pf_k$ and $\theta_k$ we have the \textit{a priori} estimates
\begin{align}
\label{eq:boundphi-theta}
 \|\partial_t \pf_k\|_{L^2([0,T];H^1(\Omega)')}^2 +  \|\partial_t \theta_k\|_{L^2([0,T];H^1(\Omega)')}^2 \leq C.
\end{align}
\end{lemma}
\begin{proof}
The result follows almost immediately from previously established \textit{a priori} results, cf.\ Lemma~\ref{lemma:discrete-apriori}, and the definition of the discrete-in-space solution~\eqref{eq:discspace}. In order to utilize these, we apply the projections of $H^1(\Omega)$ onto the respective test spaces $\mathcal{V}_k^\mathrm{ch}$ and $\mathcal{V}_k^\theta$.
Let $g\in H^1(\Omega)$. Then, we have 
\begin{align*}
 \langle\partial_t\pf_k,g\rangle &= \Big\langle \partial_t \sum_{j=1}^k a_j^k(t)\eta^\mathrm{ch}_j,g \Big\rangle = \sum_{j=1}^k \big( \eta^\mathrm{ch}_j,g \big) \left(a_j^k\right)'(t) \overset{\eqref{eq:helperprojectionofphi}}{=} \sum_{j=1}^k \big( \eta^\mathrm{ch}_j,g \big) \langle \partial_t \pf_k, \eta^\mathrm{ch}_j \rangle 
 % = \Big\langle \partial_t \pf_k, \sum_{j=1}^k \big( g,\eta^\mathrm{ch}_j \big) \eta^\mathrm{ch}_j \Big\rangle \\
 % &
 = \langle \partial_t \pf_k,\Pi_k^\mathrm{ch} g \rangle.
\end{align*}
Multiplying equation \eqref{eq:ch1discspace} by $\big( g, \eta^\mathrm{ch}_j \big)$, summing over $j=1,\ldots,k$ and taking the absolute value yields 
\begin{align*}
 \big|\langle \partial_t\pf_k,\Pi_k^\mathrm{ch} g\rangle\big| &= \big|-\big( m\nabla \mu_k,\nabla\Pi_k^\mathrm{ch} g\big) + \big( R,\Pi_k^\mathrm{ch} g\big)\big| \leq C_m \|\nabla\mu_k\|_{\vecSpL} \|\nabla\Pi_k^\mathrm{ch} g\|_{\vecSpL} + C_R \|\Pi_k^\mathrm{ch} g\|_{L^2(\Omega)} \\
 & \leq C_\Pi \max\{C_m, C_R\}\left(\|\nabla\mu_k\|_{\vecSpL} + 1 \right)\| g\|_{H^1(\Omega)},
\end{align*}
where we used the fact that $\|\Pi_k^\mathrm{ch} g\|_{H^1(\Omega)} \leq C_\Pi \|g\|_{H^1(\Omega)}$. Thus, by Lemma~\ref{lemma:discrete-apriori}, we obtain for some constant $C>0$ independent of $k$
\begin{align*}
 \|\partial_t \pf_k\|_{L^2([0,T];H^1(\Omega)')}^2 = \int_0^{T} \|\partial_t \pf_k\|_{H^1(\Omega)'}^2 \;dt =\int_0^T \underset{g \in H^1(\Omega)}{\sup} \frac{\big|\langle\partial_t\pf_k,g\rangle\big|}{\|g\|_{H^1(\Omega)}} \;dt \leq C.
\end{align*}
We can make a similar calculation based on equation \eqref{eq:flowdiscspace} to get
\begin{align*}
 \big|\langle \partial_t\theta_k,g \rangle\big| &= \big|\langle \partial_t\theta_k, \Pi_k^\theta g \rangle\big| = \big|- \big( \nabla\cdot \bm{q}_k,\Pi_k^\theta g \big) + \big( S_f,\Pi_k^\theta g \big)\big| \\
 &\leq \|\bm{q}_k\|_{\vecSpL} \|\nabla\Pi_k^\theta g\|_{\vecSpL} + C_S \|\Pi_k^\theta g\|_{L^2(\Omega)} \leq C_\Pi \max\{1, C_S\} \left(\|\bm q_k\|_{\vecSpL} + 1\right)\|g\|_{H^1(\Omega)}.
\end{align*}
Hence, we obtain in similar fashion a constant $C>0$ independent of $k$ satisfying 
\begin{align*}
 \|\partial_t \theta_k\|_{L^2([0,T];H^1(\Omega)')}^2 = \int_0^{T} \|\partial_t \theta_k\|_{H^1(\Omega)'}^2\;dt \leq C.
\end{align*}
\end{proof}

As an immediate consequence of Lemma~\ref{lemma:discrete-apriori} and Lemma~\ref{lemma:a-priori-dt} we obtain the following uniform bounds.

\begin{subequations}
\label{eq:uniformBounds}

\begin{corollary}[Uniform bounds] \label{cor:uniform_bounds}
For the discrete-in-space solution $(\pf_k, \mu_k, \bm u_k, \theta_k, \bm q_k)$ to \eqref{eq:discspace} we have that
\begin{align}
 &\pf_k \textrm{ is uniformly bounded in } L^\infty\big([0,T];H^1(\Omega)\big)\cap H^1\big([0,T];H^1(\Omega)'\big), \\
 &\mu_k \textrm{ is uniformly bounded in } L^2\big([0,T];H^1(\Omega)\big) ,\\
 &\bm u_k \textrm{ is uniformly bounded in } L^\infty\big([0,T];\vecSpHTr\big), \\
 &\theta_k \textrm{ is uniformly bounded in } L^\infty\big([0,T];L^2(\Omega)\big)\cap H^1\big([0,T];H^1(\Omega)'\big), \\
 &\bm q_k \textrm{ is uniformly bounded in } L^2\big([0,T];\vecSpL\big).
\end{align}
\end{corollary}

The above corollary holds without (A7), under which we, however, can derive even stronger regularity.

\begin{corollary}[Uniform bounds II] \label{cor:uniform_bounds-II} For fixed $\eta > 0$, $\bm u_k$, part of the solution to \eqref{eq:discspace}, satisfies
\begin{align}
\nabla \cdot \bm{u}_k\textrm{ is uniformly bounded in } H^1\big([0,T];L^2(\Omega)\big).
\end{align}
\end{corollary}

\end{subequations}

\subsection{Compactness arguments and passing to the limit} 
\label{sec:passingtolimit}
We will now use standard compactness arguments to see that the solution sequence to the discrete-in-space problem \eqref{eq:discspace} converges weakly (up to subsequence). Then, we will prove that the resulting limit quintuple solves the system of equations in the sense of Definition~\ref{def:Weak solution} under the constraint of $\pf_0 \in H_{\bm n}^2(\Omega)$, cf.\ Remark~\ref{rem:assumption_initial_pf}. We further show the energy estimate~\eqref{eq:boundenergy} and complete the proof for general initial conditions, cf.\ assumption (A6).

An important aspect of the remaining existence analysis is 
% the switch from the milder assumption (A2$^\star$) to the stricter assumption (A2); in addition
that, from here on, we explicitly require (A7) to hold, allowing to exploit the regularizing effect of secondary consolidation, while the above calculations (and still the next corollary) also hold without 
% these assumptions.
this assumption.

\begin{corollary}[Weak convergence]\label{cor:weak_convergence}
In the limit $k \rightarrow \infty$, there exists a limit quintuple $(\pf, \mu, \bm u, \theta, \bm q)$ to the solution sequence $(\pf_k, \mu_k, \bm u_k, \theta_k, \bm q_k)$ to \eqref{eq:discspace} and the following weak convergence properties hold (up to subsequence)
\begin{subequations}
\begin{align}
 &\pf_k \rightarrow \pf \textrm{ weakly-$\ast$ in } L^\infty\big([0,T];H^1(\Omega)\big),
 \label{eq:weakConvergencePhi} \\
  &\pf_k \rightarrow \pf \textrm{ weakly in } H^1\big([0,T];H^1(\Omega)'\big),
 \label{eq:weakConvergencePhi-h1} \\
 &\mu_k \rightarrow \mu \textrm{ weakly in } L^2\big([0,T];H^1(\Omega)\big), \\
 &\bm u_k \rightarrow \bm u \textrm{ weakly-$\ast$ in } L^\infty\big([0,T];\vecSpHTr\big), 
 \label{eq:weakConvergenceDisp} \\
 &\theta_k \rightarrow \theta \textrm{ weakly-$\ast$ in } L^\infty\big([0,T];L^2(\Omega)\big), \label{eq:weakConvergenceTheta} \\
 &\theta_k \rightarrow \theta \textrm{ weakly in } H^1\big([0,T];H^1(\Omega)'\big),
 \label{eq:weakConvergenceTheta-h1} \\
 &\bm q_k \rightarrow \bm q \textrm{ weakly in } L^2\big([0,T];\vecSpL\big).
\end{align}
\label{eq:weakConvergence}%
\end{subequations}
\end{corollary}
\begin{proof}
    This is a direct consequence of Corollary~\ref{cor:uniform_bounds} and standard compactness results in the form of the Banach-Alaoglu theorem (\cite{Conway1990}, p.\hspace{0.05cm}130\hspace{0.06cm}ff.)~and the Eberlein-\v{S}mulian theorem (\cite{Conway1990}, p.\hspace{0.05cm}163).
\end{proof}

Similarly, under assumption (A7), from Corollary~\ref{cor:uniform_bounds-II} it follows:
\begin{corollary}[Weak convergence II]\label{cor:weak_convergence-II}
In the limit $k \rightarrow \infty$, (up to a subsequence) it holds 
\begin{align}
 &\partial_t \nabla \cdot \bm u_k \rightarrow \partial_t \nabla \cdot \bm u \textrm{ weakly in } L^2\big([0,T];L^2(\Omega)\big). 
 \label{eq:weakConvergenceDisp-dt}
\end{align}
\end{corollary}

To discuss the model equations in the limit, we furthermore require strong convergence properties.
\begin{lemma}\label{lemma:strong-convergence-phi}
There exist subsequences of $\left\{\pf_k\right\}_k$ with the following convergence properties
\begin{align}
 &\pf_k \rightarrow \pf \textrm{ strongly in } C\big([0,T];L^2(\Omega)\big) \textrm{ and a.e.~in } \Omega \times [0,T].
 \label{eq:convphi}
\end{align}
\end{lemma}

\begin{proof}
The result follows from application of Corollary~\ref{cor:weak_convergence} and the Aubin-Lions compactness lemma (\cite{Simon1986}, Corollary 4; \cite{Boyer2013}, Theorem II.5.16).
\end{proof}
% \begin{proof}
% By application of the Aubin-Lions compactness lemma (\cite{Simon1986}, Corollary 4; \cite{Boyer2013}, Theorem II.5.16) to \eqref{eq:weakConvergencePhi} and \eqref{eq:weakConvergencePhi-h1} we obtain
% \begin{align*}
%  \pf_k \rightarrow \pf \textrm{ strongly in } C\big([0,T];L^2(\Omega)\big),
% \end{align*}
% Due to the strong convergence in $C\big([0,T];L^2(\Omega)\big)$, there is a subsequence, indexed again by k, such that
% \begin{align*}
%  \pf_k(x,t) \rightarrow \pf(x,t) \textrm{ for a.e.~} (x,t) \in \Omega \times [0,T].
% \end{align*}
% \end{proof}

\begin{lemma}\label{lemma:strong-convergence-theta}
There exists a subsequence of $\left\{\theta_k\right\}_k$ with the following convergence property
\begin{align}
 &\theta_k \rightarrow \theta \textrm{ strongly in } C\big([0,T];H^1(\Omega)'\big).
 \label{eq:convtheta}
\end{align}
\end{lemma}
\begin{proof} 
Analogous to Lemma~\ref{lemma:strong-convergence-phi}.
\end{proof}

\revision{
\begin{lemma}\label{lemma:strong-convergence-pk}
For the projected discrete pore pressure $\pi_k$, defined in~\eqref{eq:porepressure_k}, we have (up to a subsequence)
\begin{align}
 \pi_k \rightarrow p = M(\theta - \alpha \nabla \cdot \bm u )\text{ strongly in }L^2([0,T]; L^2(\Omega))\text{ and a.e.\ in }\Omega \times [0,T].
\end{align}
\end{lemma}
}

\begin{proof}

\revision{
To show strong convergence, the goal is to employ the Aubin-Lions compactness lemma.
By definition of $\Pi_k^\theta$ it holds for all fixed $i\geq 1$ that
\begin{align*}
    \underset{k\rightarrow \infty}{\mathrm{lim}}\int_0^T \left( \pi_k, \eta_i^{\theta} \right)\; dt =
    \underset{k\rightarrow \infty}{\mathrm{lim}}\int_0^T \left( p_k, \eta_i^{\theta} \right)\; dt 
    = \int_0^T \left(p, \eta_i^{\theta} \right) \; dt
\end{align*}
for $p_k = M(\theta_k - \alpha \nabla \cdot \bm u_k)$, following from Corollary~\ref{cor:weak_convergence}. And since $\Span\{\eta_i^{\theta}\}_{i \in \mathbb{N}}$ is dense in $L^2(\Omega)$, it follows
}
\begin{align}
\label{eq:pi_k_weak_convergence}
 \pi_k \rightarrow p\text{ weakly in }L^2([0,T]; L^2(\Omega)). 
\end{align}
% \partial_t \pi_k
After applying a triangle inequality, Corollary~\ref{cor:uniform_bounds} and Corollary~\ref{cor:uniform_bounds-II} yield a uniform bound for $\partial_t  \pi_k$
\begin{align}
\nonumber
\| \pi_k \|_{H^1([0,T]; H^1(\Omega)')}^2
&=
 \left\| M(\theta_k - \alpha \Pi_k^\theta \nabla \cdot \bm u_k) \right\|_{H^1([0,T]; H^1(\Omega)')}^2 \\
 \label{eq:partial_t_pi_k}
 &\leq 
 \revision{2}M^2 \left\|\theta_k \right\|_{H^1([0,T]; H^1(\Omega)')}^2 + \revision{2}(M\alpha)^2 \left\| \nabla \cdot \bm u_k \right\|_{H^1([0,T]; L^2(\Omega))}^2,
\end{align}
where we used the linearity of $\Pi_k^\theta$ as well as the continuity property $\| \Pi_k^\theta \phi \|_{L^2(\Omega)} \leq \| \phi \|_{L^2(\Omega)}$ for any $\phi \in L^2(\Omega)$.
% \pi_k
Utilizing the definition of the projection $\Pi_k^\theta$, we obtain
\begin{align*}
 \| \pi_k \|_{L^2([0,T]; L^2(\Omega))}^2 = \int_0^T \left( \pi_k, p_k\right) \; dt,
\end{align*}
yielding the uniform bound utilizing the Cauchy-Schwarz inequality and  the weak convergence of $\{p_k\}_k$
\begin{align}
\label{eq:pi_k}
 \| \pi_k \|_{L^2([0,T]; L^2(\Omega))} \leq \| p_k \|_{L^2([0,T]; L^2(\Omega))}.
\end{align}
% \nabla \pi_k
Employing the fact that $\pi_k \in H^2_{\bm{n}}(\Omega)$ and $\nabla \pi_k \in \mathcal{V}_k^{\bm q}$, cf.\ Lemma~\ref{lemma:flux-basis}, we obtain by employing~\eqref{eq:darcyflowdiscspace}
\begin{align*}
 \| \nabla \pi_k \|_{L^2([0,T]; L^2(\Omega))}^2 = -\int_0^T \left(\pi_k, \nabla \cdot \nabla \pi_k\right) \; dt = -\int_0^T \left(\kappa(\pf_k)^{-1} \bm{q}_k, \nabla \pi_k \right) \; dt.
\end{align*}
Thus, the Cauchy-Schwarz inequality yields
\begin{align}
\label{eq:grad_pi_k}
 \| \nabla \pi_k \|_{L^2([0,T]; L^2(\Omega))} \leq \left\|\kappa(\pf_k)^{-1} \bm{q}_k \right\|_{L^2([0,T]; L^2(\Omega))}.
\end{align}
Overall, combining~\eqref{eq:partial_t_pi_k}--\eqref{eq:grad_pi_k} and utilizing uniform bounds, cf.\ Corollary~\ref{cor:uniform_bounds}, Corollary~\ref{cor:uniform_bounds-II} and~\eqref{eq:pi_k_weak_convergence}, 
\begin{align*}
\pi_k\text{ is uniformly bounded in }H^1\big([0,T]; H^1(\Omega)'\big) \cap L^2\big([0,T]; H^1(\Omega)\big).
\end{align*}
Consequently, the Aubin-Lions compactness lemma allows to deduce
\begin{align*}
\pi_k \rightarrow p\text{ strongly in }L^2\big([0,T]; L^2(\Omega)\big),
\end{align*}
where the limit coincides with $p$, following from~\eqref{eq:pi_k_weak_convergence}.
\end{proof}

The last lemma is crucial to deduce strong convergence for the displacement approximations, for which we require convergence of initial data.

\begin{lemma}\label{lemma:convergence-initial-data}
The discrete approximations of the initial data, cf.\ Section~\ref{sec:disc-in-space}, satisfy
\begin{alignat*}{2}
    \pf_{k,0} &\rightarrow \pf_0 &\ &\text{ in }L^2(\Omega),\\
    \theta_{k,0} &\rightarrow \theta_0 &\ &\text{ in }L^2(\Omega),\\
    \bm{u}_{k,0} &\rightarrow \bm{u}_0 &\ &\text{ in }\vecSpHTr.
\end{alignat*}
\end{lemma}
\begin{proof}
The strong convergence of $\pf_k$ to $\pf$ in $C\big([0,T];L^2(\Omega)\big)$, cf.\ Lemma~\ref{lemma:strong-convergence-phi} implies $\pf_k(0) \rightarrow \pf(0)$ in $L^2(\Omega)$ for $k \rightarrow \infty$. Moreover, we obtain the strong convergence of $\pf_{k,0} = \Pi_k^\mathrm{ch} \pf_0$ to $\pf_0$ in $L^2(\Omega)$ (\cite{Brezis2011}, Corollary 5.10), that is $\pf_k(0) \rightarrow \pf_0$ in $L^2(\Omega)$ for $k \rightarrow \infty$. Since the limit is unique, we get the meaningful initial condition $\pf(0) = \pf_0$. For $\theta$, we have strong convergence of $\theta_k$ to $\theta$ in $C\big([0,T];H^1(\Omega)'\big)$, cf.\ Lemma~\ref{lemma:strong-convergence-theta}, which implies $\theta(0) \in H^1(\Omega)'$. Moreover, analogously as for $\pf$, we obtain the strong convergence of $\theta_{k,0} = \Pi_k^\theta \theta_0$ to $\theta_0$ in $L^2(\Omega)$, that is $\theta_k(0) \rightarrow \theta_0$ in $L^2(\Omega)$ for $k \rightarrow \infty$. Together with the componentwise continuity of the given duality pairing, we have for any arbitrary $g \in H^1(\Omega)$
\begin{align*}
 \left\langle \theta(0),g \right\rangle = \lim_{k \rightarrow \infty} \left\langle \theta_k(0),g \right\rangle = \left\langle \lim_{k \rightarrow \infty} \theta_k(0),g \right\rangle = \left\langle \theta_0,g \right\rangle.
\end{align*}
The strong convergences and continuous character of the definition of $\bm u_{k,0}$ allow to immediately deduce that $\bm{u}_{k,0} \rightarrow \bm{u}_0$ in $\vecSpHTr$, concluding the proof.
\end{proof}

\begin{lemma}\label{lemma:strong-convergence-u} 
There exists a strongly convergent subsequence of $\left\{\bm u_k\right\}_k$ such that
\begin{align}
 % &\bm u_k \rightarrow \bm u \textrm{ strongly in } L^2\big([0,T];\vecSpHTr\big) \textrm{ and a.e.~in } \Omega \times [0,T],
 % \label{eq:convdisp} \\
 &\strain[u_k] \rightarrow \strain \textrm{ strongly in } L^2\big([0,T];\tensSpL\big) \textrm{ and a.e.~in } \Omega \times [0,T].
 \label{eq:convstrain}
\end{align}
\end{lemma}
\begin{proof}
 As $\Span\{\bm \eta_i^{\bm u}\}_{i \in \mathbb{N}}$ is dense in $\vecSpHTr$, there exists a sequence $\left\{\bm v_k\right\}_{k \in \mathbb{N}}$ such that for all $k \in \mathbb{N}$ and a.e.~$t \in [0,T]$ it holds $\bm v_k(t) \in \mathcal{V}_k^{\bm u}$, $\bm v_k \rightarrow \bm u$ strongly in $L^2\big([0,T];\vecSpHTr\big)$, and $\partial_t \nabla \cdot \bm v_k \rightarrow \partial_t \nabla \cdot \bm u$ strongly in $L^2\big([0,T]; L^2(\Omega)\big)$. Thus, we further have $\bm u_k - \bm v_k \rightarrow 0$ weakly in $L^2\big([0,T];\vecSpHTr\big)$. Now, we test \eqref{eq:elasticitydiscspace}, integrated in time over $[0,T]$, with the difference $\bm u_k - \bm v_k$ and obtain using (A3)
\begin{align*}
0 &= \int_0^T \Big[
\regularization \big( \partial_t \nabla \cdot \bm u_k, \nabla \cdot (\bm u_k - \bm v_k) \big) 
+ \big( \mathbb{C}(\pf_k) (\bm{\varepsilon}(\bm u_k) - \mathcal{T}(\varphi_k)) - \alpha \revision{\pi_k} \bm I, \bm{\varepsilon}(\bm u_k - \bm{v}_k) \big)
- \langle\bm f,\bm u_k - \bm v_k \rangle\Big]\; dt \\
 &\geq 
 \frac{\regularization}{2} \left\| \nabla \cdot \left(\bm u_k(T) - \bm v_k(T) \right) \right\|_{L^2(\Omega)}^2
 -
 \frac{\regularization}{2} \left\| \nabla \cdot \left(\bm u_k(0) - \bm v_k(0) \right) \right\|_{L^2(\Omega)}^2
 +
 c_{\mathbb{C}} \|\bm \varepsilon(\bm u_k - \bm v_k)\|_{L^2([0,T];\tensSpL)}^2 \\
 &\quad + \int_0^T \Big[ \regularization \big( \partial_t \nabla \cdot \bm v_k, \nabla \cdot (\bm u_k - \bm v_k) \big) + \big( \mathbb{C}(\varphi_k) ({\bm \varepsilon}(\bm v_k)-\mathcal{T}(\varphi_k)) - \alpha \revision{\pi_k} \bm I, \bm \varepsilon(\bm u_k - \bm v_k) \big) - \langle \bm f,\bm u_k - \bm v_k \rangle \Big]\; dt.
\end{align*}
Dropping the first term on the right hand side, utilizing the fact that $\bm u_k(0) = \bm u_{k,0} \rightarrow \bm{u}(0)$ strongly in $\vecSpHTr$ by construction, cf.\ Lemma~\ref{lemma:convergence-initial-data}, and from the convergence assertions of $\bm v_k \rightarrow \bm u$, $\pf_k \rightarrow \pf$, cf.\ Lemma~\ref{lemma:strong-convergence-phi}, and $\revision{\pi_k}\rightarrow p$, cf.\ Lemma~\ref{lemma:strong-convergence-pk}, we get
\begin{align*}
 \|\bm \varepsilon(\bm u_k - \bm v_k)\|_{L^2\left([0,T];\tensSpL\right)}^2 \overset{k \rightarrow \infty}{\longrightarrow} 0. 
 % \text{ as }k \rightarrow \infty.
\end{align*}
The assertion follows together with Korn's inequality.
% \jbtodoinl{I left the rest in the comments as I believe it is standard and we have what we need.}
% as $k \rightarrow \infty$. By Korn's inequality we thus get
% \begin{align*}
%  \|\bm u_k - \bm v_k\|_{L^2\left([0,T];\vecSpHTr\right)}^2 \rightarrow 0,
% \end{align*}
% and therefore $\bm u_k - \bm v_k \rightarrow 0$ strongly in $L^2\big([0,T];\vecSpHTr\big)$. As a consequence of Minkowski's inequality we directly obtain the strong convergence of $\bm u_k$ to $\bm u$ in $L^2\big([0,T];\vecSpHTr\big)$ and, hence, together with the linearity of the strain also the strong convergence of $\strain[u_k]$ to $\strain$ in $L^2\big([0,T];\tensSpL\big)$. Therefore, by the Riesz-Fischer theorem \cite{Royden2010} there are subsequences, indexed again by k, such that
% \begin{align*}
%  &\bm u_k(x,t) \rightarrow \bm u(x,t) \textrm{ for a.e.~} (x,t) \in \Omega \times [0,T], \\
%  &\strain[u_k(x,t)] \rightarrow \strain[u(x,t)] \textrm{ for a.e.~} (x,t) \in \Omega \times [0,T].
% \end{align*}
% % which concludes the proof of the lemma.
\end{proof}

We continue with the main result.
\begin{lemma}
 The quintuple $(\pf_k, \mu_k, \bm u_k, \theta_k, \bm q_k)$ of solutions to the discrete-in-space system in \eqref{eq:discspace} in the limit $k \rightarrow \infty$ converges, up to subsequence, to a weak solution $(\pf, \mu, \bm u, \theta, \bm q)$ in the sense of Definition~\ref{def:Weak solution}.
\end{lemma}
\begin{proof}
We start by multiplying the discrete-in-space system in \eqref{eq:discspace} with an arbitrary test function $\vartheta \in C_c^\infty(0,T)$ and integrate in time (from $0$ to $T$), which yields for any arbitrary but fixed $j \in \{1,\ldots,k\}$
\begin{subequations}
\begin{align}
\int_0^T \Big[\langle \partial_t \varphi_k, \eta_j^\mathrm{ch}\rangle + \big(m(\pf_k) \nabla \mu_k,\nabla \eta_j^\mathrm{ch}\big) - \big(R,\eta_j^\mathrm{ch}\big)\Big] \vartheta(t) \;dt &= 0, 
\label{eq:ch1discspacelimit}\\
\int_0^T \Big[\big(\mu_k ,\eta_j^\mathrm{ch}\big) - \big(\delta_\varphi\mathcal{E}_\mathrm{i}(\varphi_k),\eta_j^\mathrm{ch}\big) - \big(\delta_\varphi\mathcal{E}_\mathrm{e}(\varphi_k, \strain[u_k]),\eta_j^\mathrm{ch}\big) \Big] \vartheta(t) \;dt &=0,
\label{eq:ch2discspacelimit}\\
% \int_0^T \Big[\regularization \big(\partial_t \nabla \cdot \bm u_k, \nabla \cdot \bm \eta_j^{\bm u} \big) + \big(\delta_{\bm \varepsilon}\mathcal{E}_{\revision{k}}(\varphi_k, \strain[u_k], \theta_k),\bm \varepsilon(\bm \eta_j^{\bm u})\big) - \langle \bm f,\bm \eta_j^{\bm u} \rangle\Big] \vartheta(t) \;dt &=0,
\int_0^T \Big[\regularization \big(\partial_t \nabla \cdot \bm u_k, \nabla \cdot \bm \eta_j^{\bm u} \big) + \big(\delta_{\bm \varepsilon}\mathcal{E}_\mathrm{e}(\varphi_k, \strain[u_k]),\bm \varepsilon(\bm \eta_j^{\bm u})\big) + \big(\delta_{\bm \varepsilon}\mathcal{E}_\mathrm{f,\revision{
k}}(\strain[u_k], \theta_k),\bm \varepsilon(\bm \eta_j^{\bm u})\big) - \langle \bm f,\bm \eta_j^{\bm u} \rangle\Big] \vartheta(t) \;dt &=0
\label{eq:elasticitydiscspacelimit}\\
\int_0^T \Big[\left\langle\partial_t\theta_k,\eta_j^\theta\right\rangle + \big(\nabla\cdot \bm q_k, \eta_j^\theta\big) - \big(S_f,\eta_j^\theta\big)\Big] \vartheta(t) \;dt &= 0,
\label{eq:flowdiscspacelimit}\\
\int_0^T \Big[\big(\kappa(\pf_k)^{-1}\bm q_k, \bm \eta_j^{\bm q} \big) - \big(\delta_\theta\energy_{\mathrm{f},\revision{
k}}(\strain[u_k],\theta_k),\nabla \cdot \bm \eta_j^{\bm q}\big)\Big] \vartheta(t) \;dt &= 0. \label{eq:darcyflowdiscspacelimit}
\end{align}
\label{eq:discspacelimit}%
\end{subequations}
We classify the occurring terms in three categories: Linear terms, semi-linear terms (terms with state-dependent parameters that are otherwise linear), 
and energy related nonlinear terms.
% and the remaining nonlinear terms.
In the following, we discuss each of the three types separately showing that the discrete terms converge to associated terms in the definition of a weak solution, cf.\ Definition~\ref{def:Weak solution}.

\paragraph{Linear terms}
We start our considerations with equation \eqref{eq:ch1discspacelimit} and the linear functional
\begin{align*}
 \partial_t \pf_k \mapsto \int_0^T \langle \partial_t \varphi_k, \eta_j^\mathrm{ch}\rangle \vartheta(t) \;dt.
\end{align*}
Since bounded linear operators are continuous and $\partial_t \pf_k \rightarrow \partial_t \pf$ weakly in $L^2\big([0,T];H^1(\Omega)'\big)$, we infer that
\begin{align*}
 \int_0^T \langle \partial_t \varphi_k, \eta_j^\mathrm{ch}\rangle \vartheta(t) \;dt \rightarrow \int_0^T \langle \partial_t \varphi, \eta_j^\mathrm{ch}\rangle \vartheta(t) \;dt
\end{align*}
for $k \rightarrow \infty$. The convergence of the other linear terms in \eqref{eq:discspacelimit} follows similarly. \revision{Note that this includes the terms related to the discrete fluid energy.} However, for the flux term in \eqref{eq:flowdiscspacelimit} we need to take into account that while $\bm q_k \in H_0(\Div,\Omega)$ we only attain weak convergence of $\bm q_k$ to $\bm q$ in $L^2\big([0,T];\vecSpL\big)$ so that we obtain
\begin{align*}
 \int_0^T \big(\nabla\cdot \bm q_k, \eta_j^\theta\big) \vartheta(t) \;dt = - \int_0^T \big(\bm q_k, \nabla\eta_j^\theta\big) \vartheta(t) \;dt \rightarrow - \int_0^T \big(\bm q, \nabla\eta_j^\theta\big) \vartheta(t) \;dt.
\end{align*}

\paragraph{Semi-linear terms}
\revision{The only semi-linear contributions are those involving} the mobility, permeability, and elasticity tensor. Starting with the former, the continuity of the  chemical mobility $m(\cdot)$ together with the convergence of $\pf_k$ to $\pf$ a.e.~in $\Omega\times[0,T]$ leads to
\begin{align*}
 m(\pf_k) \rightarrow m(\pf) \textrm{ a.e.~in } \Omega\times[0,T].
\end{align*}
By the boundedness of $m(\cdot)$ we further have
\begin{align*}
 \|m(\pf_k) \nabla \eta_j^\mathrm{ch} \vartheta(t)\|_{\vecSpL} \leq C_m \|\nabla \eta_j^\mathrm{ch}\|_{\vecSpL} \sup_{t \in [0,T]} |\vartheta(t)|,
\end{align*}
and therefore by application of Lebesgue's dominated convergence theorem (for the Bochner integral~\cite{Bogdan1965})
\begin{align*}
 \int_0^T \|m(\pf_k) \nabla \eta_j^\mathrm{ch} \vartheta(t) - m(\pf)\nabla \eta_j^\mathrm{ch} \vartheta(t)\|_{\vecSpL}^2 \;dt \rightarrow 0
\end{align*}
as $k \rightarrow \infty$ and thus
\begin{align*}
 m(\pf_k) \nabla \eta_j^\mathrm{ch} \vartheta(t) \rightarrow m(\pf) \nabla \eta_j^\mathrm{ch} \vartheta(t) \textrm{ strongly in } L^2\big([0,T];\vecSpL\big).
\end{align*}
Since $\nabla \mu_k \rightarrow \nabla \mu$ weakly in $L^2\big([0,T];\vecSpL\big)$, using the lemma for products of weak-strong converging sequences (\cite{Boyer2013}, Proposition II.2.12) we obtain
\begin{align*}
 \int_0^T \big( m(\pf_k) \nabla \mu_k,\nabla \eta_j^\mathrm{ch} \big) \vartheta(t) \;dt &= \int_0^T \big( \nabla \mu_k, m(\pf_k) \nabla \eta_j^\mathrm{ch} \vartheta(t) \big) \;dt
 % \\
 % &
 \rightarrow 
 % \int_0^T \big( \nabla \mu, m(\pf) \nabla \eta_j^\mathrm{ch} \vartheta(t) \big) \;dt = 
 \int_0^T \big( m(\pf) \nabla \mu,\nabla \eta_j^\mathrm{ch} \big) \vartheta(t) \;dt,
\end{align*}
for $k \rightarrow \infty$. The convergence of the term comprising the permeability $\kappa(\pf)$ in~\eqref{eq:darcyflowdiscspacelimit} follows similarly from the continuity and boundedness of $\kappa(\cdot)$, cf.\ assumption (A2), along with the weak convergence of $\bm q_k$ to $\bm q$ in $L^2\big([0,T];\vecSpL\big)$.

\paragraph{Free energy related nonlinear terms}
One term requiring an extra discussion involves the elasticity terms. Firstly, we split $\delta_\varphi\mathcal{E}_\mathrm{e}(\varphi_k, \strain[u_k])$ into a semi-quadratic and a non-convex contribution as 
\begin{align*}
 &\int_0^T \big(\delta_\varphi\mathcal{E}_\mathrm{e}(\varphi_k, \strain[u_k]),\eta_j^\mathrm{ch}\big) \vartheta(t) \;dt = \int_0^T \Big(\frac{1}{2} \big({\bm \varepsilon}(\bm u_k) - \mathcal{T}(\varphi_k)\big)\!:\!\mathbb{C}'(\varphi_k)\big({\bm \varepsilon}(\bm u_k) - \mathcal{T}(\varphi_k)\big),\eta_j^\mathrm{ch}\Big) \vartheta(t) \;dt \\
 &\hspace{6cm}- \int_0^T \big(\mathcal{T}'(\varphi_k)\!:\!\mathbb{C}(\varphi_k)\big({\bm \varepsilon}(\bm u_k) - \mathcal{T}(\varphi_k)\big),\eta_j^\mathrm{ch}\big) \vartheta(t) \;dt
\end{align*}
For the semi-quadratic contribution we obtain
% \jbtodoinl{In all semi-quadratic terms one has to extract an $L^\infty$ bound for the test function.}
\begin{align*}
 \Big\|\frac{1}{2} \big({\bm \varepsilon}(\bm u_k) - \mathcal{T}(\varphi_k)\big)\!&:\!\mathbb{C}'(\varphi_k)\big({\bm \varepsilon}(\bm u_k) - \mathcal{T}(\varphi_k)\big) \eta_j^\mathrm{ch} \vartheta(t)\Big\|_{L^1(\Omega)} \\
 &\leq \left(L_\mathbb{C} \|\strain[u_k]\|_{\tensSpL}^2 + L_\mathbb{C} C_{\mathcal{T}}^2 \|\pf_k\|_{L^2(\Omega)}^2\right) \|\eta_j^\mathrm{ch}\|_{L^{\infty}(\Omega)} \sup_{t \in [0,T]} |\vartheta(t)| \\
 % &\overset{k \rightarrow \infty}{\longrightarrow} \left(L_\mathbb{C} \|\strain[u]\|_{\tensSpL}^2 + L_\mathbb{C} C_{\mathcal{T}}^2 \|\pf\|_{L^2(\Omega)}^2\right) \|\eta_j^\mathrm{ch}\|_{L^{\infty}(\Omega)} \sup_{t \in [0,T]} |\vartheta(t)|
 \overset{\eqref{eq:bound}}&{\leq} \max\{L_\mathbb{C},L_\mathbb{C} C_{\mathcal{T}}^2\} C_T  \|\eta_j^\mathrm{ch}\|_{L^{\infty}(\Omega)} \sup_{t \in [0,T]} |\vartheta(t)|
\end{align*}
which is bounded since $\eta_j^\mathrm{ch} \in H_{\bm n}^2(\Omega) \subseteq L^\infty(\Omega)$. The non-convex contribution is in fact a semi-linear term, due to (A3) and (A4), and can be treated analogously.
Thus, together with the strong convergence $\bm{\varepsilon}(\bm{u}_k) \rightarrow \bm{\varepsilon}(\bm{u})$ in $L^2\big([0,T];\tensSpL\big)$, cf.\ Lemma~\ref{lemma:strong-convergence-u}, by application of Lebesgue's dominated convergence theorem we have
\begin{align*}
 \int_0^T \big(\delta_\varphi\mathcal{E}_\mathrm{e}(\varphi_k, \strain[u_k]),\eta_j^\mathrm{ch}\big) \vartheta(t) \;dt \rightarrow \int_0^T \big(\delta_\varphi\mathcal{E}_\mathrm{e}(\varphi, \strain[u]),\eta_j^\mathrm{ch}\big) \vartheta(t) \;dt.
\end{align*}
Secondly, in a similar manner, utilizing the continuity assumptions (A3) and (A4), we obtain for $k\rightarrow \infty$
\begin{align*}
 \int_0^T &\big(\delta_{\bm \varepsilon}\mathcal{E}_\mathrm{e}(\varphi_k, \strain[u_k]),\bm \varepsilon(\bm \eta_j^{\bm u})\big) \vartheta(t) \;dt = \int_0^T \big(\mathbb{C}(\varphi_k)\left({\bm \varepsilon}(\bm u_k)-\mathcal{T}(\varphi_k)\right),\bm \varepsilon(\bm \eta_j^{\bm u})\big) \vartheta(t) \;dt \\
&\qquad \rightarrow \int_0^T \big(\delta_{\bm \varepsilon}\mathcal{E}_\mathrm{e}(\varphi, \strain[u]),\bm \varepsilon(\bm \eta_j^{\bm u})\big) \vartheta(t) \;dt.
\end{align*}
Lastly, we consider the interfacial energy related term
\begin{align*}
 \int_0^T \big( \delta_\varphi\mathcal{E}_\mathrm{i}(\varphi_k),\eta_j^\mathrm{ch} \big) \vartheta(t) \;dt = \gamma\ell\int_0^T \big( \nabla \varphi_k, \nabla \eta_j^\mathrm{ch} \big) \vartheta(t) \;dt + \frac{\gamma}{\ell}\int_0^T \big(\Psi'(\varphi_k),\eta_j^\mathrm{ch}\big) \vartheta(t) \;dt.
\end{align*}
The first right hand side term can be represented by a linear functional and is treated accordingly. For the second term, the continuity of $\Psi'(\cdot)$ together with the convergence of $\pf_k$ to $\pf$ a.e.~in $\Omega\times[0,T]$ leads to
\begin{align*}
 \Psi'(\pf_k) \rightarrow \Psi'(\pf) \textrm{ a.e.~in } \Omega\times[0,T].
\end{align*}
Furthermore, by Hölder's inequality and assumption (A1) we obtain
\begin{align*}
 \|\Psi'(\pf_k) \eta_j^\mathrm{ch} \vartheta(t)\|_{L^1(\Omega)} &\leq \|\Psi'(\pf_k)\|_{L^1(\Omega)} \|\eta_j^\mathrm{ch}\|_{L^\infty(\Omega)} \sup_{t \in [0,T]} |\vartheta(t)| \\
 &\leq C_\Psi\left(\|\Psi(\pf_k)\|_{L^1(\Omega)} + \|\pf_k\|^2_{L^2(\Omega)}\right) \|\eta_j^\mathrm{ch}\|_{L^\infty(\Omega)} \sup_{t \in [0,T]} |\vartheta(t)| \\
 \overset{\eqref{eq:bound}}&{\leq} C_\Psi C_T \|\eta_j^\mathrm{ch}\|_{L^\infty(\Omega)} \sup_{t \in [0,T]} |\vartheta(t)|
\end{align*}
so that we can apply Lebesgue's dominated convergence theorem to get
\begin{align*}
 \int_0^T \|\Psi'(\pf_k) \eta_j^\mathrm{ch} \vartheta(t) - \Psi'(\pf) \eta_j^\mathrm{ch} \vartheta(t)\|_{L^1(\Omega)} \;dt \rightarrow 0
\end{align*}
as $k \rightarrow \infty$ and thus
\begin{align*}
 \int_0^T \big(\Psi'(\varphi_k),\eta_j^\mathrm{ch}\big) \vartheta(t) \;dt &= \int_0^T \int_\Omega \Psi'(\varphi_k) \eta_j^\mathrm{ch} \vartheta(t) \;dx \;dt 
 % \\
 % &
 \rightarrow 
 % \int_0^T \int_\Omega \Psi'(\varphi) \eta_j^\mathrm{ch} \vartheta(t) \;dx \;dt =
 \int_0^T \big(\Psi'(\varphi),\eta_j^\mathrm{ch}\big) \vartheta(t) \;dt.
\end{align*}

Now, we pass to the limit $k \rightarrow \infty$ in \eqref{eq:discspacelimit}, leading to
\begin{align*}
 \int_0^T \Big[\langle \partial_t \varphi, \eta_j^\mathrm{ch}\rangle + \big(m(\pf) \nabla \mu,\nabla \eta_j^\mathrm{ch}\big) - \big(R,\eta_j^\mathrm{ch}\big)\Big] \vartheta(t) \;dt &= 0, \\
 \int_0^T \Big[\big(\mu ,\eta_j^\mathrm{ch}\big) - \big(\delta_\varphi\mathcal{E}_\mathrm{i}(\varphi),\eta_j^\mathrm{ch}\big) - \big(\delta_\varphi\mathcal{E}_\mathrm{e}(\varphi, \strain[u]),\eta_j^\mathrm{ch}\big) \Big] \vartheta(t) \;dt &=0, \\
 \int_0^T \Big[\regularization \big(\partial_t \nabla \cdot \bm u, \nabla \cdot \bm \eta_j^{\bm u} \big) + \big(\delta_{\bm \varepsilon}\mathcal{E}_\mathrm{e}(\varphi, \strain[u]),\bm \varepsilon(\bm \eta_j^{\bm u})\big) + \big(\delta_{\bm \varepsilon}\mathcal{E}_\mathrm{f}( \strain[u], \theta),\bm \varepsilon(\bm \eta_j^{\bm u})\big) - \langle \bm f,\bm \eta_j^{\bm u} \rangle\Big] \vartheta(t) \;dt &=0, \\
 \int_0^T \Big[\left\langle\partial_t\theta,\eta_j^\theta\right\rangle + \big(\bm q, \nabla \eta_j^\theta\big) - \big(S_f,\eta_j^\theta\big)\Big] \vartheta(t) \;dt &= 0, \\
 \int_0^T \Big[\big(\kappa(\pf)^{-1}\bm q, \bm \eta_j^{\bm q} \big) - \big(\delta_\theta\energy_\mathrm{f}(\strain[u],\theta),\nabla \cdot \bm \eta_j^{\bm q}\big)\Big] \vartheta(t) \;dt &= 0.
\end{align*}
As $\vartheta$ and $j$ were chosen arbitrarily, the above holds true for all $\vartheta \in C_c^\infty(0,T)$ and $j \in \mathbb{N}$. By applying the fundamental lemma of the calculus of variations, using the density of $\Span\{\eta^\mathrm{ch}_j\}_{j \in \mathbb{N}}$ in $H^1(\Omega)$, of $\Span\{\bm \eta^{\bm u}_j\}_{j \in \mathbb{N}}$ in $\vecSpHTr$, of $\Span\{\eta^{\theta}_j\}_{j \in \mathbb{N}}$ in $H^1(\Omega)$ and of $\Span\{\bm \eta^{\bm q}_j\}_{j \in \mathbb{N}}$ in $H(\Div, \Omega)$ 
we can infer that the quintuple $(\pf, \mu, \bm u, \theta, \bm q)$ satisfies \eqref{eq:weak}.

From Lemma~\ref{lemma:convergence-initial-data} and in particular its proof, it follows that $\nabla \cdot \bm{u}$, $\pf$ and $\theta$ attain their respective initial conditions. 
% \jbtodoinl{discussion on initial data moved to separate lemma}
% The first follows immediately from the strong convergence result in Lemma~\ref{lemma:strong-convergence-u}. For $\pf$, we have strong convergence of $\pf_k$ to $\pf$ in $C\big([0,T];L^2(\Omega)\big)$, cf.\ Lemma~\ref{lemma:strong-convergence-phi}, which implies $\pf_k(0) \rightarrow \pf(0)$ in $L^2(\Omega)$ for $k \rightarrow \infty$. Moreover, we obtain the strong convergence of $\pf_{k,0} = \Pi_k^\mathrm{ch} \pf_0$ to $\pf_0$ in $L^2(\Omega)$ (\cite{Brezis2011}, Corollary 5.10), that is $\pf_k(0) \rightarrow \pf_0$ in $L^2(\Omega)$ for $k \rightarrow \infty$. Since the limit is unique, we get the meaningful initial condition $\pf(0) = \pf_0$. For $\theta$, we have strong convergence of $\theta_k$ to $\theta$ in $C\big([0,T];H^1(\Omega)'\big)$, cf.\ Lemma~\ref{lemma:strong-convergence-theta}, which implies $\theta(0) \in H^1(\Omega)'$. Moreover, analogously as for $\pf$, we obtain the strong convergence of $\theta_{k,0} = \Pi_k^\theta \theta_0$ to $\theta_0$ in $L^2(\Omega)$, that is $\theta_k(0) \rightarrow \theta_0$ in $L^2(\Omega)$ for $k \rightarrow \infty$. Together with the componentwise continuity of the given duality pairing, we have for any arbitrary $g \in H^1(\Omega)$
% \begin{align*}
%  \left\langle \theta(0),g \right\rangle = \lim_{k \rightarrow \infty} \left\langle \theta_k(0),g \right\rangle = \left\langle \lim_{k \rightarrow \infty} \theta_k(0),g \right\rangle = \left\langle \theta_0,g \right\rangle.
% \end{align*}
We thus conclude that the limit quintuple $(\pf, \mu, \bm u, \theta, \bm q)$ is a weak solution in the sense of Definition~\ref{def:Weak solution}.
\end{proof}

\begin{lemma}[Energy estimate]
The quintuple $(\pf, \mu, \bm u, \theta, \bm q)$, being a weak solution in the sense of Definition~\ref{def:Weak solution}, satisfies the energy estimate~\eqref{eq:boundenergy}.
\end{lemma}
\begin{proof}
Combining the estimate~\eqref{eq:bound} resulting from Grönwall's inequality with \eqref{eq:boundmu} and \eqref{eq:boundphi-theta}, we obtain 
\begin{equation}
\begin{aligned}
 \underset{t \in [0,T]}{\sup} \Big(\|\Psi&(\pf_k(t))\|_{L^1(\Omega)} + \|\pf_k(t)\|_{H^1(\Omega)}^2 + \|\bm u_k(t)\|_{\vecSpH}^2 + \|\theta_k(t)\|_{L^2(\Omega)}^2\Big) + \regularization \|\partial_t \nabla \cdot \bm{u}_k\|^2_{L^2([0,T];L^2(\Omega))} \\
 &+ \|\mu_k\|_{L^2([0,T];H^1(\Omega))}^2 + \|\bm q_k\|_{L^2([0,T];\vecSpL)}^2 + \|\pf_k\|_{H^1\left([0,T];H^1(\Omega)'\right)}^2 + \|\theta_k\|_{H^1\left([0,T];H^1(\Omega)'\right)}^2 \leq C.
\end{aligned}
\label{eq:boundenergydisc}
\end{equation}
The continuity of $\Psi(\cdot)$ together with the convergence of $\pf_k$ to $\pf$ a.e.~in $\Omega\times[0,T]$ leads to
\begin{align*}
 \Psi(\pf_k) \rightarrow \Psi(\pf) \textrm{ a.e.~in } \Omega\times[0,T].
\end{align*}
Then, the non-negativity of $\Psi(\cdot)$ and Fatou's lemma (\cite{Brezis2011}, Lemma 4.1) yield
\begin{align*}
 \|\Psi(\pf(t))\|_{L^1(\Omega)} = \int_\Omega \Psi(\pf(t)) \;dx &= \int_\Omega \liminf_{k \rightarrow \infty} \Psi(\pf_k(t)) \;dx \\
 &\leq \liminf_{k \rightarrow \infty} \int_\Omega \Psi(\pf_k(t)) \;dx = \liminf_{k \rightarrow \infty} \|\Psi(\pf_k(t))\|_{L^1(\Omega)}
\end{align*}
for a.e.~$t \in [0,T]$. Finally, we make use of the weak and weak-$\ast$ lower semicontinuity of the norms (\cite{Brezis2011}, §3, Remark 6) allowing us to pass to the limit $k \rightarrow \infty$ in \eqref{eq:boundenergydisc} which yields the desired energy estimate~\eqref{eq:boundenergy}.
\end{proof}

Above, weak solutions in the sense of Definition~\ref{def:Weak solution} are merely established for initial data with increased regularity $\pf_0 \in H^2_{\bm{n}}(\Omega)$. In order to complete the proof of Theorem~\ref{thm:existence}, it remains to discuss the initial conditions satisfying lower regularity $\pf_0 \in H^1(\Omega)$ with $\Psi(\pf_0) \in L^1(\Omega)$.

\begin{proof}[Finalization of proof to Theorem~\ref{thm:existence} (general initial conditions)]
Following ideas of \cite{Colli2012}, p.\hspace{0.05cm}440, and \cite{Garcke2021}, §3.5, for such $\pf_0$, we can construct a regularizing sequence $\{\pf_{0,n}\}_{n\in\mathbb{N}}\subset H^2_{\bm{n}}(\Omega)$ such that $\pf_{0,n} \rightarrow \pf_0$ strongly in $L^2(\Omega)$ and weakly in $H^1(\Omega)$. Furthermore, a uniform $L^1(\Omega)$ bound can be derived for $\Psi(\pf_{0,n})$.  Here, assumption (A1) and in particular the twice continuous differentiability of $\Psi(\cdot)$ with $\Psi'(0) = 0$ as well as~\eqref{eq:boundpsidoubleprime} enter.
Finally, for each regularized initial datum $\pf_{0,n}\in H_{\bm n}^2(\Omega)$, let the quintuple $(\pf_n, \mu_n, \bm u_n, \theta_n, \bm q_n)$ be a weak solution to \eqref{eq:model} in the sense of Definition~\ref{def:Weak solution} with initial conditions $\pf_{0,n}$ and $\theta_0$. Then, by arguments similar to those for the Faedo-Galerkin method, we can pass to the limit $n \rightarrow \infty$ with the quintuple of limit functions denoted again as $(\pf, \mu, \bm u, \theta, \bm q)$ that is a weak solution to \eqref{eq:model} in the sense of Definition~\ref{def:Weak solution} and fulfills the energy estimate~\eqref{eq:boundenergy} which concludes the proof of Theorem~\ref{thm:existence}.
\end{proof}

\section{Continuous dependence and uniqueness} \label{sec:continuous-dependence}

The goal of this section is to establish continuous dependence of weak solutions on the data, and eventually also uniqueness of weak solutions of the Cahn-Hilliard-Biot model, cf.\ Section~\ref{sec:model}. For this, we make stronger assumptions on the model compared to the existence analysis in Section~\ref{sec:existence}. In particular, we assume all material parameter laws to be independent of the phase-field variable $\pf$. In addition, we consider merely solutions with $H(\Div,\Omega)$ regularity for fluxes, allowing for retaining a natural saddle point structure of the corresponding governing equations~\eqref{eq:weak}, also cf.\ Remark~\ref{remark:hdiv-flux}. In addition, we neglect secondary consolidation type regularization and set $\regularization=0$.

We consider the following simplifying assumptions on the material parameters.
\begin{itemize}
 \item[(B1)] The derivative $\Psi'(\pf)$ of the double-well potential is Lipschitz continuous with Lipschitz constant $L_{\Psi'}$.
 % We further assume that there exists a convex-concave split of $\Psi = \Psi_1 - \Psi_2$ with $\Psi_i$ convex and $\Psi_2'$ being Lipschitz continuous with Lipschitz constant $L_{\Psi_2'}$.
 \item[(B2)] Chemical mobility $m(\pf) = m$, permeability $\kappa(\pf) = \kappa$, fluid compressibility coefficient $M(\pf) = M$ and Biot-Willis coupling coefficient $\alpha(\pf) = \alpha$ are positive constants.
 \item[(B3)] Homogeneous elasticity: The elastic stiffness tensor $\mathbb{C}(\pf) = \mathbb{C}$ is constant, symmetric and positive definite, satisfying $c_{\mathbb{C}}|\bm \varepsilon|^2 \leq \bm \varepsilon\!:\!\mathbb{C} \bm \varepsilon \leq C_{\mathbb{C}}|\bm \varepsilon|^2$ for all $\bm \varepsilon \in \mathbb{R}^{d\times d}_\mathrm{sym}$ with positive constants $c_{\mathbb{C}}$ and $C_{\mathbb{C}}$.
 \item[(B4)] Vegard's law: The eigenstrain $\mathcal{T}$ satisfies the affine linear ansatz $\mathcal{T}(\pf) = \bm{\hat \varepsilon} \pf + \bm{\varepsilon^\ast}$ with constant symmetric tensors $\bm{\hat \varepsilon}$ and $\bm{\varepsilon^\ast}$, satisfying $\|\bm{\hat \varepsilon}\|_{\tensSpL} \leq C_{\bm{\hat \varepsilon}}$ and $\|\bm{\hat \varepsilon}\tilde\pf\|_{\tensSpL} \leq C_{\bm{\hat \varepsilon}} \|\tilde\pf\|_{L^2(\Omega)}$ with positive constant $C_{\bm{\hat \varepsilon}}$.
\end{itemize}

\begin{theorem}[Continuous dependence\label{thm:cdweak}]
Let assumptions (B1)--(B4) hold. Let $\left(R_i, \bm{f}_i, S_{f_i}, \pf_{0,i}, \theta_{0,i}\right)$, $i=1,2$, denote two sets of data in the sense of (A5)--(A6), and $\{(\pf_i, \mu_i, \bm u_i, \theta_i, \bm q_i)\}_{i \in \{1,2\}}$ be corresponding solutions to \eqref{eq:model} in the sense of Definition~\ref{def:Weak solution}. Under the additional assumption of increased regularity for the fluxes $\bm q_i \in L^2\big([0,T];H_0(\Div,\Omega)\big)$, there exists a constant $C>0$ independent of the solutions such that
\begin{align}
 &\underset{t\in[0,T]}{\mathrm{ess\, sup}}\,\|\pf_1(t) - \pf_2(t)\|_{\left(H^1(\Omega) \cap L_0^2(\Omega) \right)'}^2 + \underset{t\in[0,T]}{\mathrm{ess\, sup}}\,\|\theta_1(t) - \theta_2(t)\|_{\left(H^1(\Omega) \cap L_0^2(\Omega) \right)'}^2 \nonumber \\
 &\quad+
 \int_0^T \bigg[
 \|\pf_1 - \pf_2\|_{H^1(\Omega)}^2 
 + \|\mu_1 - \mu_2\|_{H^1(\Omega)'}^2
 + \|\theta_1 - \theta_2\|_{L^2(\Omega)}^2 
 + \|\bm{u}_1 - \bm{u}_2\|_{\vecSpH}^2 
 + \|\bm{q}_1 - \bm{q}_2\|_{H(\Div,\Omega)'}^2 \, 
 \bigg]dt \nonumber \\
 &\leq C \bigg[\int_0^T \left(
 \|R_1 - R_2\|_{L^2(\Omega)}^2
 + \|\bm{f}_1 - \bm{f}_2\|_{\bm{H^{-1}(\Omega)}}^2 
 + \|S_{f_1} - S_{f_2}\|_{L^2(\Omega)}^2
 \right) dt \nonumber\\
 &\qquad
 + \int_0^T
 \left\|\left(\pf_{0,1}^\Omega - \pf_{0,2}^\Omega\right) + t \left(R_1^\Omega - R_2^\Omega\right) \right\|_{L^2(\Omega)}^2 dt
 + \|\pf_{0,1} - \pf_{0,2}\|_{L^2(\Omega)}
 + \|\theta_{0,1} - \theta_{0,2}\|_{\left(H^1(\Omega) \cap L_0^2(\Omega) \right)'}^2 
\bigg]
\label{eq:continuity}
\end{align}
where $(\cdot)^\Omega := |\Omega|^{-1}\int_\Omega(\cdot)\, dx$ denotes the mean value.
\end{theorem}

A direct consequence of the continuous dependence property is the uniqueness of weak solutions.
\begin{theorem}[Uniqueness of weak solutions]
Let assumptions (A5)--(A6) and (B1)--(B4) hold. Then there exists at most one quintuple $(\pf, \mu, \bm u, \theta, \bm q)$ being the weak solution to~\eqref{eq:model} in the sense of Definition~\ref{def:Weak solution}, under the additional assumption of increased regularity for the flux $\bm q \in L^2\big([0,T];H(\nabla\cdot,\Omega)\big)$.
\end{theorem}

\begin{proof}[Proof of Theorem~\ref{thm:cdweak}]
In the following, we utilize auxiliary problems to define suitable test functions, which are related to the dual of the inherent dissipation potentials used above. In particular, we implicitly utilize the generalized gradient flow structure also depicted in~\cite{Storvik2022}.

We introduce the shorthand notation for the differences 
\begin{align*}
    \tilde{R}:= R_1 - R_2,\quad 
    \tilde{\bm{f}} := \bm{f}_1 - \bm{f}_2,\quad
    \tilde{S}_{f}:= S_{f_1} - S_{f_2},\quad
    \tilde{\pf}_0 := \pf_{0,1} - \pf_{0,2},\quad 
    \tilde{\theta}_0 := \theta_{0,1} - \theta_{0,2},
\end{align*}
as well as 
\begin{align}
\label{eq:diff-quintuple}
 \tilde{\pf} := \pf_1 - \pf_2, \quad \tilde{\mu} := \mu_1 - \mu_2, \quad \bm{\tilde u} := \bm u_1 - \bm u_2, \quad \tilde{\theta} := \theta_1 - \theta_2, \quad \bm{\tilde q} := \bm q_1 - \bm q_2,
\end{align}
satisfying
\begin{align*} 
\tilde\pf \in L^\infty\big(&[0,T];H^1(\Omega)\big)\cap H^1\big([0,T];H^1(\Omega)'\big), \quad \tilde\mu \in L^2\big([0,T];H^1(\Omega)\big), \quad \bm{\tilde u} \in L^\infty\big([0,T];\vecSpHTr\big), \\
&\tilde\theta \in L^\infty\big([0,T];L^2(\Omega)\big)\cap H^1\big([0,T];H^1(\Omega)'\big), \quad \bm{\tilde q} \in L^2\big([0,T];H(\nabla\cdot,\Omega)\big).
\end{align*}
We obtain the difference system by subtracting~\eqref{eq:weak} for both solutions
\begin{subequations}
\begin{align}
 \left\langle \partial_t \tilde \pf, \tilde\eta^\mathrm{ch} \right\rangle + \big( m \nabla\tilde\mu, \nabla\tilde\eta^\mathrm{ch} \big) &= \big( \tilde R, \tilde\eta^\mathrm{ch} \big)
 \label{eq:ch1weakdiffcd}\\
 \big( \tilde\mu,\tilde\eta^\mathrm{ch}\big) - \gamma\ell\big(\nabla\tilde\pf,\nabla\tilde\eta^\mathrm{ch}\big) - \frac{\gamma}{\ell} \big( \Psi'(\pf_1) - \Psi'(\pf_2),\tilde\eta^\mathrm{ch}\big) + \big( \bm{\hat \varepsilon}\!:\!\mathbb{C}\big(\bm \varepsilon(\bm{\tilde u}) - \bm{\hat \varepsilon}\tilde\pf\big),\tilde\eta^\mathrm{ch}\big) &= 0
 \label{eq:ch2weakdiffcd}\\
 \big(\mathbb{C}\big(\bm \varepsilon(\bm{\tilde u}) - \bm{\hat \varepsilon}\tilde\pf\big),\bm \varepsilon(\bm{\tilde\eta^{\bm u}})\big) + \big(-M\alpha\tilde\theta\bm{I}+M\alpha^2\Div\bm{\tilde u}\bm{I},\bm \varepsilon(\bm{\tilde\eta^{\bm u}})\big) &= \langle \bm{\tilde f},\bm{\tilde\eta^{\bm u}} \rangle 
 \label{eq:elasticityweakdiffcd}\\
 \left\langle\partial_t\tilde\theta,\tilde\eta^{\theta}\right\rangle + \big(\nabla\cdot \bm{\tilde q},\tilde\eta^{\theta}\big) &= \big( \tilde S_f,\tilde\eta^{\theta}\big)
 \label{eq:flowweakdiffcd}\\
 \big( \kappa^{-1} \bm{\tilde q},\bm{\tilde\eta^{\bm q}} \big) - \big( M(\tilde\theta - \alpha\Div\bm{\tilde u}),\nabla \cdot \bm{\tilde\eta^{\bm q}} \big) &= 0 
 \label{eq:darcyflowweakdiffcd} 
\end{align}
\label{eq:weakdiffcd}%
\end{subequations}
holding for all $\tilde\eta^\mathrm{ch} \in H^1(\Omega)$, $\bm{\tilde\eta^{\bm u}} \in \vecSpHTr$, $\tilde\eta^\theta \in L^2(\Omega)$, and $\bm{\tilde\eta^{\bm q}} \in H(\nabla\cdot,\Omega)$ and for a.e.~$t \in [0,T]$. In addition, the initial conditions are met $\tilde{\pf}(0) = \tilde{\pf}_0$ and $\tilde{\theta}(0) = \tilde{\theta}_0$ in the sense of Definition~\ref{def:Weak solution}.

From~\eqref{eq:ch1weakdiffcd}, we deduce, by testing with the non-zero, constant function $\tilde\eta^\mathrm{ch} = |\Omega|^{-1}$ that the mean value of $\tilde{\varphi}$ satisfies $\partial_t \tilde{\varphi}^\Omega = \tilde{R}^\Omega$, such that
\begin{align}
\label{eq:mean-phi-tilde}
    \tilde{\varphi}^\Omega(t) = \tilde{\varphi}^\Omega(0) + t \tilde{R}^\Omega = \tilde{\varphi}_0^\Omega + t \tilde{R}^\Omega,\quad t\in [0,T].
\end{align}
Under assumptions (A5)--(A6), $\left|\tilde{\varphi}^\Omega\right|$ is finite. Using orthogonality arguments and the Poincar\'e-Wirtinger inequality, introducing a constant $C_P > 0$, we obtain
\begin{align}
\label{eq:poincare-wirtinger}
    \| \tilde{\varphi} \|_{L^2(\Omega)}^2  =\| \tilde{\varphi} - \tilde{\varphi}^\Omega\|_{L^2(\Omega)}^2 + \mathbox{\tilde{\varphi}^\Omega}^2 |\Omega| \leq C_P \|\nabla \tilde{\varphi} \|_{L^2(\Omega)}^2 + \mathbox{\tilde{\varphi}^\Omega}^2 |\Omega|.
\end{align}

With the goal of reducing mixed formulations of subproblems in~\eqref{eq:weakdiffcd} to primal formulations, utilizing the inherent gradient flow structures, we choose the test functions $\tilde{\eta}^\mathrm{ch}$, $\bm{\tilde\eta^{\bm u}}$ and $\tilde{\eta}^\theta$ as follows:
\begin{itemize}
\item Consider the unique solution $-\Delta_m^{-1} \tilde\varphi := \hat{\mu} \in H^1(\Omega) \cap L_0^2(\Omega)$ of the auxiliary problem
\begin{align}
 \big( m \nabla\hat\mu, \nabla\hat\eta^\mathrm{ch} \big) &= \big( \tilde{\varphi}, \hat\eta^\mathrm{ch} \big)
\label{eq:auxproblem-ch-phi}
\end{align}
for all $\hat\eta^\mathrm{ch} \in H^1(\Omega) \cap L_0^2(\Omega)$. We then set $\tilde{\eta}^\mathrm{ch} = -\Delta_m^{-1} \tilde\varphi \in H^1(\Omega)$. Testing~\eqref{eq:auxproblem-ch-phi} with $\hat\eta^\mathrm{ch} = \tilde{\varphi} - \tilde{\varphi}^\Omega$, allows to deduce the duality property
\begin{align}
 \left\| \tilde{\varphi} - \tilde{\varphi}^\Omega \right\|_{L^2(\Omega)}^2 \leq \sqrt{\big( m \nabla \tilde{\varphi}, \nabla \tilde{\varphi} \big)} \cdot \sqrt{\big( -\Delta_m^{-1} \tilde{\varphi}, \tilde{\varphi}\big)}.
\label{eq:duality-property}
\end{align}
In addition, $-\Delta_m^{-1}$ induces a natural inner product on $\left(H^1(\Omega)\cap L_0^2(\Omega) \right)'$, satisfying a Cauchy-Schwarz-type inequality
\begin{align}
 \big( -\Delta_m^{-1} \tilde{\varphi}, \tilde{\psi} \big) \leq \sqrt{\big( -\Delta_m^{-1} \tilde{\varphi}, \tilde{\varphi} \big)} \cdot \sqrt{\big( -\Delta_m^{-1} \tilde{\psi}, \tilde{\psi} \big)}
\label{eq:cauchy-schwarz-auxiliary-problem}
\end{align}
and inducing a norm on $\left(H^1(\Omega)\cap L_0^2(\Omega) \right)'$.

\item Let $\bm{\tilde{\eta}^{\bm{u}}} = \bm{\tilde u} \in \vecSpHTr$.

\item Consider the unique solution $(\bm{\hat q}, \hat{\theta}) \in H_0(\Div, \Omega) \times L_0^2(\Omega)$ of the auxiliary problem
\begin{subequations}
\begin{align}
 \big( \kappa^{-1} \bm{\hat q}, \bm{\hat{\eta}^{\bm{q}}} \big)  - \big( \hat{\theta}, \nabla \cdot \bm{\hat{\eta}^{\bm{q}}} \big) &=0 \label{eq:aux-mixed-problem-darcy-flux} \\
 \big( \nabla \cdot \bm{\hat q}, \hat{\eta}^{\theta} \big) &= \big( \tilde{\theta}, \hat{\eta}^\theta \big) \label{eq:aux-mixed-problem-darcy-theta}%
\end{align}
\end{subequations}
for all $(\bm{\hat{\eta}^{\bm{q}}}, \hat{\eta}^\theta) \in H(\Div; \Omega) \times L_0^2(\Omega)$. In compact notation, we define $-\Delta_\kappa^{-1} \tilde{\theta}:= \hat{\theta}$. We then choose the test function $\tilde{\eta}^\theta = -\Delta_\kappa^{-1} \tilde{\theta} \in L^2(\Omega)$.
\end{itemize}

With the use of these test functions, we make the following observations, important to utilize the gradient flow structure where we introduce the shorthand notation $X_i:= (\varphi_i, \bm{\varepsilon(\bm{u}_i)}, \theta_i)$, $i=1,2$:
\begin{itemize}
\item Testing \eqref{eq:auxproblem-ch-phi} with $\hat{\eta}^{\mathrm{ch}} = \tilde{\mu}$ and \eqref{eq:ch2weakdiffcd} with $\tilde{\eta}^\mathrm{ch} = \tilde{\varphi}$, yields
\begin{align*}
% \big( m \nabla \tilde{\mu}, \nabla \left(-\Delta_m^{-1} \tilde{\varphi} \right)\big) &= 
 \big( m\nabla \tilde{\mu}, \nabla \hat{\mu} \big) \overset{\eqref{eq:auxproblem-ch-phi}}&{=} \big(\tilde{\mu}, \tilde{\varphi} \big) \overset{\eqref{eq:ch2weakdiffcd}}{=} \gamma\ell \big(\nabla\tilde\pf,\nabla \tilde\varphi \big) + \frac{\gamma}{\ell} \big( \Psi'(\pf_1) - \Psi'(\pf_2),\tilde\varphi\big) - \big( \bm{\hat \varepsilon}\!:\!\mathbb{C}\big(\bm \varepsilon(\bm{\tilde u}) -\bm{\hat \varepsilon}\tilde\pf\big),\tilde\varphi \big) \\
 &= \big( \delta_\varphi \mathcal{E}(X_1) - \delta_\varphi \mathcal{E}(X_2), \tilde\varphi \big).
\end{align*}

\item Testing \eqref{eq:ch1weakdiffcd} with $\tilde{\eta}^\mathrm{ch} = -\Delta_m^{-1} \tilde{\varphi} = \hat{\mu}$ yields (together with the above identity)
\begin{align*}
 \left\langle \partial_t \tilde \pf, -\Delta_m^{-1} \tilde{\varphi} \right\rangle +  \big( \delta_\varphi \mathcal{E}(X_1) - \delta_\varphi \mathcal{E}(X_2), \tilde\varphi \big) = \big( \tilde R, -\Delta_m^{-1} \tilde{\varphi} \big).
\end{align*}

\item Considering the left hand side of \eqref{eq:elasticityweakdiffcd} with $\bm{\tilde{\eta}^{\bm{u}}} = \bm{\tilde u}$ yields
\begin{align*}
 \big(\mathbb{C}\big(\bm \varepsilon(\bm{\tilde u}) - \bm{\hat \varepsilon}\tilde\pf\big),\bm \varepsilon(\bm{\tilde\eta^{\bm u}})\big) + \big(-M\alpha\tilde\theta\bm{I}+M\alpha^2\Div\bm{\tilde u}\bm{I},\bm \varepsilon(\bm{\tilde\eta^{\bm u}})\big) = \big( \delta_{\bm{\varepsilon}}\mathcal{E}(X_1) - \delta_{\bm{\varepsilon}}\mathcal{E}(X_2), \bm{\varepsilon}(\bm{\tilde u}) \big).
\end{align*}

\item Testing \eqref{eq:aux-mixed-problem-darcy-flux} with $\bm{\hat{\eta}^{\bm{q}}} = \bm{\tilde q}$, \eqref{eq:darcyflowweakdiffcd} with $\bm{\tilde{\eta}^{\bm{q}}} = \bm{\hat q}$, and~\eqref{eq:aux-mixed-problem-darcy-theta} with $\hat{\eta}^\theta = \delta_\theta \mathcal{E}(X_1) - \delta_\theta \mathcal{E}(X_2) \in L_0^2(\Omega)$ allows to show
\begin{align*}
 % \big( \nabla \cdot \bm{\tilde q}, \left(- \Delta_\kappa^{-1} \tilde{\theta}\right) \big) = 
 \big( \nabla \cdot \bm{\tilde q}, \hat\theta \big) \overset{\eqref{eq:aux-mixed-problem-darcy-flux}}&{=} \big( \kappa^{-1} \bm{\tilde q}, \bm{\hat q} \big) \overset{\eqref{eq:darcyflowweakdiffcd}}{=} \big( M(\tilde\theta - \alpha\Div\bm{\tilde u}),\nabla \cdot \bm{\hat q} \big) = \big( \delta_\theta \mathcal{E}(X_1) - \delta_\theta \mathcal{E}(X_2), \nabla \cdot \bm{\hat q} \big) \\
 \overset{\eqref{eq:aux-mixed-problem-darcy-theta}}&{=} \big( \delta_\theta \mathcal{E}(X_1) - \delta_\theta \mathcal{E}(X_2), \tilde{\theta} \big).
\end{align*}

\item Testing \eqref{eq:flowweakdiffcd} with $\tilde{\eta}^\theta = -\Delta_\kappa^{-1} \tilde{\theta} = \hat{\theta}$ yields (together with the above identity)
\begin{align*}
 \left\langle \partial_t \tilde \theta, -\Delta_\kappa^{-1} \tilde{\theta} \right\rangle +  \big( \delta_\theta \mathcal{E}(X_1) - \delta_\theta \mathcal{E}(X_2), \tilde\theta \big) = \big( \tilde S_f, -\Delta_\kappa^{-1} \tilde{\theta} \big).
\end{align*}
\end{itemize}
Finally, we summarize. Inserting $\tilde{\eta}^\mathrm{ch} = -\Delta_m^{-1} \tilde{\varphi}$, $\bm{\tilde{\eta}^{\bm{u}}} = \bm{\tilde u}$ and $\tilde{\eta}^\theta = - \Delta_\kappa^{-1} \tilde{\theta}$ into the equations~\eqref{eq:ch1weakdiffcd}, \eqref{eq:elasticityweakdiffcd}, and \eqref{eq:flowweakdiffcd}, summing them up, and utilizing the above observations, yields
\begin{align}
 \left\langle \partial_t \tilde{\varphi},  -\Delta_m^{-1} \tilde{\varphi} \right\rangle + \left\langle \partial_t \tilde{\theta},  -\Delta_\kappa^{-1} \tilde{\theta} \right\rangle + \big( \delta_{(\varphi, \bm \varepsilon, \theta)} \mathcal{E}(X_1) &- \delta_{(\varphi, \bm \varepsilon, \theta)} \mathcal{E}(X_2), X_1 - X_2 \big) \nonumber \\
 &= \big( \tilde{R}, -\Delta_m^{-1} \tilde{\varphi} \big) + \langle \bm{\tilde f}, \bm{\tilde u} \rangle  + \big( \tilde{S}_f, -\Delta_\kappa^{-1} \tilde{\theta} \big).
\label{eq:gradient-flow-tested}
\end{align}
Due to linearity and self-adjointness of $-\Delta_m^{-1}$ and $-\Delta_\kappa^{-1}$, it holds that
\begin{align}
 \left\langle \partial_t \tilde{\varphi},  -\Delta_m^{-1} \tilde{\varphi} \right\rangle + \left\langle \partial_t \tilde{\theta},  -\Delta_\kappa^{-1} \tilde{\theta} \right\rangle = \frac{d}{dt} \left(\frac{1}{2} \big( -\Delta_m^{-1} \tilde{\varphi} , \tilde{\varphi} \big) + \frac{1}{2} \big( -\Delta_\kappa^{-1} \tilde{\theta}, \tilde{\theta} \big) \right).
\label{eq:sub:gradient-flow-tested:dissipation}
\end{align}
Furthermore, under the given assumptions, the energy $\mathcal{E}$ can be decomposed as follows
\begin{align*}
 \mathcal{E}(\varphi, \bm{\varepsilon}(\bm{u}), \theta) = \mathcal{E}_{\mathrm{quad}}(\varphi, \bm{\varepsilon}(\bm{u}), \theta) + \frac{\gamma}{\ell} \int_\Omega \Psi(\varphi) \;dx
 % \frac{\gamma}{\ell} \int_\Omega \Psi_1(\varphi) \;dx - \frac{\gamma}{\ell} \int_\Omega \Psi_2(\varphi) \;dx
\end{align*}
with the quadratic contribution 
% ($\mathcal{T}(\pf) = \bm{\hat \varepsilon} \pf + \bm{\varepsilon^\ast}$)
\begin{align*}
 \mathcal{E}_{\mathrm{quad}}(\varphi, \bm{\varepsilon}(\bm{u}), \theta) = \gamma \frac{\ell}{2} \big( \nabla \varphi, \nabla \varphi \big) + \frac{1}{2} \big( \left(\bm{\varepsilon}(\bm{u}) - \mathcal{T}(\varphi)\right),\mathbb{C}\left(\bm{\varepsilon}(\bm{u}) - \mathcal{T}(\varphi)\right) \big) + \frac{1}{2} \big( M(\theta - \alpha \nabla \cdot \bm{u}), \theta - \alpha \nabla \cdot \bm{u} \big),
\end{align*}
also defining a semi-norm on $H^1(\Omega) \times \bm{H_0^1(\Omega)} \times L^2(\Omega)$. The latter can be shown by employing binomial identities together with
\begin{align}
 |||X_1 - X_2 |||^2 := \big( \delta_{(\varphi, \bm \varepsilon, \theta)} \mathcal{E}_\mathrm{quad}(X_1) - \delta_{(\varphi, \bm \varepsilon, \theta)} \mathcal{E}_\mathrm{quad}(X_2), X_1 - X_2 \big)
 % = 2 \mathcal{E}_{\mathrm{quad}}(\varphi_1 - \varphi_2, \bm{\varepsilon}(\bm{u}_1 - \bm{u}_2), \theta_1 - \theta_2)
\label{eq:sub:gradient-flow:norm}
\end{align}
(here for arbitrary $X_i$).
Utilizing assumption (B1), \eqref{eq:poincare-wirtinger} and~\eqref{eq:duality-property} together with Young's inequality and the Cauchy-Schwarz inequality, we obtain  
\begin{align}
\label{eq:sub:gradient-flow-tested:energy}
 \big( &\delta_{(\varphi, \bm \varepsilon, \theta)} \mathcal{E}(X_1) - \delta_{(\varphi, \bm \varepsilon, \theta)} \mathcal{E}(X_2), X_1 - X_2 \big) \nonumber\\
 &= |||X_1 - X_2 |||^2 + \frac{\gamma}{\ell}\int_\Omega (\Psi'(\varphi_1) - \Psi'(\varphi_2))\tilde{\varphi} \;dx \nonumber \\
 % \frac{\gamma}{\ell}\int_\Omega (\Psi_1'(\varphi_1) - \Psi_1'(\varphi_2))\tilde{\varphi} \;dx - \frac{\gamma}{\ell}\int_\Omega (\Psi_2'(\varphi_1) - \Psi_2'(\varphi_2))\tilde{\varphi} \;dx \nonumber\\
 \overset{\text{(B1)}}&{\geq} |||X_1 - X_2 |||^2 - \frac{\gamma L_{\Psi'}}{\ell} \|\tilde{\varphi}\|_{L^2(\Omega)}^2 \nonumber\\
 \overset{\eqref{eq:poincare-wirtinger}}&{=} |||X_1 - X_2 |||^2 - \frac{\gamma L_{\Psi'}}{\ell} \|\tilde{\varphi} - \mathbox{\tilde{\varphi}^\Omega}\|_{L^2(\Omega)}^2 - \frac{\gamma L_{\Psi'}}{\ell} |\Omega| \mathbox{\tilde{\varphi}^\Omega}^2\nonumber\\
 \overset{\eqref{eq:duality-property}}&{\geq} |||X_1 - X_2 |||^2 - \frac{\gamma L_{\Psi'}}{\ell} \sqrt{\big( m \nabla \tilde{\varphi}, \nabla \tilde{\varphi} \big)} \cdot \sqrt{\big( -\Delta_m^{-1} \tilde{\varphi}, \tilde{\varphi}\big)} - \frac{\gamma L_{\Psi'}}{\ell} |\Omega| \mathbox{\tilde{\varphi}^\Omega}^2 \nonumber\\
 &\geq |||X_1 - X_2 |||^2 - \gamma\frac{\ell}{4} \|\nabla \tilde\varphi\|_{\vecSpL}^2 - \frac{m \gamma L_{\Psi'}^2}{\ell^3} \big( -\Delta_m^{-1} \tilde{\varphi}, \tilde{\varphi} \big) - \frac{\gamma L_{\Psi'}}{\ell} |\Omega| \mathbox{\tilde{\varphi}^\Omega}^2.
\end{align}

For the right hand side of~\eqref{eq:gradient-flow-tested}, we employ the Cauchy-Schwarz-type inequality~\eqref{eq:cauchy-schwarz-auxiliary-problem}, as, e.g., 
\begin{align*}
 \big( \tilde{R}, -\Delta_m^{-1} \tilde{\varphi} \big) &= \big( \tilde{R} - \tilde{R}^\Omega, -\Delta_m^{-1} \tilde{\varphi} \big) \leq \sqrt{\big( -\Delta_m^{-1} (\tilde{R} - \tilde{R}^\Omega), \tilde{R} - \tilde{R}^\Omega\big)} \cdot \sqrt{\big( -\Delta_m^{-1}\tilde{\varphi}, \tilde{\varphi}\big)} \\
 &= \sqrt{\big( -\Delta_m^{-1} \tilde{R}, \tilde{R} - \tilde{R}^\Omega\big)} \cdot \sqrt{\big( -\Delta_m^{-1}\tilde{\varphi}, \tilde{\varphi}\big)} = \sqrt{\big( -\Delta_m^{-1} \tilde{R}, \tilde{R} \big)} \cdot \sqrt{\big( -\Delta_m^{-1}\tilde{\varphi}, \tilde{\varphi}\big)}
\end{align*}
utilizing that $-\Delta_m^{-1}$ maps into $L_0^2(\Omega)$ and $-\Delta_m^{-1}(\tilde{R} - \tilde{R}^\Omega) = -\Delta_m^{-1}\tilde{R}$ as $-\Delta_m^{-1}$ is defined over test functions in $L_0^2(\Omega)$ only, eradicating differences through constants. Employing such Cauchy-Schwarz and Young's inequalities for some $\delta > 0$, we finally obtain for the right hand side of~\eqref{eq:gradient-flow-tested} that it holds
\begin{align}
 &\big( \tilde{R}, -\Delta_m^{-1} \tilde{\varphi} \big) + \langle \bm{\tilde f}, \bm{\tilde u} \rangle  + \big( \tilde{S}_f, -\Delta_\kappa^{-1} \tilde{\theta} \big) \nonumber \\
 &\quad\leq \delta \left( \big( -\Delta_m^{-1} \tilde{\varphi}, \tilde{\varphi} \big) + \|\bm{\tilde{u}}\|_{\vecSpH}^2 + \big( - \Delta_\kappa^{-1} \tilde{\theta}, \tilde{\theta} \big) \right) +\frac{1}{4\delta} \left(\big( -\Delta_m^{-1} \tilde{R}, \tilde{R} \big) + \|\bm{\tilde f}\|_{\bm{H^{-1}(\Omega)}}^2 + \big( -\Delta_\kappa^{-1} \tilde{S}_f, \tilde{S}_f \big)\right).
\label{eq:sub:gradient-flow-tested:rhs}
\end{align}

Inserting \eqref{eq:sub:gradient-flow-tested:dissipation}, \eqref{eq:sub:gradient-flow-tested:energy} and \eqref{eq:sub:gradient-flow-tested:rhs} into~\eqref{eq:gradient-flow-tested} and integrating over $[0,T]$ yields 
\begin{align*}
 &\frac{1}{2} \big( -\Delta_m^{-1} \tilde{\varphi}(T) , \tilde{\varphi}(T) \big) + \frac{1}{2} \big( -\Delta_\kappa^{-1} \tilde{\theta}(T), \tilde{\theta}(T) \big) + \int_0^T  |||X_1 - X_2|||^2  - \delta \|\bm{\tilde{u}}\|_{\vecSpH}^2 - \gamma\frac{\ell}{4} \|\nabla \tilde\varphi\|_{\vecSpL}^2 \;dt  \\
 &\quad \leq \left(\delta + \frac{m\gamma L_{\Psi'}^2}{\ell^3}\right) \int_0^T \left( \big( -\Delta_m^{-1} \tilde{\varphi}, \tilde{\varphi} \big) + \big( - \Delta_\kappa^{-1} \tilde{\theta}, \tilde{\theta} \big)\right) \;dt 
% \\ &\qquad 
 +\frac{1}{2} \big( -\Delta_m^{-1} \tilde{\varphi}(0) , \tilde{\varphi}(0) \big) + \frac{1}{2} \big( - \Delta_\kappa^{-1} \tilde{\theta}(0), \tilde{\theta}(0) \big) \\
 &\quad \qquad+\frac{1}{4\delta}\int_0^T \left(
    \big( -\Delta_m^{-1} \tilde{R}, \tilde{R} \big) + \|\bm{\tilde f}\|_{\bm{H^{-1}(\Omega)}}^2 + \big( -\Delta_\kappa^{-1} \tilde{S}_f, \tilde{S}_f \big)\right) \;dt 
 % \\ &\qquad
 + \frac{\gamma L_{\Psi'}}{\ell} |\Omega| \int_0^T \mathbox{\tilde{\varphi}^\Omega}^2 \;dt.
\end{align*}
Employing Grönwall's inequality, cf.\ Lemma~\ref{lemma:Grönwall}, results in 
the existence of a constant $C$ merely depending on $\delta$ and model constants satisfying
\begin{align*}
 &\underset{t\in[0,T]}{\mathrm{ess\, sup}}\,\big( -\Delta_m^{-1} \tilde{\varphi}(t) , \tilde{\varphi}(t) \big) + \underset{t\in[0,T]}{\mathrm{ess\, sup}}\, \big( -\Delta_\kappa^{-1} \tilde{\theta}(t), \tilde{\theta}(t) \big) \\
 &\qquad\qquad+ 2\int_0^T  |||X_1 - X_2|||^2 - \delta \|\bm{\tilde{u}}\|_{\vecSpH}^2 - \gamma\frac{\ell}{4} \|\nabla \tilde\varphi\|_{\vecSpL}^2 \;dt \\ 
 &\hspace{0.18cm} \leq C \bigg[\int_0^T \left(\big( -\Delta_m^{-1} \tilde{R}, \tilde{R} \big) + \|\bm{\tilde f}\|_{\bm{H^{-1}(\Omega)}}^2 + \big( -\Delta_\kappa^{-1} \tilde{S}_f, \tilde{S}_f \big)\right) \;dt 
 + \big( -\Delta_m^{-1} \tilde{\varphi}(0) , \tilde{\varphi}(0) \big) + \big( -\Delta_\kappa^{-1} \tilde{\theta}(0), \tilde{\theta}(0) \big)\\
 &\qquad \qquad + |\Omega| \int_0^T \mathbox{\tilde{\varphi}^\Omega}^2 \;dt\bigg].
\end{align*}
% \CR{with constant $C = 2 \max\left\{\frac{1}{4\delta},\frac{1}{2},\frac{\gamma L_{\Psi'}}{\ell}\right\}\exp{\left[2\left(\delta + \frac{m\gamma L_{\Psi'}^2}{\ell^3}\right)T\right]}$.}
Identifying respective norms on $\left(H^1(\Omega) \cap L_0^2(\Omega) \right)'$
% \todo{shorter notation?} 
based on the auxiliary operators $-\Delta_m^{-1}$ and $-\Delta_\kappa^{-1}$ as $\big( -\Delta_m^{-1} \cdot, \cdot \big) = \| \cdot \|_{\left(H^1(\Omega) \cap L_0^2(\Omega) \right)'}^2$, and $\big( -\Delta_\kappa^{-1} \cdot, \cdot \big) = \| \cdot \|_{\left(H^1(\Omega) \cap L_0^2(\Omega) \right)'}^2$, yields 
\begin{align*}
 &\underset{t\in[0,T]}{\mathrm{ess\, sup}}\,\|\tilde\pf(t)\|_{\left(H^1(\Omega) \cap L_0^2(\Omega) \right)'}^2 + \underset{t\in[0,T]}{\mathrm{ess\, sup}}\,\|\tilde\theta(t)\|_{\left(H^1(\Omega) \cap L_0^2(\Omega) \right)'}^2 \\
 &\qquad\qquad+ 2\int_0^T  |||X_1 - X_2|||^2 - \delta \|\bm{\tilde{u}}\|_{\vecSpH}^2 - \gamma\frac{\ell}{4} \|\nabla \tilde\varphi\|_{\vecSpL}^2 \;dt \\
 &\quad \leq C \bigg[\int_0^T \left(\|\tilde{R}\|_{\left(H^1(\Omega) \cap L_0^2(\Omega) \right)'}^2 + \|\bm{\tilde f}\|_{\bm{H^{-1}(\Omega)}}^2 + \|\tilde{S}_f\|_{\left(H^1(\Omega) \cap L_0^2(\Omega) \right)'}^2\right) \;dt + \|\tilde\pf(0)\|_{\left(H^1(\Omega) \cap L_0^2(\Omega) \right)'}^2 \\
 &\qquad \qquad + \|\tilde\theta(0)\|_{\left(H^1(\Omega) \cap L_0^2(\Omega) \right)'}^2 + |\Omega| \int_0^T \mathbox{\tilde{\varphi}^\Omega}^2 \;dt\bigg].
\end{align*}
Based on the norm definition in \eqref{eq:sub:gradient-flow:norm} using the assumptions of Vegard's law, cf.\ (B4), and homogeneous elasticity, cf.\ (B3), we get
\begin{align*}
 |||X_1 - X_2|||^2 &= \big( \delta_{(\varphi, \bm \varepsilon, \theta)} \mathcal{E}_\mathrm{quad}(X_1) - \delta_{(\varphi, \bm \varepsilon, \theta)} \mathcal{E}_\mathrm{quad}(X_2), X_1 - X_2 \big) = 2 \mathcal{E}_{\mathrm{quad}}(\tilde\pf, \bm{\varepsilon}(\bm{\tilde u}), \tilde\theta) \\
 &= \gamma \ell \big( \nabla \tilde\pf, \nabla \tilde\pf \big) + \big( \left(\bm{\varepsilon}(\bm{\tilde u}) - \bm{\hat \varepsilon}\tilde\pf\right),\mathbb{C}\left(\bm{\varepsilon}(\bm{\tilde u}) - \bm{\hat \varepsilon}\tilde\pf\right) \big) + \big( M(\tilde\theta - \alpha \nabla \cdot \bm{\tilde u}), \tilde\theta - \alpha \nabla \cdot \bm{\tilde u} \big) \\
 &\geq \gamma\ell\|\nabla\tilde\pf\|_{\vecSpL}^2 + c_{\mathbb{C}} \|\bm{\varepsilon}(\bm{\tilde u}) - \bm{\hat \varepsilon}\tilde\pf\|_{\tensSpL}^2 + M \|\tilde\theta - \alpha \nabla \cdot \bm{\tilde u}\|_{L^2(\Omega)}^2.
\end{align*}
For the first term on the right hand side, we obtain from~\eqref{eq:poincare-wirtinger}
\begin{align*}
 \gamma \ell \| \nabla \tilde\pf \|^2_{L^2(\Omega)} &\geq \frac{\gamma \ell}{2} \| \nabla \tilde\pf \|^2_{L^2(\Omega)} + \frac{\gamma \ell}{2 C_P} \| \tilde\pf - \tilde{\varphi}^\Omega \|^2_{L^2(\Omega)} = \frac{\gamma \ell}{2} \| \nabla \tilde\pf \|^2_{L^2(\Omega)} + \frac{\gamma \ell}{2 C_P} \| \tilde\pf \|^2_{L^2(\Omega)} - \frac{\gamma \ell}{2C_P}\, |\Omega| \mathbox{\tilde{\varphi}^\Omega}^2 \\
 &\geq C_1 \| \tilde\pf \|^2_{H^1(\Omega)} + \frac{\gamma \ell}{4} \| \nabla \tilde\pf \|^2_{L^2(\Omega)} + \frac{\gamma \ell}{4 C_P} \| \tilde\pf \|^2_{L^2(\Omega)} - C_4 |\Omega| \mathbox{\tilde{\varphi}^\Omega}^2.
\end{align*} 
for suitable constants $C_1,\ C_4>0$. By employing standard binomial arguments and Young's inequality, one can deduce that there exist constants $C_2,\ C_3,\ \xi >0$ such that
\begin{align*}
 \frac{\gamma \ell}{4 C_P} \| \tilde\pf \|^2_{L^2(\Omega)} + c_{\mathbb{C}}\|\bm{\varepsilon}(\bm{\tilde u}) - \bm{\hat \varepsilon}\tilde\pf\|_{\tensSpL}^2 &\geq (C_2+\xi) \|\bm{\varepsilon}(\bm{\tilde u})\|_{\tensSpL}^2 \\
 \xi \|\bm{\varepsilon}(\bm{\tilde u})\|_{\tensSpL}^2 + M \|\tilde\theta - \alpha \nabla \cdot \bm{\tilde u}\|_{L^2(\Omega)}^2 &\geq C_3 \|\tilde\theta\|_{L^2(\Omega)}^2
\end{align*}
such that with $\delta=\frac{C_2}{2C_K^2}$, overall, we obtain
\begin{align*}
 |||X_1 - X_2|||^2 - \delta \|\bm{\tilde{u}}\|_{\vecSpH}^2 - \gamma\frac{\ell}{4} \|\nabla \tilde\varphi\|_{\vecSpL}^2 \geq C_1 \|\tilde\pf\|_{H^1(\Omega)}^2 + \frac{C_2}{2} \|\bm{\varepsilon}(\bm{\tilde u})\|_{\tensSpL}^2 + C_3 \|\tilde\theta\|_{L^2(\Omega)}^2 - C_4 |\Omega| \mathbox{\tilde{\varphi}^\Omega}^2.
\end{align*}
Together with $\|\cdot\|_{\left(H^1(\Omega) \cap L_0^2(\Omega) \right)'} \leq \|\cdot\|_{L^2(\Omega)}$,
% \CR{\leq \|\cdot\|_{H^1(\Omega)}}$,
Korn's inequality, and collecting constants suitably, we obtain the bound
\begin{align}
 &\underset{t\in[0,T]}{\mathrm{ess\, sup}}\,\|\tilde\pf(t)\|_{\left(H^1(\Omega) \cap L_0^2(\Omega) \right)'}^2 + \underset{t\in[0,T]}{\mathrm{ess\, sup}}\,\|\tilde\theta(t)\|_{\left(H^1(\Omega) \cap L_0^2(\Omega) \right)'}^2 + \int_0^T \|\tilde\pf\|_{H^1(\Omega)}^2 + \|\bm{\tilde u}\|_{\vecSpH}^2 + \|\tilde\theta\|_{L^2(\Omega)}^2 \;dt \nonumber \\
 &\leq C \bigg[\int_0^T \left(\|\tilde{R}\|_{L^2(\Omega)}^2 + \|\bm{\tilde f}\|_{\bm{H^{-1}(\Omega)}}^2 + \|\tilde{S}_f\|_{L^2(\Omega)}^2\right) dt + \|\tilde\pf(0)\|_{L^2(\Omega)}^2 + \|\tilde\theta(0)\|_{\left(H^1(\Omega) \cap L_0^2(\Omega) \right)'}^2 + |\Omega| \int_0^T \mathbox{\tilde{\varphi}^\Omega}^2 dt
 \bigg].
 \label{eq:continuity-I-proof}
\end{align}

In a subsequent step we are able to show continuous dependence estimates for $\tilde\mu$ and $\bm{\tilde q}$.
From~\eqref{eq:ch2weakdiffcd}, we obtain for any $\tilde{\eta}^\mathrm{ch}\in H^1(\Omega)$
\begin{align*}
 \big( \tilde\mu,\tilde\eta^\mathrm{ch}\big) &= \gamma\ell\big(\nabla\tilde\pf,\nabla\tilde\eta^\mathrm{ch}\big) + \frac{\gamma}{\ell} \big( \Psi'(\pf_1) - \Psi'(\pf_2),\tilde\eta^\mathrm{ch}\big) - \big( \bm{\hat \varepsilon}\!:\!\mathbb{C}\big(\bm \varepsilon(\bm{\tilde u}) - \bm{\hat \varepsilon}\tilde\pf\big),\tilde\eta^\mathrm{ch}\big) \\
 &\leq \left(\gamma\ell \| \nabla\tilde\pf \|_{L^2(\Omega)} + \frac{\gamma}{\ell} L_{\Psi'}  \|\tilde{\pf} \|_{L^2(\Omega)} + C_{\bm{\hat \varepsilon}} C_{\mathbb{C}} \|\bm \varepsilon(\bm{\tilde u})\|_{\tensSpL} + C_{\mathbb{C}} C_{\bm{\hat \varepsilon}}^2 \|\tilde\pf \|_{L^2(\Omega)}\right) \cdot  \|\tilde{\eta}^\mathrm{ch} \|_{H^1(\Omega)},
\end{align*}
where we used Hölder's inequality, the Lipschitz continuity of $\Psi'(\pf)$, cf.\ (B1),
% homogeneous elasticity, cf.\ (B3), as well as Vegard's law, cf.\ (B4)}.
along with (B3) and (B4).
Thus, by employing the conventional definition of the $H^1(\Omega)'$ norm, we get for some constant $C_5 > 0$
% \todo[inline]{consistent numbering of constants (also existence?)}
\begin{align}
 \int_0^T \| \tilde{\mu} \|_{H^1(\Omega)'}^2 \;dt \leq C_5 \left(\int_0^T \|\tilde{\pf} \|_{H^1(\Omega)}^2 \;dt + \int_0^T \|\bm \varepsilon(\bm{\tilde u})\|_{\tensSpL}^2 \;dt\right).
\label{eq:continuity-mu-proof}
\end{align}
Analogously, from~\eqref{eq:darcyflowweakdiffcd}, we obtain for any $\bm{\tilde\eta^{\bm q}} \in H(\Div,\Omega)$
\begin{align*}
 \big( \kappa^{-1} \bm{\tilde q},\bm{\tilde\eta^{\bm q}} \big) = \big( M(\tilde\theta - \alpha\Div\bm{\tilde u}),\nabla \cdot \bm{\tilde\eta^{\bm q}} \big) 
 \leq \left(M \big(\|\tilde{\theta}\|_{L^2(\Omega)} + \alpha \|\Div\bm{\tilde u}\|_{L^2(\Omega)}\big)\right) \cdot \|\bm{\tilde\eta^{\bm q}}\|_{H(\Div,\Omega)}
\end{align*}
which is equivalent to
\begin{align*}
 \big( \bm{\tilde q},\bm{\tilde\eta^{\bm q}} \big) \leq \left(\kappa M \big(\|\tilde{\theta}\|_{L^2(\Omega)} + \alpha \|\Div\bm{\tilde u}\|_{L^2(\Omega)}\big)\right) \cdot \|\bm{\tilde\eta^{\bm q}}\|_{H(\Div,\Omega)}
\end{align*}
as $\kappa$ is assumed to be constant, cf.\ (B2).
Thus, by employing the conventional definition of the $H(\Div,\Omega)'$ norm, we get for some constant $C_6 > 0$
\begin{align}
 \int_0^T \| \bm{\tilde q} \|_{H(\Div,\Omega)'}^2 \;dt \leq C_6 \left(\int_0^T \|\tilde{\theta}\|_{L^2(\Omega)}^2 \;dt + \int_0^T \|\bm \varepsilon(\bm{\tilde u})\|_{\tensSpL}^2 \;dt\right).
\label{eq:continuity-q-proof}
\end{align}
Combining~\eqref{eq:continuity-I-proof}, \eqref{eq:continuity-mu-proof} and \eqref{eq:continuity-q-proof} concludes the proof.
\end{proof}

In a further post-processing step we can attain continuous dependence results in stronger norms for $\pf$ and $\mu$.
\begin{lemma}
Let $(\pf_i,\mu_i)$, $i=1,2$, be as introduced in Theorem~\ref{thm:cdweak}. Then there exists a constant $C > 0$ satisfying
\begin{align}
 &\underset{t\in[0,T]}{\mathrm{ess\, sup}}\,\|\pf_1(t) - \pf_2(t)\|_{L^2(\Omega)}^2 + \int_0^T \|\mu_1 - \mu_2\|_{L^2(\Omega)}^2 \;dt \nonumber \\
 &\qquad \leq C \bigg[ \|\pf_{0,1} - \pf_{0,2}\|_{L^2(\Omega)}^2 + \int_0^T \left[\|\bm{u}_1 - \bm{u}_2\|_{\vecSpH}^2 + \|R_1 - R_2 \|_{L^2(\Omega)}^2 \right]\;dt \bigg]
\label{eq:continuity-phi-mu}
\end{align}
which together with Theorem~\ref{thm:cdweak} provides additional continuous dependence.
\end{lemma}
\begin{proof}
Testing equation \eqref{eq:ch1weakdiffcd} with $\tilde\pf$ and \eqref{eq:ch2weakdiffcd} with $\tilde\mu$, yields
\begin{subequations}
\begin{align}
 \frac{1}{2} \partial_t \|\tilde\pf\|_{L^2(\Omega)}^2 + m \big( \nabla\tilde\mu, \nabla\tilde\pf \big) &= \big( \tilde R,\tilde\pf \big) 
 \label{eq:weakdiffcdintermediate1}\\
 \|\tilde\mu\|_{L^2(\Omega)}^2 - \gamma\ell\big(\nabla\tilde\pf,\nabla\tilde\mu\big) -  \frac{\gamma}{\ell}\big( \Psi'(\pf_1) - \Psi'(\pf_2),\tilde\mu\big) + \big( \bm{\hat \varepsilon}\!:\!\mathbb{C}\big(\bm \varepsilon(\bm{\tilde u}) -\bm{\hat \varepsilon}\tilde\pf\big),\tilde\mu\big) &= 0.
 \label{eq:weakdiffcdintermediate2}
\end{align}
\label{eq:weakdiffcdintermediate}%
\end{subequations}
We add the product of $\frac{\gamma\ell}{m}$ with \eqref{eq:weakdiffcdintermediate1} to \eqref{eq:weakdiffcdintermediate2} and obtain
\begin{align*}
 \frac{\gamma\ell}{2m} \partial_t &\|\tilde\pf\|_{L^2(\Omega)}^2 + \|\tilde\mu\|_{L^2(\Omega)}^2 = \frac{\gamma\ell}{m} \big( \tilde R,\tilde\pf \big) + \frac{\gamma}{\ell}\big( \Psi'(\pf_1) - \Psi'(\pf_2),\tilde\mu \big) - \big( \bm{\hat \varepsilon}\!:\!\mathbb{C}\big(\bm \varepsilon(\bm{\tilde u}) - \bm{\hat \varepsilon}\tilde\pf\big),\tilde\mu \big).
\end{align*}
Applying Hölder's inequality, Young's inequality, the Lipschitz continuity of $\Psi'(\pf)$, cf.\ (B1), as well as Vegard's law, cf.\ (B4), we obtain
\begin{align*}
 &\frac{\gamma\ell}{2m} \partial_t \|\tilde\pf\|_{L^2(\Omega)}^2 + \|\tilde\mu\|_{L^2(\Omega)}^2 \\
 &\hspace{0.3cm}\leq \frac{\gamma\ell}{m} \|\tilde R\|_{L^2(\Omega)} \|\tilde\pf\|_{L^2(\Omega)} + \frac{\gamma}{\ell}L_{\Psi'} \|\tilde\pf\|_{L^2(\Omega)} \| \tilde\mu\|_{L^2(\Omega)} + C_{\bm{\hat \varepsilon}} C_{\mathbb{C}} \|\bm \varepsilon(\bm{\tilde u})\|_{\tensSpL} \|\tilde\mu\|_{L^2(\Omega)} + C_{\mathbb{C}} C_{\bm{\hat \varepsilon}}^2 \|\tilde\pf\|_{L^2(\Omega)} \|\tilde\mu\|_{L^2(\Omega)} \\
 &\hspace{0.3cm}\leq \left(\frac{\gamma\ell}{2m} + \frac{\gamma^2}{\ell^2}L_{\Psi'}^2 + 2 C_{\mathbb{C}}^2 C_{\bm{\hat \varepsilon}}^4\right) \|\tilde\pf\|_{L^2(\Omega)}^2 + \frac{1}{2} \|\tilde\mu\|_{L^2(\Omega)}^2 +  2 C_{\bm{\hat \varepsilon}}^2 C_{\mathbb{C}}^2 \|\bm \varepsilon(\bm{\tilde u})\|_{\tensSpL}^2 + \frac{\gamma\ell}{2m} \|\tilde R\|_{L^2(\Omega)}^2.
\end{align*}
Reformulation and integrating over $[0,T]$ yields
\begin{align*}
 &\|\tilde\pf(T)\|_{L^2(\Omega)}^2 + \frac{m}{\gamma\ell} \int_0^T \|\tilde\mu\|_{L^2(\Omega)}^2 \;dt \\
 &\quad \leq \|\tilde\pf(0)\|_{L^2(\Omega)}^2 + \int_0^T \bigg[\frac{2m}{\gamma\ell} \left(\frac{\gamma\ell}{2m} + \frac{\gamma^2}{\ell^2}L_{\Psi'}^2 + 2 C_{\mathbb{C}}^2 C_{\bm{\hat \varepsilon}}^4\right) \|\tilde\pf\|_{L^2(\Omega)}^2 + \frac{4m}{\gamma\ell} C_{\bm{\hat \varepsilon}}^2 C_{\mathbb{C}}^2  \|\bm \varepsilon(\bm{\tilde u})\|_{\tensSpL}^2  + \| \tilde{R} \|_{L^2(\Omega)}^2\bigg]\;dt.
\end{align*}
Employing Grönwall's inequality, cf.\ Lemma~\ref{lemma:Grönwall}, results in the assertion.
\end{proof}

Following ideas from~\cite{Nochetto1985} also used in~\cite{Arbogast1996,Radu2004}, we show additional continuous dependence of the flux $\bm{q}$ (integrated in time) in a stronger norm.
\begin{lemma}
 Let $\bm{q}_i$, $i=1,2$, be as introduced in Theorem~\ref{thm:cdweak}.
It holds
\begin{align}
 &\underset{t\in[0,T]}{\mathrm{ess\, sup}}\,\left\| \kappa^{-1/2} \int_0^t (\bm{q}_1 - \bm{q}_2) \;ds \right\|_{\vecSpL}^2 \nonumber \\
 &\quad \leq
 3\left(T \int_0^T \|S_{f_1} - S_{f_2}\|_{L^2(\Omega)}^2 dt + \underset{t\in[0,T]}{\mathrm{ess\, sup}}\,\|\theta_1(t) - \theta_2(t)\|_{\left(H^1(\Omega) \cap L_0^2(\Omega) \right)'}^2 + \|\theta_{0,1} - \theta_{0,2}\|_{\left(H^1(\Omega) \cap L_0^2(\Omega) \right)'}^2\right),
\label{eq:continuity-II}
\end{align}
which together with Theorem~\ref{thm:cdweak} provides additional continuous dependence.
\end{lemma}
\begin{proof}
We consider variables integrated in time for $t\in(0,T)$
\begin{align*}
 \bm{\tilde q}_{\int}(t) := \int_0^t \bm{\tilde q} \;ds, \qquad
 \tilde{\theta}_{\int}(t) := \int_0^t \tilde{\theta} \;ds, \qquad
 \bm{\tilde u}_{\int}(t) := \int_0^t \bm{\tilde u} \;ds
\end{align*}
with the assumptions from Theorem~\ref{thm:cdweak} implying $\bm{\tilde q}_{\int} \in L^\infty([0,T]; H_0(\Div, \Omega))$, $\tilde{\theta}_{\int} \in L^{\infty}([0,T]; L^2(\Omega))$ and $\bm{\tilde u}_{\int} \in L^{\infty}([0,T]; \vecSpHTr)$ since $T$ is finite. 
% \jbrem{By the Cauchy-Schwarz inequality we have
% \begin{align*}
%  \|\tilde{\theta}_{\int}\|_{L^{\infty}([0,T]; L^2(\Omega))}^2 &= \underset{t\in[0,T]}{\mathrm{ess\, sup}}\,\|\tilde{\theta}_{\int}\|_{L^2(\Omega)}^2 = \underset{t\in[0,T]}{\mathrm{ess\, sup}}\,\int_{\Omega} \left(\int_0^t \tilde{\theta} \;dt\right)^2 dx \leq \underset{t\in[0,T]}{\mathrm{ess\, sup}}\,\int_{\Omega} \left(t \int_0^t \tilde{\theta}^2 \;dt\right) dx \\
%  &= \underset{t\in[0,T]}{\mathrm{ess\, sup}}\,\left[t\int_0^t \left(\int_{\Omega}\tilde{\theta}^2 \;dx\right) dt\right] = T \int_0^T \|\tilde{\theta}\|_{L^2(\Omega)}^2 \;dt = T \|\tilde{\theta}\|_{L^2([0,T]; L^2(\Omega))}^2
% \end{align*}}
By integrating~\eqref{eq:flowweakdiffcd}--\eqref{eq:darcyflowweakdiffcd} in time over the interval $(0,t)$ for $t\in(0,T)$, we obtain
\begin{align}
 \big( \kappa^{-1} \bm{\tilde q}_{\int}(t),\bm{\tilde\eta^{\bm q}} \big) - \big( M(\tilde\theta_{\int}(t) - \alpha\Div\bm{\tilde u}_{\int}(t)),\nabla \cdot \bm{\tilde\eta^{\bm q}} \big) &= 0 
 \label{eq:darcyflowweakdiffcd-int}.
 \\
\big(\nabla\cdot \bm{\tilde q}_{\int}(t),\tilde\eta^{\theta}\big) &= \big(t \tilde S_f + \tilde\theta(0)- \tilde\theta(t),\tilde\eta^{\theta}\big)
 \label{eq:flowweakdiffcd-int} 
\end{align}
for all $\tilde{\eta}^\theta \in L^2(\Omega)$ and $\bm{\tilde\eta^{\bm q}} \in H(\Div,\Omega)$.
Let $\mathcal{B}: H_0(\Div,\Omega) \times L_0^2(\Omega) \rightarrow H_0(\Div,\Omega)' \times L_0^2(\Omega)'$  (note that $L_0^2(\Omega)' = L_0^2(\Omega)$), defined as
\begin{align*}
\langle \mathcal{B}(\bm{q},\theta), (\bm{\eta}^{\bm{q}}, \eta^\theta) \rangle := \big( \kappa^{-1} \bm{q}, \bm{\eta}^{\bm{q}} \big) - \big( \theta, \nabla \cdot \bm{\eta}^{\bm{q}} \big)  + \big( \nabla \cdot \bm{q}, \eta^\theta \big),\qquad (\bm{q},\theta),\ (\bm{\eta}^{\bm{q}}, \eta^\theta) \in H_0(\Div,\Omega) \times L_0^2(\Omega).
\end{align*}
It is well-known that $\mathcal{B}$ defines an isomorphism~\cite{Boffi2013}, allowing to also consider $\mathcal{B}^{-1}\!: H_0(\Div,\Omega)' \times L_0^2(\Omega)' \rightarrow H_0(\Div,\Omega) \times L_0^2(\Omega)$.
From well-posedness of $\mathcal{B}$, we identify the 1-1 correspondence for all $t\in[0,T]$
\begin{align*}
 \mathcal{B}\!\left(\bm{\tilde q}_{\int}, M(\tilde{\theta}_{\int} - \alpha \nabla \cdot \bm{\tilde u}_{\int}) \right) = \left(\bm{0}, t\tilde{S}_f + \tilde{\theta}(0) - \tilde{\theta}\right) \in H_0(\Div,\Omega)' \times L_0^2(\Omega)'.
\end{align*}
Note that by the fact that $\bm{\tilde q}_{\int}(t) \in H_0(\Div,\Omega)$, equation~\eqref{eq:flowweakdiffcd-int} implies $t \tilde S_f + \tilde\theta(0)- \tilde\theta(t) \in L_0^2(\Omega)$ for a.e.\ $t\in[0,T]$.
The isomorphism property of $\mathcal{B}$ in particular implies
\begin{align}
 &\left\langle \mathcal{B}\!\left(\bm{\tilde q}_{\int}, M(\tilde{\theta}_{\int} - \alpha \nabla \cdot \bm{\tilde u}_{\int}) \right), \left(\bm{\tilde q}_{\int}, M(\tilde{\theta}_{\int} - \alpha \nabla \cdot \bm{\tilde u}_{\int}) \right) \right\rangle \nonumber \\
 &\quad= \left\langle \mathcal{B}^{-1}\!\left(\bm{0}, t\tilde{S}_f + \tilde{\theta}(0) - \tilde{\theta}\right), \left(\bm{0}, t\tilde{S}_f + \tilde{\theta}(0) - \tilde{\theta}\right) \right\rangle.
\label{eq:B-B-inv}
\end{align}
For the left hand side of~\eqref{eq:B-B-inv}, by definition of $\mathcal{B}$, we obtain
\begin{align}
 \left\langle \mathcal{B}\!\left(\bm{\tilde q}_{\int}, M(\tilde{\theta}_{\int} - \alpha \nabla \cdot \bm{\tilde u}_{\int}) \right), \left(\bm{\tilde q}_{\int}, M(\tilde{\theta}_{\int} - \alpha \nabla \cdot \bm{\tilde u}_{\int}) \right) \right\rangle = \big( \kappa^{-1} \bm{\tilde q}_{\int}, \bm{\tilde q}_{\int} \big) = \| \kappa^{-1/2} \bm{\tilde q}_{\int} \|_{\vecSpL}^2.
 \label{eq:B-B-inv-lhs}
\end{align}
For the right hand side of~\eqref{eq:B-B-inv}, we can identify $-\Delta_\kappa^{-1}(\cdot) = \Pi_\theta\mathcal{B}^{-1}(\bm{0}, \cdot)\!: L_0^2(\Omega) \rightarrow L_0^2(\Omega)$, where $\Pi_\theta$ denotes the restriction onto the $\theta$-component. Thus, it holds
\begin{align}
 &\left\langle \mathcal{B}^{-1}\!\left(\bm{0}, t\tilde{S}_f + \tilde{\theta}(0) - \tilde{\theta}\right), \left(\bm{0}, t\tilde{S}_f + \tilde{\theta}(0) - \tilde{\theta}\right) \right\rangle \nonumber \\ 
 &\quad= \left( -\Delta_\kappa^{-1}\!\left(t\tilde{S}_f + \tilde{\theta}(0) - \tilde{\theta} \right), t\tilde{S}_f + \tilde{\theta}(0) - \tilde{\theta} \right) \nonumber\\
&\quad= \|t\tilde{S}_f + \tilde{\theta}(0) - \tilde{\theta} \|_{\left(H^1(\Omega) \cap L_0^2(\Omega) \right)'}^2. \label{eq:B-B-inv-rhs}
\end{align}
By expanding the right hand side using the fundamental inequality $(a+b+c)^2 \leq 3(a^2 + b^2 + c^2)$ and combining all results~\eqref{eq:B-B-inv}--\eqref{eq:B-B-inv-rhs}, and taking the supremum over $t\in[0,T]$, we conclude with the asserted continuity result for the flux $\bm{\tilde q}$.
\end{proof}

\section{Concluding remarks}\label{sec:conclusion}
In this work, we have established a well-posedness result for the recently presented Cahn-Hilliard-Biot model~\cite{Storvik2022}. The major results include the existence of weak solutions, continuous dependence with respect to initial and right hand side data, and as a consequence uniqueness of weak solutions. Both major results utilize the underlying gradient flow structure of the problem that allows for natural \textit{a priori} bounds for discrete approximations and continuous dependence estimates. In addition, we highlight the use of mixed formulations for both the evolution of the phase-field and the fluid flow model, which explicitly encode a mass conservative character - such formulation is of major relevance for practical approximations. Throughout the analysis, several assumptions are imposed. We particularly highlight the assumption of constant Biot modulus and Biot-Willis coefficient, along with secondary consolidation type regularization, that aids in passing to the limit. Additionally, we emphasize constant material parameters in the analysis of continuous dependence as well as the non-degeneracy conditions, providing a convex setting. Relaxing these assumptions remains an interesting task of future research. In addition, a dedicated study of practical discrete approximations is of future interest. However, we expect that the techniques used in this work can be adopted for the analysis of practical mixed discretization schemes of the continuous system of PDEs. In addition, the development of dedicated numerical solution strategies, building on robust iterative decoupling and exploiting the gradient flow structure of the underlying problem, is envisioned in the future.

\section*{Acknowledgement}
Funded in part by Deutsche Forschungsgemeinschaft (DFG, German Research Foundation) under Germany's Excellence Strategy - EXC 2075 – 390740016. CR acknowledges the support by the Stuttgart Center for Simulation Science (SC SimTech). JWB acknowledges support from the UoB Akademia-project FracFlow. 
FAR acknowledges the support of the VISTA program, The Norwegian Academy of Science and Letters and Equinor. The authors also thank Jan M. Nordbotten for helpful discussions. \revision{The authors further want to thank the anonymous reviewers for their thorough proofreading and valuable comments that contributed to improving the manuscript.}

\bibliographystyle{unsrt}
{\bibliography{bibliography.bib}}

\end{document}